\documentclass{ctexart}
\usepackage{}
\usepackage{geometry}
\usepackage{amsmath}
\usepackage{amssymb}
\usepackage{amsthm}
\usepackage{graphicx}
\usepackage{authblk}
\usepackage{color}
\usepackage{setspace}

\usepackage{paralist}

\usepackage{ulem}

\allowdisplaybreaks[4]

\makeatletter
\@addtoreset{equation}{section}
\makeatother

\newtheorem{lemma}{Lemma}[section]
\newtheorem{proposition}{Proposition}[section]
\makeatletter
\renewcommand{\section}{\@startsection{section}{1}{0mm}
 {0.7\baselineskip}{0.5\baselineskip}{\Large\bf\leftline}}  %{0.5\baselineskip}{0.5\baselineskip},
\makeatother

\newtheorem{theorem}{Theorem}[section]

\theoremstyle{definition}

\newtheorem{remark}{Remark}

\usepackage[colorlinks=true]{hyperref}
\hypersetup{
  linkcolor=cyan,    %forestgreen  cyan
  citecolor=green,
  filecolor=blue,}

\begin{document}
\title{\bf Asymptotic stability of the superposition of viscous contact wave with rarefaction waves for the compressible Navier-Stokes-Maxwell equations} \vskip 0.5cm

\author{Huancheng Yao, Changjiang Zhu\thanks{Corresponding author.
\authorcr E-mail addresses: mayaohch@mail.scut.edu.cn (Yao), machjzhu@scut.edu.cn (Zhu).}}
\affil{ \normalsize  School of Mathematics, South China University of Technology, \authorcr Guangzhou 510641, P.R. China }

\date{}
\maketitle

\begin{abstract}
We study the large-time asymptotic behavior of solutions toward the combination of a viscous contact wave with two rarefaction waves for the compressible non-isentropic Navier-Stokes equations coupling with the Maxwell equations through the Lorentz force (called the Navier-Stokes-Maxwell equations).
It includes the electrodynamic effects into the dissipative structure of the hyperbolic-parabolic system and turns out to be more complicated than that in the simpler compressible Navier-Stokes equations.
Based on a new observation of the specific structure of the Maxwell equations in the Lagrangian coordinates,
we prove that this typical composite wave pattern is time-asymptotically stable for the Navier-Stokes-Maxwell equations under some smallness conditions on the initial perturbations and wave strength,
and also under the assumption that the dielectric constant is bounded.
The main result is proved by using elementary energy methods.
This is the first result about the nonlinear stability of the combination of two different wave patterns for the compressible Navier-Stokes-Maxwell equations.

\vspace{3mm}



\end{abstract}

\begin{spacing}{1.00}
\tableofcontents
\end{spacing}

\section{Introduction and main results}
\subsection{The problem}

Plasma dynamics is a field of studying flow problems of electrically conducting fluids.
The research objects of plasma dynamics are very broad.
A complete analysis in this field consists of the study of the gasdynamic field, the electromagnetic field and the radiation field simultaneously (see \cite{Pai1962}).
In the present paper, we consider the flow of an electrically conducting fluid under the influence of the electromagnetic field.
At the macroscopic level, the motion of this electrically conducting fluid such as the movement in the electromagnetic field generated by itself is described by hydrodynamics equations,
for example the compressible Navier-Stokes equations.
Because the dynamic motion of the fluid and the electromagnetic field interact tightly on each other,
the governing system in the non-isentropic case is derived from fluid mechanics with appropriate modifications to account for the electromagnetic effects,
which is composed of the laws of conservation of mass, momentum and energy, Maxwell's law, and the law of conservation of electric charge (see \cite{Imai}, \cite{Kawashima1}, \cite{xux}):
\begin{small}
\begin{equation}\label{chushifangcheng}
\left\{
\begin{aligned}
& \partial_{t} \rho+\operatorname{div}(\rho \boldsymbol{u})=0,   \\[2mm]
& \rho \left( \partial_{t}\boldsymbol{u} + \boldsymbol{u} \cdot \nabla \boldsymbol{u} \right)   +   \nabla p  =  \mu' \Delta\boldsymbol{u} + (\lambda+\mu') \nabla(\operatorname{div} \boldsymbol{u}) + \rho_e \boldsymbol{E} + \boldsymbol{J} \times \boldsymbol{B},  \\[2mm]
& \rho \frac{\partial e}{\partial \theta} (\partial_t \theta + \boldsymbol{u} \cdot \nabla \theta) + \theta \frac{\partial p}{\partial \theta} \operatorname{div}\boldsymbol{u} = \operatorname{div}(\kappa \nabla \theta)  +  \mathbb{N}(\boldsymbol{u})  +  (\boldsymbol{J}  -  \rho_e \boldsymbol{u}) \cdot (\boldsymbol{E} + \boldsymbol{u}\times\boldsymbol{B}),  \\[2mm]
& \varepsilon \partial_{t} \boldsymbol{E}  -  \frac{1}{\mu_0}\operatorname{curl}\boldsymbol{B}+\boldsymbol{J}=0, \\[2mm]
& \partial_{t} \boldsymbol{B}  +  \operatorname{curl} \boldsymbol{E}=0,   \\[2mm]
&\partial_t \rho_e + \operatorname{div}\boldsymbol{J} = 0, \quad \varepsilon \operatorname{div}\boldsymbol{E} = \rho_e, \quad \operatorname{div}\boldsymbol{B} = 0,
\end{aligned}\right.
\end{equation}
\end{small}
where $(\boldsymbol{x}, t)\in\mathbb{R}^{3}\times\mathbb{R}_{+}$. Here, $\rho(\boldsymbol{x},t)>0$ denotes the mass density;
$\boldsymbol{u}=\left(u_{1}, u_{2}, u_{3}\right)\in\mathbb{R}^{3}$ is the fluid velocity;
$ \theta(\boldsymbol{x},t) > 0 $ is the absolute temperature; $\boldsymbol{E}=\left(E_{1}, E_{2}, E_{3}\right)\in\mathbb{R}^{3}$ and $\boldsymbol{B}=\left(B_{1}, B_{2}, B_{3}\right)\in\mathbb{R}^{3}$ denote the electric field and the magnetic field, respectively.
$ \rho_e (\boldsymbol{x},t) $ is the electric charge density.
The pressure $ p $ and the internal energy $ e $ are expressed by the equations of states.
For the sake of simplicity, we will focus on only polytropic fluids throughout this paper, namely
\begin{equation}\label{p(theta)}
p = R \rho \theta,\qquad e = \frac{R}{\gamma-1} \theta.
\end{equation}
$ \mathbb{N}(\boldsymbol{u}) $ in $ \eqref{chushifangcheng}_3 $ denotes the viscous dissipation function
\begin{equation*}
\mathbb{N}(\boldsymbol{u}) = \sum_{i,j=1}^3 \frac{\mu'}{2} \left( \frac{\partial u_i}{\partial x_j} + \frac{\partial u_j}{\partial x_i} \right)^2 + \lambda \left( \operatorname{div}\boldsymbol{u} \right)^2 .
\end{equation*}
What's more, the electric current density $\boldsymbol{J}$ can be
expressed by Ohm's law
$$\boldsymbol{J} = \rho_e \boldsymbol{u} + \sigma(\boldsymbol E + \boldsymbol u \times \boldsymbol B).$$

Next, let's discuss the parameters in the above equations.
First,
$ R >0 $ is the gas constant and $ \gamma>1 $ is the adiabatic exponent.
$\mu'$ and $\lambda$ in $\eqref{chushifangcheng}_2$ and $ \mathbb{N}(\boldsymbol{u}) $ are the viscosity coefficients of the fluid which satisfy $\mu'>0$ and $2\mu'+3\lambda >0$.
In Ohm's law, $\sigma>0$ denotes the electic conductivity coefficient.
Moreover,
the heat conductivity coefficient $\kappa$ in $\eqref{chushifangcheng}_3$ and the magnetic permeability $\mu_0$ in $\eqref{chushifangcheng}_4$ are assumed to be positive constants.
Finally, $\varepsilon>0$ in $\eqref{chushifangcheng}_4$ is the dielectric constant.

There is quite limited mathematical progress for the original nonlinear system since as pointed out by Kawashima in \cite{Kawashima1}, the system $\eqref{chushifangcheng}$ is neither symmetric hyperbolic nor strictly hyperbolic.
For this reason the classical local well-posedness theorem (cf. \cite{Kato}) cannot be directly applied to the system $\eqref{chushifangcheng}$.
Moreover, the hydrodynamic and electrodynamic effects are strongly coupled.
The presence of strong nonlinearities and the lack of dissipation mechanism of the magnetic field also lead to many difficulties.
Due to the mathematically complicate structure of the original nonlinear system $\eqref{chushifangcheng}$,
some simplified models are derived according to the actual physical application.
As it was pointed out by Imai in \cite{Imai}, the assumption that the electric charge density $\rho_{e} \simeq 0$ is physically very good for the study of plasmas.
Here, it is essential to recognize that the quasi-neutrality assumption $\rho_{e} \simeq 0$ is different from the assumption of exact neutrality $\rho_{e} = 0$ since the latter would lead to the superfluous condition $\operatorname{div}\boldsymbol{E}=0 $.
According to this quasi-neutrality assumption, we can eliminate the terms involving $\rho_e$ in the system $ \eqref{chushifangcheng} $ and derive the following simplified system (cf. \cite{xux}):
\begin{equation}\label{NSM}
\left\{
\begin{aligned}
& \partial_{t} \rho+\operatorname{div}(\rho \boldsymbol{u})=0,   \\[2mm]
& \rho \left( \partial_{t}\boldsymbol{u} + \boldsymbol{u} \cdot \nabla \boldsymbol{u} \right)   +   \nabla p  =  \mu' \Delta\boldsymbol{u} + (\lambda+\mu') \nabla(\operatorname{div} \boldsymbol{u})  +  \boldsymbol{J} \times \boldsymbol{B},  \\[2mm]
& \frac{R}{\gamma-1} \rho (\partial_t \theta + \boldsymbol{u} \cdot \nabla \theta) + p \operatorname{div}\boldsymbol{u} = \operatorname{div}(\kappa \nabla \theta)  +  \mathbb{N}(\boldsymbol{u})  +  \boldsymbol{J} \cdot (\boldsymbol{E} + \boldsymbol{u}\times\boldsymbol{B}),  \\[2mm]
& \varepsilon \partial_{t} \boldsymbol{E}  -  \frac{1}{\mu_0}\operatorname{curl}\boldsymbol{B}+\boldsymbol{J}=0, \\[2mm]
& \partial_{t} \boldsymbol{B}  +  \operatorname{curl} \boldsymbol{E}=0, \quad \operatorname{div}\boldsymbol{B} = 0,
\end{aligned}\right.
\end{equation}
with the electric current density $\boldsymbol{J}  =  \sigma(\boldsymbol E + \boldsymbol u \times \boldsymbol B)$.
The system $ \eqref{NSM} $ is obtained from the Navier-Stokes equations coupling with the Maxwell equations through the Lorentz force,
so it is usually called the Navier-Stokes-Maxwell equations for convenience of presentation.
Notice that the same terminology was used in Duan \cite{duan2012} and Masmoudi \cite{Masmoudi1} but for the different modeling system.

In this paper, we shall restrict ourselves to the one-dimensional motion (see \cite{fanhu}, \cite{yaozhu2021}) on the whole spatial domain $ \mathbb{R} $:
\begin{equation}\label{yuanfangcheng}
\left\{\begin{aligned}
& \rho_{t}+(\rho u)_{x}=0, \\[2mm]
& \rho(u_t + uu_x) + p_x = (\lambda + 2 \mu')u_{xx} - (E + ub)b,  \\[2mm]
& \frac{R}{\gamma-1} \rho (\theta_t + u \theta_x) + pu_x =  \kappa \theta_{xx}  +  (\lambda + 2 \mu')u_{x}^2  +  (E + ub)^2,    \\[2mm]
& \varepsilon E_{t}-b_{x}+E+u b=0, \\[2mm]
& b_{t}-E_{x}=0,
\end{aligned}\right.
\end{equation}
with the following initial data and the far field behaviors:
\begin{equation}\label{chuzhi}
  (\rho,u,\theta,E,b)(x,0)=(\rho_0,u_0,\theta_0,E_0,b_0)(x),\quad x \in \mathbb{R},
\end{equation}
\begin{equation}\label{wuqiong}
  \lim_{x \to \pm \infty}(\rho_0, u_0, \theta_0, E_0, b_0)(x)=(\rho_\pm, u_\pm, \theta_\pm, E_\pm, b_\pm),
\end{equation}
where $\rho_\pm$, $u_\pm$, $\theta_\pm$, $E_\pm$, $b_\pm$ are constants.
Indeed, for spatial variable $ \boldsymbol{x} = ( x_1 , x_2 , x_3 )$, we assume that the electromagnetic flow is moving only in the longitudinal direction $ x_1 $ (below $ x_1 $ will be denoted by $ x $) and uniform in the transverse directions $( x_2 , x_3 )$.
We also take $\mu_0 = \sigma = 1$ for simplicity.
Then it is easy to derive \eqref{yuanfangcheng} from \eqref{NSM}, based on the specific choice of dependent variables:
$$\boldsymbol{u}=(u(x,t), 0, 0),  \quad \boldsymbol{B}=(0, b(x,t), 0),  \quad \boldsymbol{E}=(0, 0, E(x,t)).$$

We are interested in the large time behavior of solutions to $ \eqref{yuanfangcheng} $, $ \eqref{chuzhi} $, $ \eqref{wuqiong} $, and it is more convenient to use the Lagrangian coordinates to explore this problem. Hence, we introduce the coordinate transformation as follows:
\begin{equation*}
 x \Rightarrow \int _{(0,0)}^{(x,t)} \rho(z,s)\,\mathrm{d}z - (\rho u)(z,s)\,\mathrm{d}s, \qquad t \Rightarrow \tau.
\end{equation*}
We still denote the Lagrangian coordinates by $ (x,t) $, and $ (\lambda+2\mu') $ by $ \mu $ for the simplicity of notation. Then the Cauchy problem $ \eqref{yuanfangcheng} $, $ \eqref{chuzhi} $, $ \eqref{wuqiong} $ can be transformed as the following form
\begin{equation}\label{lagrange}
\left\{\begin{aligned}
&v_{t}-u_{x}=0, \\[2mm]
&u_{t}+p_{x}=\mu\left(\frac{u_{x}}{v}\right)_{x} - v(E+ub)b, \\[2mm]
&\frac{R}{\gamma-1}\theta_{t} + p u_{x} =\kappa\left(\frac{\theta_{x}}{v}\right)_{x} + \mu \frac{u^2_{x}}{v} + v(E+ub)^2, \\[2mm]
&\varepsilon \left(E_{t}-\frac{u}{v}E_x\right)-\frac{1}{v}b_{x}+ E+u b=0, \\[2mm]
&b_{t}-\frac{u}{v}b_x-\frac{1}{v}E_{x}=0.
\end{aligned}\right.
\end{equation}
Here $v = 1 / \rho$ denotes the specific volume and the pressure $p$ is expressed by
$$
p=\frac{R\theta}{v} = A v^{-\gamma} \exp{\frac{\gamma-1}{R} s},
$$
where $s$ is the entropy.
The initial data and the far field constant states are given by
\begin{equation}\label{chuzhi'}
(v,u,\theta,E,b)(x,0)=(v_0,u_0,\theta_0,E_0,b_0)(x),\quad x \in \mathbb{R},
\end{equation}
\begin{equation}\label{wuqiong'}
\lim_{x \to \pm \infty}(v_0, u_0, \theta_0, E_0, b_0)(x)=(v_\pm, u_\pm, \theta_\pm, E_\pm, b_\pm).
\end{equation}

%%%%%%%%%%%%%%%%%%%%%%%%%%%%%%%%%%%%%%%%%%%%%%%%

Let us recall some known results about the Navier-Stokes-Maxwell equations.
There have been some research on the existence and large-time behavior of solutions to the compressible Navier-Stokes-Maxwell equations.
In \cite{Kawashima2} and \cite{Kawashima3}, Kawashima and Shizuta established the global existence of smooth solutions for small data and studied its zero dielectric constant limit in the whole space $ \mathbb{R}^2 $.
Later, Jiang and Li in \cite{jiangsong1} studied the zero dielectric constant limit and obtained the convergence of the system $ \eqref{NSM} $ to the full compressible magnetohydrodynamic equations in the torus $ \mathbb{T}^3 $.
Recently, Xu in \cite{xux} studied the large-time behavior of the classical solution toward some given constant states and obtained the time-decay estimates in the whole space $ \mathbb{R}^3 $ with small initial perturbation.
For the one-dimensional non-isentropic Navier-Stokes-Maxwell equations,
Fan and Hu in \cite{fanhu} proved the uniform estimates with respect to the dielectric constant and the global-in-time existence in a bounded interval without vacuum.
In addition, Fan and Ou in \cite{fanou} studied the one-dimensional full equations for a thermo-radiative electromagnetic fluid in a form similar to that in \eqref{yuanfangcheng};
and established the similar result to \cite{fanhu}.
However, for the compressible Navier-Stokes-Maxwell equations,
there are few results about the large-time behavior of the solution toward some non-constant states, especially the hyperbolic elementary waves.
For example, to the authors' best knowledge, there are only three relevant results.
Luo-Yao-Zhu in \cite{luoyaozhu2021} and Yao-Zhu in \cite{yaozhu2021} established the stability of rarefaction wave for the isentropic Navier-Stokes-Maxwell equations and non-isentropic ones under small $ H^1 $-initial perturbations, respectively.
What' more, Huang-Liu in \cite{huangyt} consider the stability of rarefaction wave for a macroscopic model derived from the Vlasov-Maxwell-Boltzmann system,
in which the model they consider is obviously different from ours in this paper,
except for the similar dissipative term $E + ub$.

The main purpose of this paper is to study the asymptotic stability of the combination of viscous contact wave with rarefaction waves to the Cauchy problem \eqref{lagrange}, \eqref{chuzhi'} and \eqref{wuqiong'} for the large time behavior.
We also notice that the one-dimensional Navier-Stokes-Maxwell equations $\eqref{lagrange}$ is a combination of the compressible non-isentropic Navier-Stokes equations and the Maxwell equations through the Lorentz force.
When omitting the electrodynamic effects, the system $\eqref{lagrange}$ reduces to the classical Navier-Stokes equations.
In fact, there are extensive studies concerning those results for the Navier-Stokes equations in the context of gas dynamics.
It is well known that the large-time behavior of solutions to the Cauchy problem for Navier-Stokes equations can converge to the Riemann solutions for the corresponding Euler equations.
As for the Riemann problem, under a proper assumption on the flux function, there are three kinds of elementary wave solutions: shock wave, rarefaction wave and contact discontinuity, and then the Riemann solution generally forms a multi-wave pattern given by a various linear combination of these three elementary waves (cf. \cite{Matsumura3}, \cite{Smoller} etc.).
Here, we mention several works on the asymptotic stability analysis of viscous wave pattern for the Navier-Stokes equations:
\cite{liu1986}, \cite{Kawashima1985}, \cite{MatsuNishi85} for the shock wave;
\cite{Matsumura1986}, \cite{Kawashima1986}, \cite{liu1988}, \cite{Matsumura1992}, \cite{nishi-yang-zhao2004} for the rarefaction wave;
\cite{huang-matsu-shi}, \cite{huangfm2006}, \cite{huang-xin-yang2008}, \cite{huang-zhao2003} for the contact discontinuity.
In particular, the asymptotics toward two rarefaction waves was established in Kawashima-Matsumura-Nishihara \cite{Kawashima1986} by using an energy form associated with the physical total energy and a monotone property of the smooth approximate rarefaction waves.
Moreover, motivated by the theory of viscous shock waves \cite{Kawashima1985},
Huang-Xin-Yang \cite{huang-xin-yang2008} succeeded in using the anti-derivative method to obtain not only the stability but also the convergence rate of the viscous contact wave for the Navier-Stokes equations and the Boltzmann equation.
However, their method is not available for the stability of the composite wave of viscous contact wave with rarefaction waves,
since the anti-derivative method is invalid for the stability of the rarefaction wave.
%Huang-Matsumura-Xin \cite{huangfm2006}
We should mention that the stability of the superposition of several wave patterns is more complicated and challenging
due to the fact that the stability analysis essentially depends on the underlying properties of basic wave pattern and the wave interaction between different families of wave patterns is complicated.
Recently, Huang-Matsumura in \cite{Huangmatsu2009} showed that the composite wave of two viscous shock waves for the Navier-Stokes equations is asymptotically stable without the zero initial mass condition.
Later Huang-Li-Matsumura in \cite{huangfm2010} proved the stability of the combination of viscous contact wave with rarefaction waves by deriving a new estimate on the heat kernel.
This new technique was needed by Zeng in \cite{zenghh2009} to obtain the stability of the superposition of shock waves with contact discontinuities for systems of viscous conservation laws,
and was also used by Huang-Wang in \cite{huangfmWang2016} for the global stability of the same wave patterns and system as in \cite{huangfm2010}.
In addition, some results about the asymptotic stability of the composite wave as in \cite{huangfm2010} were also shown for some more complex models, and we refer interested readers to \cite{ruan2017}, \cite{Hakhong2017}, \cite{li-wang-wang2018}, \cite{chenzz2019}, \cite{luo2020}, and references therein.
Due to the frameworks of viscous shock wave and rarefaction wave are not compatible with each other,
the time-asymptotic stability of the combination of viscous shock wave and rarefaction wave is still an interesting and challenging question for the Navier-Stokes equations!
Recently, Moon-Jin Kang, Alexis F. Vasseur and Yi Wang have submitted the paper \cite{kang-wang2021} titled ``Time-asymptotic stability of composite waves of viscous shock and rarefaction for barotropic Navier-Stokes equation'' to the online database: arXiv.org.
Maybe they have made some progress on the time asymptotic stability of the combination of viscous shock wave and rarefaction wave for the Navier-Stokes equations.

%%%%%%%%%%%%%%%%%%%%%%%%%%%%%%%%%%%%%%%%%%%%%%%%%
Motivated by the relationship between the compressible non-isentropic Navier-Stokes equations and Navier-Stokes-Maxwell equations, we provisionally assume the initial data of the electric field and the
magnetic field at both far fields $x=\pm \infty$ are zero, respectively, that is to say,
\begin{equation}\label{wuqiong1}
\lim_{x \to \pm \infty}(v_0, u_0, \theta_0, E_0, b_0)(x)=(v_\pm, u_\pm, \theta_\pm, 0, 0).
\end{equation}
We focus on the global solution in time of the Cauchy problem $\eqref{lagrange}$, $\eqref{chuzhi'}$, $\eqref{wuqiong1}$ and their large-time behavior toward viscous wave patterns in the relations with the spatial asymptotic states $(v_\pm, u_\pm, \theta_\pm, 0, 0)$. And the more challenging case for $E_-\neq E_+$ and $b_-\neq b_+$ is left for study in future.

Here, we briefly give some remarks on our problem and review some key analytical techniques.
To state our ideas clearly, we first consider the stability of only a viscous contact discontinuity wave.

Comparing with the stability of wave patterns for the Navier-Stokes equations, the main difficulties in the present paper to prove the nonlinear stability of wave patterns lie in the additional terms produced by the electrodynamic effects.
The extra terms about electromagnetic field $ E $ and $ b $ in the fluid equations part of \eqref{lagrange} are $ - v(E+ub)b  $ and $  v(E+ub)^2 $.
They both contain no derivatives of the solution,
which is disadvantageous to the derivation of the zero-order energy estimates.
The first bad term we suffered is $-\int_0^t\int_\mathbb{R}v(E+\psi b+\bar u b)\psi b \,{\rm{d}}x{\rm{d}}\tau$.
However, the lack of damping decay mechanism of the magnetic field $ b $ hinders us from obtaining the time-space integrable good term $ \int_0^t\int_{\mathbb{R}} b^2 \,\mathrm{d}x \mathrm{d}\tau $.
It brings a lot of difficulties when deriving the zero-order estimates.
To overcome this obstacle, we try to use the structure of the Maxwell equations and package extra terms together with the bad term $-\int_0^t\int_\mathbb{R}v(E+\psi b+\bar u b)\psi b \,{\rm{d}}x{\rm{d}}\tau$ to produce a compound time-space integrable good term $\int_0^t\int_\mathbb{R} v(E+\psi b+\bar u b)^2 \,\mathrm{d}x{\rm{d}}\tau$,
which is crucial to obtain the zero-order energy estimates (see \eqref{daiquanE}, \eqref{jibenneng} in Lemma \ref{dijieguji} etc.) and is essential to get high-order energy estimates.

Secondly, in order to absorb some nonlinear bad terms by some good terms concerning the electric field $ E $ or the magnetic field $ b $,
we require a technical condition \eqref{jiedianchangshuxiao} that $ \varepsilon $ is bounded for some specific positive constants $\bar C$.
Through some elaborate analysis, we can finally determine the value of the constant $\bar{C}$.
For example, see some bad terms in \eqref{danjifen1} absorbed by the good term $ \int_{\mathbb{R}} \left( \frac{1}{2}\varepsilon vE^2  +  \frac{1}{2} vb^2  \right) \,\mathrm{d}x $ in Lemma \ref{dijieguji};
and some ones in \eqref{lemshi3} absorbed by $ \| E_x \|^2 $ in Lemma \ref{exbxyijiedao}.
For the composite wave case,
see some bad terms in \eqref{tuidao2'}, \eqref{jibennengliang''} and \eqref{lemshi3''} absorbed by the corresponding good terms:
$  \frac{1}{2}\int_{\mathbb{R}} [(U^r_-)_x+(U^r_+)_x] (\varepsilon E^2 + b^2)\,{\rm{d}}x  $, $ \int_{\mathbb{R}} \frac{1}{2}(\varepsilon vE^2+vb^2) \,{\rm{d}}x  $ and $ \| E_x \|^2 $, respectively.

% We have to treat the terms with the attached weight function $ \bar{u}_x $ such as $ \int_0^t\int_{\mathbb{R}} \bar{u}_x \left( \varepsilon E^2 + b^2 \right) \,\mathrm{d}x\mathrm{d}\tau $ (see in \eqref{tuidao2}), etc.

Thirdly, unlike the smooth approximate rarefaction wave, the sign of $ \bar{u}_x $ is not determined for viscous contact wave.
We have to estimate the terms including the weight $ \bar{u}_x $ such as $ \int_0^t\int_{\mathbb{R}} \bar{u}_x \left( \varepsilon E^2 + b^2 \right) \,\mathrm{d}x\mathrm{d}\tau $ (see in \eqref{tuidao2}), etc.
To overcome this difficulty, we observe that the fifth equation of $\eqref{raodong} $ can be transformed as $ \left( vb \right)_t - \left( E + \psi b + \bar{u} b  \right)_x = 0   $,
which contains the derivative of the compound dissipative term $E + \psi b + \bar{u} b$.
This transformation of $\eqref{raodong}_5 $ is quite different from the structure of $ \eqref{yuanfangcheng}_5 $ in the Eulerian coordinates.
Thanks to this key observation, by using the heat kernel type of inequality in Lemma \ref{heatkernel} borrowed from \cite{huangfm2010} and then employing the previous compound time-space integrable good term $\int_0^t\int_\mathbb{R} v(E+\psi b+\bar u b)^2\mathrm{d}x{\rm{d}}\tau$,
we can obtain the desired estimates; see the proof of \eqref{daiquanguji1} in Lemma \ref{quanguji}.

Fourthly, the Maxwell equations in the Lagrangian coordinates appear stronger nonlinearities than that in the Eulerian coordinates.
This will bring us some trouble terms like $ \int_{\mathbb{R}} (\left| \phi_x \right| + \left| \psi_x \right|)(E_x^2 + b_x^2) \,\mathrm{d}x $ in $ \eqref{Exbx4} $ when deriving the space integration term $\|(\sqrt{\varepsilon}E_x,b_x)\|^2$ in Lemma \ref{exbxyijiedao}.
To deal with these nonlinear terms we need the smallness of $ \| \phi_x \|_{L^\infty} $,
which just requires the {\it a priori} assumption that $ \| \phi \|_{H^2} $ is small.
Comparing with the {\it a priori} assumption that $ \| \phi \|_{H^1} $ is small for the Navier-Stokes equations in \cite{huangfm2010},
we still need to control the space integration term $ \| \phi_{xx} \|^2 $.
This fact also urges us to improve the regularity of $ \psi $ and $ \zeta $ to the Sobolev space $ L^{\infty}\left(I ; H^{2}\right) $,
where $ I $ is an interval with respect to the time variable $ t $.
For this aim, we propose the {\it a priori} assumption that $ \| \psi \|_{H^2} $ and $ \| \zeta \|_{H^2} $ are both small.

Finally, for the stability of the superposition of viscous contact wave with rarefaction waves to the compressible Navier-Stokes-Maxwell equations,
we need to consider the wave interactions of these two types of basic waves,
under the combined effects of the electric field and the magnetic field.
Borrowing some similar computations from the stability of these two typical wave patterns to the compressible Navier-Stokes equations in \cite{huangfm2010},
and combining with the arguments in \cite{yaozhu2021} of the stability of a single rarefaction wave for the Navier-Stokes-Maxwell equations in the Eulerian coordinates,
we can modify our method of a single viscous contact wave slightly to overcome the difficulties caused by the two rarefaction waves.

%\vbox{}
\vspace{2.2mm}
\noindent \textbf{Notations:} Throughout this paper, $C$ denotes some universal positive constant which is independent of $x$ and $t$ and
may vary from line to line. $\|\cdot\|_{L^p}$ stands the $L^p$-norm on the Lebesgue space ${L^p}(\mathbb{R})\;(1\leq p\leq \infty)$. For the
sake of convenience, we always denote $\|\cdot\|=\|\cdot\|_{L^2}$. What's more, $H^k$ will be
used to denote the usual Sobolev space $W^{k,2}(\mathbb{R})\;(k\in \mathbb{Z}_+)$ with respect to variable $x$.

\subsection{Smooth approximate profiles and main results}
By employing asymptotic analysis arguments with the setting of $ E_\pm = b_\pm =0 $ for the electromagnetic field,
the large-time behavior of solutions for the Cauchy problem $\eqref{lagrange}$, $\eqref{chuzhi'}$, $\eqref{wuqiong1}$ is expected to be determined by the following Riemann problem:
\begin{equation}\label{nstuidao2}
\left\{
\begin{aligned}
&v_{t}-u_{x}=0, \\[2mm]
&u_{t}+p_{x}=0, \\[2mm]
& \left( \frac{R}{\gamma-1}\theta + \frac{u^2}{2} \right)_t + (p u)_{x} = 0,
\end{aligned}
\right.
\end{equation}
with initial data
\begin{equation}\label{Riemanchuzhi}
(v,u,\theta)(x,0)=\left\{
\begin{aligned}
&(v_-,u_-,\theta_-),\quad x<0,\\[2mm]
&(v_+,u_+,\theta_+),\quad x>0.
\end{aligned}
\right.
\end{equation}
The system of conservation laws $ \eqref{nstuidao2} $ with $ \eqref{Riemanchuzhi} $ has three distinct real eigenvalues (see Ref. \cite{Smoller})
\begin{equation*}
 	\lambda_1 = -\sqrt{\frac{\gamma p}{v}} < 0,\qquad \lambda_2 = 0, \qquad \lambda_3 = \sqrt{\frac{\gamma p}{v}} > 0,
 \end{equation*}
which implies the first and third characteristic fields are genuinely nonlinear and the second field is linearly degenerate.
The basic theory of hyperbolic systems of conservation laws (for example, see Ref. \cite{Smoller}) implies that for any given constants $ (v_-, u_-, \theta_-) $ with $ v_->0 $, $ \theta_->0 $, there exists a suitable neighborhood $ \Omega(v_-, u_-, \theta_-) $ of $ (v_-, u_-, \theta_-) $ such that for any $ (v_+, u_+, \theta_+) \in \Omega(v_-, u_-, \theta_-) $,
the Riemann problem $ \eqref{nstuidao2} $ and $ \eqref{Riemanchuzhi} $ admits a solution consisting of the basic wave patterns.

%%%%%%%%%%%%%%%%%%%%%%%%%%%%%%%%%%%%%%%
%%%%%%%%%%%%%%%%%%%%%%%%%%%%%%%%%%%%%%%

In this paper, we will conclude the asymptotic stability of the linear combination of a viscous contact wave and two rarefaction waves for the Cauchy problem of the compressible non-isentropic
Navier-Stokes-Maxwell equations $\eqref{lagrange}$, $\eqref{chuzhi'}$, $\eqref{wuqiong1}$, provided the perturbations of initial data are suitably small and the dielectric constant is bounded.
We first consider the case of only a viscous contact discontinuity wave.
When omitting the effect of the electromagnetic field, the system \eqref{lagrange} reduces to the classical Navier-Stokes equations.
Recalling for the Riemann problem \eqref{nstuidao2}--\eqref{Riemanchuzhi}, the contact discontinuity wave $ (v^c, u^c, \theta^c)(x,t) $ takes the form
\begin{equation*}
   (v^c, u^c, \theta^c)(x,t) = \left\{
   \begin{aligned}
    \left(v_{-}, u_{-}, \theta_{-}\right), \quad & x<0,\; t>0, \\[2mm]
    \left(v_{+}, u_{+}, \theta_{+}\right), \quad & x>0,\; t>0,
   \end{aligned}\right.
\end{equation*}
provided that
\begin{equation}\label{czuoyouzhuangtai}
  u_{-} = u_{+}, \qquad p_{-} := \frac{R \theta_{-}}{v_{-}} = \frac{R \theta_{+}}{v_{+}} =: p_{+} .
\end{equation}
In the setting of the compressible Navier-Stokes equations, the corresponding wave $ (\bar{v}, \bar{u}, \bar{\theta}) $ to the contact discontinuity $ (v^c, u^c, \theta^c) $ behaves as a diffusion wave due to the dissipation effect, which called a viscous contact wave (see \cite{huang-xin-yang2008}, \cite{huangfm2010}).
Now we construct the viscous contact wave $ (\bar{v}, \bar{u}, \bar{\theta}) $ as follows.
Since the pressure for the profile $ (\bar{v}, \bar{u}, \bar{\theta}) $ is hoped to be almost constant (see \cite{huangfm2006}), we take
\begin{equation*}
	\bar{p}: = \frac{R \bar{\theta}}{\bar{v}} = p_+,
\end{equation*}
which indicates the leading part of the temperature equation is
\begin{equation}\label{gouzao1}
  \frac{R}{\gamma -1} \bar{\theta}_t  + p_+ \bar{u}_x = \kappa \left( \frac{\bar{\theta}_x}{\bar{v}} \right)_x.
\end{equation}
The equations \eqref{gouzao1} and $ \bar{v}_t - \bar{u}_x = 0 $ can deduce a nonlinear diffusion equation:
\begin{equation}\label{kuosanfangcheng}
  \bar{\theta}_t  = a \left( \frac{\bar{\theta}_x}{\bar{\theta}} \right)_x, \quad \bar{\theta}(\pm\infty,t) = \theta_\pm, \quad a = \frac{\kappa p_+ (\gamma-1)}{\gamma R^2}>0.
\end{equation}
Then we have the following lemma (cf. \cite{hsiaoliu}, \cite{vanduyn}):
\begin{lemma}\label{jianduanboshuaijian}
Set the strength of wave $ \delta : = \left| \theta_+ - \theta_- \right| $.
Then the problem \eqref{kuosanfangcheng} has a unique self-similarity solution $ \bar{\theta}(\xi) $, $ \xi=\frac{x}{\sqrt{1+t}} $ satisfying
\begin{itemize}
\item[$\mathrm{(i)}$] $ \bar{\theta}(\xi) $ is a monotone function, increasing if $ \theta_- < \theta_+  $ and decreasing if $ \theta_- > \theta_+  ;$

\item[$\mathrm{(ii)}$] There exists a positive constant $ \bar{\delta} $, such that for $ \delta  \leq \bar{\delta}  $, $ \bar{\theta}( \frac{x}{\sqrt{1+t}} )  $ satisfies
\begin{equation}\label{shuaijianlv1}
  (1+t)^{\frac{k}{2}}\left| \partial_x^k \bar{\theta} \right|+\left|\bar{\theta} -\theta_{\pm}\right| \leq c_{1} \delta \mathrm{e}^{-\frac{\hat{c} x^{2}}{1+t}}, \quad \text { as }|x| \rightarrow \infty, \;\; k \geq 1,
\end{equation}
where $ c_1 $ and $ \hat{c} $ are both positive constants depending only on $ \theta_- $ and $ \bar{\delta} $.
\end{itemize}

\end{lemma}
Once $ \bar{\theta} = \bar{\theta} ( \frac{x}{\sqrt{1+t}} ) $ is determined, the viscous contact wave profile $ (\bar{v}, \bar{u}, \bar{\theta})(x,t) $ can be defined as follows:
\begin{equation}\label{smoothjianduanbo}
  \bar{v}(x,t) = \frac{R}{p_+} \bar{\theta}, \qquad \bar{u}(x,t) = u_- + \frac{\kappa (\gamma-1)}{\gamma R} \frac{\bar{\theta}_x}{\bar{\theta}}, \qquad \bar{\theta}(x,t) = \bar{\theta}.
\end{equation}
Importantly,
the contact wave $ (\bar{v}, \bar{u}, \bar{\theta})(x,t) $ solves the compressible Navier-Stokes equations time asymptotically, i.e.,
\begin{equation*}
  \left\{
 \begin{aligned}
 & \bar{v}_t   -  \bar{u}_x  = 0,  \\[2mm]
 & \bar{u}_t  +  \left( \frac{R \bar{\theta}}{\bar{v}} \right)_x   =  \mu \left( \frac{\bar{u}_x}{\bar{v}} \right)_x    +  R_1 ,     \\[2mm]
 & \frac{R}{\gamma-1}\bar{\theta}_{t} + {p}_+ \bar{u}_{x} = \kappa\left(\frac{\bar{\theta}_{x}}{\bar{v}}\right)_{x}  +  \mu \frac{\bar{u}^2_{x}}{\bar{v}}    + R_2,  \\[2mm]
\end{aligned}
  \right.
\end{equation*}
where
%\begin{equation*}
\begin{align*}
R_{1}
& =\frac{\kappa(\gamma-1)}{\gamma R} \left( (\ln \bar{\theta})_{x t}-\mu\left(\frac{p_{+}}{R \bar{\theta}}(\ln \bar{\theta})_{x x}\right)_{x} \right)   \\[2mm]
&=O(\delta)(1+t)^{-\frac{3}{2}} \mathrm{e}^{-\frac{\hat{c} x^{2}}{1+t}}, \qquad \text { as }|x| \rightarrow \infty,  \\[2mm]
R_{2}
& =\left(\frac{\kappa(\gamma-1)}{\gamma R}\right)^{2}\left((\ln \bar{\theta})_{x}(\ln \bar{\theta})_{x t}-\mu\left(\frac{p_{+}}{R \bar{\theta}}(\ln \bar{\theta})_{x}(\ln \bar{\theta})_{x x}\right)_{x}\right)   \\[2mm]
&=O(\delta)(1+t)^{-2} \mathrm{e}^{-\frac{\hat{c} x^{2}}{1+t}}, \qquad \text { as }|x| \rightarrow \infty.
\end{align*}
%\end{equation*}

Define the perturbation of the solution as
\begin{equation}\nonumber
(\phi, \psi,\zeta)(x,t) = (v - \bar v, u - \bar u, \theta - \bar{\theta})(x,t).
\end{equation}
For interval $I \subset [0, \infty)$, we define a function space $X(I)$ as
$$X(I)=\left \{
(\phi, \psi, \zeta, E, b) \left|\;
\begin{aligned}
& \left(\phi, \psi, \zeta \right) \in L^{\infty}\left(I ; H^{2}(\mathbb{R}) \right),  \;\;  \left( E, b\right) \in L^{\infty}\left(I ; H^{1}(\mathbb{R})\right),   \\[2mm]
& \;\; \phi_{x} \in L^{2}\left(I ; H^{1}(\mathbb{R}) \right), \; \;   \left(\psi_{x}, \zeta_{x}\right) \in L^{2}\left(I ; H^{2}(\mathbb{R}) \right),   \\[2mm]
& \left( E_{x}, b_{x} \right)  \in L^{2}\left(I ; L^{2}(\mathbb{R}) \right)
\end{aligned}\right.
\right\}.$$
We now can state our first main result:
\begin{theorem}[Viscous contact wave]\label{thm1}
Assume that the initial data satisfy $\eqref{wuqiong1}$, $\eqref{czuoyouzhuangtai}$, and the dielectric constant $\varepsilon$ satisfies
\begin{equation}\label{jiedianchangshuxiao}
0<\varepsilon<\bar{C}
\end{equation}
 for some positive constant $\bar{C}$ $($depending only on $v_\pm$ and $|u_-|$$)$.
 There exist two small positive constants $ \varepsilon_1 $ and $ \delta_1 $ which are independent of $T$, such that if $ 0 < \delta < \min \{\delta_1, \bar{\delta}\}$ and the initial data $(v_0, u_0, \theta_0, E_0, b_0)(x)$ satisfy
\begin{equation*}
  \|\left( v_{0}(x)- \bar v(x,0),\, u_{0}(x) - \bar u(x,0),\,  \theta_0(x)- \bar {\theta}(x,0) \right) \|^2_{H^2}       + \|( E_{0}(x),\,  b_{0}(x)) \|_{H^1}^2 \leq \varepsilon_1,
\end{equation*}
then the Cauchy problem $\eqref{lagrange}$, $\eqref{chuzhi'}$, $\eqref{wuqiong1}$ admits a unique global solution $(v,u,\theta,E,b)(x,t)$ satisfying $ \left( v-\bar{v}, u - \bar{u}, \theta - \bar{\theta}, E, b \right)(x,t) \in X ( [0, \infty) ) $ and
\begin{equation*} %\label{mainjieguo1}
\sup_{t> 0}  \left(   \|(v - \bar v,\, u - \bar u,\, \theta - \bar {\theta})(\cdot,t)\|_{H^2}^2      + \|(E,\, b)(\cdot,t)\|_{H^1}^2  \right)   \leq  C_{0}\varepsilon_1.
\end{equation*}
Moreover, the solution $(v,u,\theta,E,b)(x,t)$ tends time-asymptotically to the viscous contact wave in the sense that
\begin{equation}\label{dashijianxingwei1}
\lim_{t\to{+\infty}}\sup_{x\in{\mathbb{R}}}|(v,u,\theta,E,b)(x,t) - (\bar{v}, \bar{u}, \bar{\theta}, 0, 0)({x},{t})|=0.
\end{equation}
\end{theorem}

%\vspace{4mm}

\begin{remark}\label{remark1}
In the case of $ u_- \neq 0 $, due to \eqref{vepu-2}, we can take the constant $\bar{C}$ in \eqref{jiedianchangshuxiao} as
\begin{equation}\label{barc}
 \bar{C}  =  \frac{\min\{v_\pm\}}{64 \max\{v_\pm\} u_-^2 } .
\end{equation}
Then for each given $\varepsilon$ satisfying the condition \eqref{jiedianchangshuxiao}, our system \eqref{lagrange} is definitely well-defined.
On the one hand, when we take $ \left| u_- \right|  $ suitably small,
 the dielectric
constant $ \varepsilon $ can be large enough, which can be seen from the conditions \eqref{jiedianchangshuxiao} and \eqref{barc}.
This fact can relax the requirement of smallness of $ \varepsilon $.
On the other hand, the conditions \eqref{jiedianchangshuxiao} and \eqref{barc} together can relax the restriction on $ \left| u_- \right|  $ as long as the dielectric constant $ \varepsilon $ is suitably small.
Furthermore, in the case of $ u_- = 0 $,
the dielectric constant $\varepsilon$ can be an arbitrary given positive constant.
In other words, we can remove the technical condition \eqref{jiedianchangshuxiao}.
This can be easily verified from \eqref{danjifen1} in Lemma \ref{dijieguji}; \eqref{Exbx4} and \eqref{lemshi2} in Lemma \ref{exbxyijiedao} mutually.
Thus, an interesting problem occurs,
that is
how to remove the technical condition \eqref{jiedianchangshuxiao} for the case of $ u_- \neq 0 $ in future.
\end{remark}

%\vspace{4mm}

\begin{remark}\label{remark2}

It should be pointed out that, we can also prove the asymptotic stability of a single rarefaction wave for the compressible non-isentropic Navier-Stokes-Maxwell equations in the Lagrangian coordinates.
Indeed, by employing the property of the smooth approximate rarefaction wave and using similar elementary energy method for a single viscous contact wave in this article,
we can analogously prove that the Cauchy problem $\eqref{lagrange}$, $\eqref{chuzhi'}$, $\eqref{wuqiong1}$ admits a unique global solution $(v,u,\theta,E,b)(x,t)$,
which tends time-asymptotically to the rarefaction wave in the sense of $ L^\infty_x $ norm.
For more details of calculation, we refer interested readers to \cite{yaozhu2021} for reference, in which the Navier-Stokes-Maxwell equations are described by the Eulerian coordinates.

%It should be pointed out that, for the case of the asymptotic stability of a single rarefaction wave in the Lagrangian coordinates, by employing the property of the smooth approximate rarefaction wave and using similar elementary energy method for a single viscous contact wave in this article, we can analogously prove that the Cauchy problem $\eqref{lagrange}$, $\eqref{chuzhi'}$, $\eqref{wuqiong1}$ admits a unique global solution $(v,u,\theta,E,b)(x,t)$, which tends time-asymptotically to the rarefaction wave in the sense of $ L^\infty_x $ norm. We refer interested readers to \cite{yaozhu2021} in the Eulerian coordinates for reference.

\end{remark}

Next, we will state the second main result. To investigate the asymptotic behavior of the solution toward the combination of viscous contact wave with two rarefaction waves, we assume that
\begin{equation*}
\left(v_{+}, u_{+}, \theta_{+}\right) \in R_{1} C R_{3}\left(v_{-}, u_{-}, \theta_{-}\right) \subset \Omega\left(v_{-}, u_{-}, \theta_{-}\right),
\end{equation*}
where
%\begin{equation*}
\begin{align*}
& \quad  R_{1} C R_{3} \left(v_{-}, u_{-}, \theta_{-}\right) :=     \\[2mm]
& \left\{(v, u, \theta) \in \Omega\left(v_{-}, u_{-}, \theta_{-}\right)   \left|\; s \neq s_{-},\;\,  u \geq u_{-}  - \int_{v_{-}}^{\mathrm{e}^{\frac{\gamma-1}{R \gamma}\left(s_{-}-s\right)} v}  \lambda_{-}\left(\eta, s_{-}\right) \mathrm{d} \eta, \right.  \right.    \\[2mm]
& \qquad\qquad\qquad\qquad\qquad\qquad\qquad\quad  \left.\quad  u \geq u_{-} -\int^v _{\mathrm{e}^{\frac{\gamma-1}{R \gamma} \left(s-s_{-}\right)} {v_{-}}} \lambda_{+}(\eta, s) \,\mathrm{d} \eta.\;\; \right\}
\end{align*}
%\end{equation*}
with
\begin{equation*}
s=\frac{R}{\gamma-1} \ln \frac{R \theta}{A}+R \ln v, \qquad s_{\pm}=\frac{R}{\gamma-1} \ln \frac{R \theta_{\pm}}{A}+R \ln v_{\pm}
\end{equation*}
and
\begin{equation*}
  \lambda_{\pm}(v, s) = \pm \sqrt{A \gamma v^{-\gamma-1} \mathrm{e}^{(\gamma-1) \frac{s}{R}}} .
\end{equation*}
As in \cite{Smoller} (or \cite{huangfm2010}) pointed out, there exist some suitably small $ \tilde{\delta} > 0 $ such that for
\begin{equation}\label{r1cr3youzhuangtai}
\left(v_{+}, u_{+}, \theta_{+}\right) \in R_{1} C R_{3}\left(v_{-}, u_{-}, \theta_{-}\right), \quad   \left| \theta_+  -  \theta_- \right| \leq \tilde{\delta},
\end{equation}
there exists a positive constant $ C_1 = C ( \theta_-, \tilde{\delta} ) $ and a unique pair of points $ \left( v^m_-, u^m, \theta^m_- \right)  $ and $ \left( v^m_+, u^m, \theta^m_+ \right)  $ in $  \Omega\left(v_{-}, u_{-}, \theta_{-}\right) $ satisfying
\begin{equation*}
  \frac{R \theta_{-}^{m}}{v_{-}^{m}}=\frac{R \theta_{+}^{m}}{v_{+}^{m}} =: p^{m}
\end{equation*}
and
\begin{equation*}
\left|v_{\pm}^{m}-v_{\pm}\right|+\left|u^{m}-u_{\pm}\right|+\left|\theta_{\pm}^{m}-\theta_{\pm}\right| \leq C_1 \left|\theta_{+}-\theta_{-}\right| .
\end{equation*}
The points $ \left( v^m_-, u^m, \theta^m_- \right)  $ and $ \left( v^m_+, u^m, \theta^m_+ \right)  $ not only lie in the 1-rarefaction wave curve $ R_-\left(v_{-}, u_{-}, \theta_{-}\right) $ and the 3-rarefaction wave curve $ R_+\left(v_{+}, u_{+}, \theta_{+}\right) $ respectively,
but also may correspondingly coincide with $ \left( v_-, u_-, \theta_- \right)  $ and $ \left( v_+, u_+, \theta_+ \right)  $.
Here
\begin{equation*}
R_{\pm}\left(v_{\pm}, u_{\pm}, \theta_{\pm}\right) = \left\{(v, u, \theta) \left|\; s=s_{\pm},\;\; u=u_{\pm}-\int_{v_{\pm}}^{v} \lambda_{\pm}\left(\eta, s_{\pm}\right) \mathrm{d} \eta,\; v>v_{\pm}   \right. \right\}.
\end{equation*}

The 1-rarefaction wave $ \left( v^r_-, u^r_-, \theta^r_- \right)(\frac{x}{t})  $ (\,respectively,\, the 3-rarefaction wave $ \left( v^r_+, u^r_+, \theta^r_+ \right)(\frac{x}{t}) $) connecting $ \left( v_-, u_-, \theta_- \right)  $ and $ \left( v^m_-, u^m, \theta^m_- \right) $
(respectively $ \left( v^m_+, u^m, \theta^m_+ \right)  $ and $ \left( v_+, u_+, \theta_+ \right) $) is the weak solution of the Riemann problem of the Euler system \eqref{nstuidao2} with the following Riemann initial data
\begin{equation}\label{rareriemanchuzhi}
(v^r_\pm, u^r_\pm, \theta^r_\pm)(x,0)=\left\{
\begin{aligned}
&(v^m_\pm, u^m, \theta^m_\pm), \quad  \pm x<0,\\[2mm]
&(v_\pm, u_\pm, \theta_\pm), \quad \;\pm x>0.
\end{aligned}
\right.
\end{equation}
Since the rarefaction waves $ (v^r_\pm, u^r_\pm, \theta^r_\pm)(\frac{x}{t}) $ are not smooth enough solutions, we shall construct smooth approximate ones. Motivated by \cite{Matsumura1986}, the smooth
solutions of Euler system \eqref{nstuidao2}, $ (V^r_\pm, U^r_\pm, \Theta^r_\pm)(x,t) $, which approximate $ (v^r_\pm, u^r_\pm, \theta^r_\pm)(\frac{x}{t}) $, are given by
\begin{equation}\label{gouzaoxishubo}
\left\{
\begin{aligned}
&\lambda_{\pm}\left(V_{\pm}^{r}(x, t), s_{\pm}\right) = w_{\pm}(x, 1+t) ,   \\[2mm]
&U_{\pm}^{r}(x, t) = u_{\pm}-\int_{v_{\pm}}^{V_{\pm}^{r}(x, t)} \lambda_{\pm}\left(\eta, s_{\pm}\right) \mathrm{d} \eta ,  \\[2mm]
&\Theta_{\pm}^{r}(x, t) = \theta_{\pm} v_{\pm}^{\gamma-1} \left( V_{\pm}^{r}(x, t) \right)^{1-\gamma} ,
\end{aligned}
\right.
\end{equation}
where $ w_- (x, t)$ (respectively $ w_+ (x, t)$) is the solution of the initial problem for the typical
Burgers equation:
\begin{equation}\label{wcauchy}
\left\{\begin{array}{l}
{w}_t+{w}{w}_x=0, \\[2mm]
\displaystyle {w}(x, 0) = \frac{w_{r} + w_{l}}{2}  +\frac{w_{r} - w_{l}}{2} \cdot\frac{\mathrm{e}^x-\mathrm{e}^{-x}}{\mathrm{e}^x+\mathrm{e}^{-x}},
\end{array}\right.
\end{equation}
with $ w_l = \lambda_-(v_-, s_-) $, $ w_r = \lambda_-(v^m_-, s_-) $ (respectively $ w_l = \lambda_+(v^m_+, s_+) $, $ w_r = \lambda_+(v_+, s_+) $).

Let $ \left( V^c, U^c, \Theta^c \right)(x,t) $ be the viscous contact wave constructed in \eqref{kuosanfangcheng} and \eqref{smoothjianduanbo} with $ (v_\pm, u_\pm, \theta_\pm) $ replaced by $ (v^m_\pm, u^m, \theta^m_\pm) $, and $ p_+ $ replaced by $ p^m $, respectively.
Now we define the smooth composite wave as
\begin{equation*}%\label{comwavefuhao}
\left(\begin{array}{l}
       V \\[2mm]
       U \\[2mm]
       \Theta
      \end{array}\right)(x, t)
= \left(\begin{array}{l}
    V^{c} + V_{-}^{r} + V_{+}^{r} \\[2mm]
    U^{c} + U_{-}^{r} + U_{+}^{r} \\[2mm]
    \Theta^{c} + \Theta_{-}^{r} + \Theta_{+}^{r}
    \end{array}\right)(x, t)
- \left(\begin{array}{c}
         v_{-}^{m} + v_{+}^{m} \\[2mm]
         2 u^{m} \\[2mm]
         \theta_{-}^{m} + \theta_{+}^{m}
        \end{array}\right)
\end{equation*}
and the perturbation as
\begin{equation*}
(\phi, \psi, \zeta)(x, t) = (v-V, u-U, \theta-\Theta)(x, t).
\end{equation*}
The following theorem is our second main result.

\begin{theorem}[Composite wave]\label{thm2}
Assume that the initial data satisfy $\eqref{wuqiong1}$, $\eqref{r1cr3youzhuangtai}$ for some small $ \tilde{\delta} > 0 $,
and the dielectric constant $\varepsilon$ satisfies
\begin{equation*}
0<\varepsilon<\bar{C}
\end{equation*}
 for some positive constant $\bar{C}$ $($depending only on $ v_\pm $, $ u_\pm $ and $ \theta_\pm $$)$. There exist two small positive constants $ \varepsilon_2 $ and $ \delta_2 $ $($$\leq \min \{\bar{\delta}, \tilde{\delta}\} $$)$ which are independent of $T$,
 such that if $ 0 < \delta < \delta_2 $ and the initial data $(v_0, u_0, \theta_0, E_0, b_0)(x)$ satisfy
\begin{equation*}
  \|\left( v_{0}(x)- V(x,0),\, u_{0}(x) - U(x,0),\,  \theta_0(x)- \Theta(x,0) \right) \|^2_{H^2}       + \|( E_{0}(x),\,  b_{0}(x)) \|_{H^1}^2 \leq \varepsilon_2,
\end{equation*}
then the Cauchy problem $\eqref{lagrange}$, $\eqref{chuzhi'}$, $\eqref{wuqiong1}$ admits a unique global solution $ (v,u,\theta,E,b)(x,t) $ satisfying $ \left( v - V, u - U, \theta - \Theta, E, b \right)(x,t) \in X\left( [0, \infty) \right) $ and
\begin{equation*} %\label{mainjieguo1'}
\sup_{t> 0}  \left(   \|(v - V,\, u - U,\, \theta - {\Theta})(\cdot,t)\|_{H^2}^2      + \|(E,\, b)(\cdot,t)\|_{H^1}^2  \right)   \leq  C_{0}\varepsilon_2.
\end{equation*}
Moreover, the solution $(v,u,\theta,E,b)(x,t)$ tends time-asymptotically to the composite wave in the sense that
\begin{equation}\label{dashijianxingwei2}
\lim _{t \rightarrow \infty} \sup _{x \in \mathbb{R}}
\left(\begin{array}{c}
\left|\left(v - V^{c} - v_{-}^{r} - v_{+}^{r} + v_{-}^{m} + v_{+}^{m}\right)(x, t)\right|   \\[2mm]
\left|\left(u - U^c - u_{-}^{r} - u_{+}^{r} + 2u^{m}\right)(x, t)\right|   \\[2mm]
\left|\left(\theta - \Theta^{c} - \theta_{-}^{r} - \theta_{+}^{r} + \theta_{-}^{m} + \theta_{+}^{m}\right)(x, t)\right|
\end{array}\right) = 0,
\end{equation}
where $(v^r_-, u^r_-, \theta^r_-)(x,t)$ and $(v^r_+, u^r_+, \theta^r_+)(x,t)$ are the 1-rarefaction and 3-rarefaction waves uniquely determined by the Riemann problem \eqref{nstuidao2}, \eqref{rareriemanchuzhi}, respectively.
\end{theorem}

\begin{remark}\label{remark3}
From \eqref{vepu-2'}, we can take the constant $ \bar{C} $ as
\begin{equation*}
 \bar{C}  =    \min \left\{ \frac{\min\{v_\pm\}}{ 80 \max\{v_\pm\}  \cdot  \left( \max \left\{ \left| u_\pm \right|  \right\} \right)^2 } , \;\;        \frac{   \left( \sqrt{\gamma R} \max \left\{   \sqrt{\theta_+} v_+ ^{-1},\, \sqrt{\theta_-} v_- ^{-1}   \right\} \right)^{-1}  }{   32 \max \left\{ v_\pm \right\} \cdot  \max \left\{ \left| u_\pm \right| \right\} }  \right\} .
\end{equation*}
Like the discussion on Remark \ref{remark1}, when we take $ \max \left\{ \left| u_- \right|, \left| u_+ \right|  \right\} $ suitably small, the dielectric constant $ \varepsilon  $ can be large enough, and vice versa.

\end{remark}

\begin{remark}\label{remark4}
For the compressible non-isentropic Navier-Stokes-Maxwell equations,
the stability of the remaining wave patterns, namely, shock wave, stationary wave and their compositions with viscous contact wave or rarefaction waves can also be taken into account, which will be studied by the authors in future.
\end{remark}

The rest of the paper is organized as follows.
In section \ref{section2}, we will
list the elementary inequality which is the key point for analysing viscous contact wave;
and collect some fundamental facts concerning viscous contact wave and rarefaction waves which are necessary in our subsequent analysis.
The proof of Theorem \ref{thm1} and Theorem \ref{thm2} are given in section \ref{section3} and section \ref{section4}, respectively.

%%%%%%%%%%%%%%%%%%%%%%%%%%%%%%%%%%%%%%%%%%%%%%%%%%%%%%%%%%%%%%%%%%%%%%%%%%%%%%%%%%%%%%%%%%%%%%%%%

\section{Preliminaries}\label{section2}
We first list a basic inequality concerning the time-space integrable estimation with the square of the heat kernel as a weight function,
which will play an important role in the sequel.
This result is borrowed from \cite{huangfm2010} and we skip the details here.
For $ \alpha > 0 $, we define
\begin{equation}\label{omegarehe}
  \omega(x,t)=\left( 1+t \right)^{-\frac{1}{2}}\mathrm{e}^{-\frac{\alpha x^2}{1+t}},\qquad g(x,t) = \int_{-\infty}^x \omega(y,t)\,\mathrm{d}y,
\end{equation}
and can easily check that
\begin{equation}\label{gtguanxi}
  \omega_t = \frac{1}{4\alpha} \omega_{xx},  \quad  4 \alpha g_t = \omega_x, \quad \| g(\cdot,t) \|_{L^\infty} = \sqrt{\pi} \alpha^{-\frac{1}{2}}.
\end{equation}

\begin{lemma}\label{heatkernel}
  For $0 < T \leq +\infty$, suppose that $ h(x,t) $ satisfies
\begin{equation*}
 h_x \in L^2(0,T; L^2(\mathbb{R})),\quad h_t \in L^2(0,T; H^{-1}(\mathbb{R})).
\end{equation*}
Then the following estimate holds:
\begin{equation}\label{reheguji}
  \int_0^T\int_{\mathbb{R}} h^2 \omega^2 \,\mathrm{d}x \mathrm{d}t    \leq  4 \pi \| h_0 \|^2    +  \frac{4 \pi}{\alpha} \int_0^T \| h_x \|^2 \,\mathrm{d}t      + 8 \alpha \int_0^T \left< h_t, h g^2 \right> \,\mathrm{d}t,
\end{equation}
where $\left<\;\cdot\;, \;\cdot\; \right>$ denotes the inner product on $H^{-1}(\mathbb{R}) \times H^1(\mathbb{R})$.

\end{lemma}

%%%%%%%%%%%%%%%%%%%%%%%%%%%%%%%%%%%%%%%%
Next, the $ L^2 $-norm of the viscous contact wave $ \left( \bar{v}, \bar{u}, \bar{\theta} \right)  $ defined by \eqref{smoothjianduanbo} are useful in the subsequent sections.

\begin{lemma}\label{jianduanbol2mo}
Set the strength of wave $ \delta : = \left| \theta_+ - \theta_- \right| \leq  \bar{\delta} $, where $\bar{\delta}$ is a small positive constant in Lemma \ref{jianduanboshuaijian}.
Then for $ k \geq 1 $, the viscous contact wave $ \left( \bar{v}, \bar{u}, \bar{\theta} \right)  $ has the following $ L^2 $-estimates:
\begin{itemize}
\item[$\mathrm{(i)}$]
\begin{equation*}
  \| \left( \partial_x^k \bar{v},\, \partial_x^k \bar{\theta} \right) \|  \leq C \delta \left( 1+t \right)^{-\frac{2k - 1}{4}}.
\end{equation*}

\item[$\mathrm{(ii)}$]
\begin{equation*}
   \| \partial_x^k \bar{u} \| \sim \| \partial_x^{k+1} \bar{\theta}\|  \leq C \delta \left( 1+t \right)^{-\frac{2k + 1}{4}}.
\end{equation*}

 \item[$\mathrm{(iii)}$]
\begin{equation*}
    \| \partial_t^k \bar{u} \| \sim \| \partial_x^{2k+1} \bar{\theta}\|  \leq C \delta \left( 1+t \right)^{-\frac{4k + 1}{4}} .
\end{equation*}

\end{itemize}

\end{lemma}

Moreover, by using the Sobolev inequality
\begin{equation}\label{lwuqiong}
\|f\|_{\infty}\leq \sqrt{2}\|f\|^\frac{1}{2}\|f_x\|^\frac{1}{2},\quad {\rm for}\,\,f(x)\in H^1(\mathbb{R}),
\end{equation}
together with the {\it a priori} assumption \eqref{xianyanjiashe} for a single viscous contact wave (respectively for the composite wave case),
it is easy to deduce the perturbation of the solution satisfying
\begin{equation}\label{raodongwuqiongguji}
\|(\phi, \psi, \zeta, E, b, \phi_x, \psi_x, \zeta_x)\|_{L^{\infty}}\leq \sqrt{2}\varepsilon_0.
\end{equation}
Due to $\eqref{raodongwuqiongguji}$, \eqref{smoothjianduanbo} (respectively \eqref{gouzaoxishubo}) and the smallness of $\varepsilon_0$, we can get the following useful results.
\begin{lemma}\label{changyongshangxiajie}
For a single viscous contact wave case, it holds that
\begin{itemize}
\item[$\mathrm{(i)}$]
If $ \theta_- < \theta_+ $, then
\begin{equation*}
  \theta_- \leq  \bar{\theta}(x,t)  \leq \theta_+, \quad  v_-  \leq  \bar{v}(x,t) \leq  v_+ .
\end{equation*}
%%%%%%%%%%%%%%%%%%%%%%%%%%%%%%%%%%%%
If $ \theta_+ < \theta_- $, then
\begin{equation*}
  \theta_+ \leq  \bar{\theta}(x,t)  \leq \theta_-, \quad  v_+  \leq  \bar{v}(x,t) \leq  v_- .
\end{equation*}

\item[$\mathrm{(ii)}$] If $ \delta = \left| \theta_+ - \theta_- \right|  $ is small enough, then
\begin{gather}
0<\frac{1}{2}\min\{v_\pm\} \leq v = \phi + \bar{v}  \leq \frac{3}{2}\max\{v_\pm\},      \nonumber  \\[2mm]
0<\frac{1}{2}\min\{\theta_\pm\} \leq \theta = \zeta + \bar{\theta}  \leq \frac{3}{2}\max\{\theta_\pm\},      \nonumber  \\[2mm]
\left| \bar{u} \right| \leq  \frac{5}{4} \left| u_- \right|,    \quad  |u|\leq |\psi| + |\bar{u}|  \leq \frac{3}{2}|u_-| .   \label{contactushangjie}
\end{gather}

\end{itemize}

%%%%%%%%%%%%%%%%%%%%%%%%%%%%%%%%%%%%
For the composite wave case, it holds that
\begin{itemize}
\item[$\mathrm{(iii)}$]
\begin{equation*}
   0 < v_\pm < V^r_\pm(x,t) < v^m_\pm ,\quad
   0 < \theta^m_\pm < \Theta^r_\pm(x,t) < \theta_\pm .
\end{equation*}

\item[$\mathrm{(iv)}$]
\begin{equation*}
  u_- < U^r_-(x,t) < u^m, \quad u^m < U^r_+(x,t) < u_+ .
\end{equation*}

\item[$\mathrm{(v)}$]
If $ \theta^m_-  < \theta^m_+ ,$ then
\begin{equation*}
 \theta^m_- < \Theta^c(x,t) < \theta^m_+ , \quad v^m_- < V^c(x,t) < v^m_+ .
\end{equation*}

If $ \theta^m_+  < \theta^m_- ,$ then
\begin{equation*}
  \theta^m_+ < \Theta^c(x,t) < \theta^m_- , \quad  v^m_+ < V^c(x,t) < v^m_- .
\end{equation*}

\item[$\mathrm{(vi)}$] If $ \delta = \left| \theta_+ - \theta_- \right| $ is small enough, then
\begin{align}
\frac{1}{2} \min\left\{ v_\pm \right\} < V < \frac{3}{2} \max \left\{ v_\pm \right\},&   \quad    \frac{1}{4} \min\left\{ v_\pm \right\} < v < 2 \max \left\{ v_\pm \right\},   \label{fuhevshangjie}    \\[2mm]
\frac{1}{2} \min\left\{ \theta_\pm \right\} < \Theta < \frac{3}{2} \max \left\{ \theta_\pm \right\},&  \quad   \frac{1}{4} \min\left\{ \theta_\pm \right\} < \theta < 2 \max \left\{ \theta_\pm \right\},       \nonumber  \\[2mm]
\left| U \right| \leq \frac{3}{2} \max \left\{ \left| u_\pm \right| \right\},&   \quad    \left| u \right| \leq 2 \max \left\{ \left| u_\pm \right| \right\} .   \label{Uushangxiajie}
\end{align}

\end{itemize}

\end{lemma}

%%%%%%%%%%%%%%%%%%%%%%%%%%%%%%%%%%%%%%%%%

In addition, we state the following properties of the solution to the problem \eqref{wcauchy} according to \cite{Matsumura1986}.

\begin{lemma}\label{burgerxingzhi}
For given $w_l\in\mathbb{R}$ and $\bar w>0$, let $w_r \in \{w\left|\;0 < \tilde w := w - w_l < \bar w \right.\}$.
Then the problem \eqref{wcauchy} has a unique smooth global solution in time satisfying the following properties:
\begin{itemize}
\item[$\mathrm{(i)}$] $w_l<w(x,t)<w_r$, \, $w_x>0$ $(x\in\mathbb{R},t>0)$.
\item[$\mathrm{(ii)}$] For $ q \in[1,\infty]$, there exists some positive constant $C = C(q,w_l,\bar w)$ such that
for $\tilde w\geq0$ and $t\geq0$,
$$
\|w_x(\cdot,t)\|_{L^q}\leq C\min\{\tilde w,\tilde w^{\frac{1}{q}}t^{-1 + \frac{1}{q}}\}, \quad
\|w_{xx}(\cdot,t)\|_{L^q}\leq C\min\{\tilde w,t^{-1}\}.
$$
\item[$\mathrm{(iii)}$] If $w_l>0$, for any $(x,t)\in(-\infty,0]\times[0,\infty)$,
$$
|w(x,t)-w_l|\leq\tilde w \mathrm{e}^{-2(|x|+w_lt)},\quad
|w_x(x,t)|\leq 2\tilde w \mathrm{e}^{-2(|x|+w_lt)}.
$$
\item[$\mathrm{(iv)}$] If $w_r<0$, for any $(x,t)\in[0,\infty)\times[0,\infty)$,
$$
|w(x,t)-w_r|\leq\tilde w \mathrm{e}^{-2(x + |w_r|t)},\quad
|w_x(x,t)|\leq2\tilde w \mathrm{e}^{-2(x + |w_r|t)}.
$$
\item[$\mathrm{(v)}$]
For the Riemann solution $w^r({x}/{t})$ of the scalar equation $\eqref{wcauchy}_1$ with the Riemann initial data
\begin{equation*}
w(x,0)=\left\{
\begin{array}{ll}
w_l,&\quad x<0,\\[2mm]
w_r,&\quad x>0,
\end{array}
\right.
\end{equation*}
we have
$$
\lim_{t\rightarrow+\infty}\sup_{x\in\mathbb{R}}|w(x,t)-w^r({x}/{t})|=0.
$$
\end{itemize}

\end{lemma}

Finally, we divide $\mathbb{R}\times(0,t)$ into
three parts, that is $\mathbb{R}\times(0,t)=\Omega_-\cup\Omega_c\cup\Omega_+$ with
$$
\Omega_{\pm}=\left\{(x,t)\left|\;\pm2x>\pm\lambda_{\pm}(v_{\pm}^m,s_{\pm})t\right.\right\}
$$
and
$$
\Omega_{c}=\left\{(x,t)\left|\;\lambda_-(v_-^m,s_-)t\leq 2x \leq \lambda_{+}(v_{+}^m, s_{+})t \right.\right\}.
$$
Then Lemma \ref{burgerxingzhi} and Lemma \ref{jianduanboshuaijian} lead to the following valuable lemma, which can be found in \cite{huangfm2010}.

\begin{lemma}\label{rare-pro}
For any given $(v_-,u_-,\theta_-)$, we assume that $ (v_+, u_+, \theta_+) $ satisfies \eqref{r1cr3youzhuangtai} with   $ \delta = \left| \theta_+ - \theta_-  \right| < \bar{\delta} $.
Then the smooth rarefaction waves $(V_{\pm}^r, U_{\pm}^r, \Theta_{\pm}^r)$
constructed in \eqref{gouzaoxishubo} and the viscous contact wave $(V^{c}, U^{c}, \Theta^{c})$ constructed in \eqref{smoothjianduanbo} satisfying the following properties:
\begin{itemize}
\item[$\mathrm{(i)}$] $(U_{\pm}^r)_x \geq 0$ \;$(x \in \mathbb{R},\, t>0)$.
\item[$\mathrm{(ii)}$] For $1\leq q \leq \infty$, there exists a positive constant $C = C(q, v_-, u_-, \theta_-, \bar{\delta}, \tilde{\delta})$ such that for $\delta = \left| \theta_+ - \theta_- \right| $ and $ t \geq 0 $,
$$
\left\| \left[ (V_{\pm}^r)_x,(U_{\pm}^r)_x,(\Theta_{\pm}^r)_x  \right](\cdot, t) \right\|_{L^q}
\leq C\min\left\{\delta, \delta^{\frac{1}{q}}(1+t)^{-1 + \frac{1}{q}}\right\}
$$
and
$$
\left\| \left[ (V_{\pm}^r)_{xx},(U_{\pm}^r)_{xx},(\Theta_{\pm}^r)_{xx} \right] (\cdot,t) \right\|_{L^q}
\leq C\min \left\{\delta, (1+t)^{-1}\right\}.
$$
\item[$\mathrm{(iii)}$] There exists a positive constant $C=C(v_-, u_-, \theta_-, \bar{\delta}, \tilde{\delta})$
such that for
$$
c_0=\frac{1}{10}\min\left\{\left| \lambda_-(v_-^m,s_-) \right| , \lambda_+(v_+^m,s_+), \hat{c} \lambda_-^2(v_-^m,s_-), \hat{c} \lambda_+^2(v_+^m,s_+), 1 \right\},
$$
we have in $\Omega_c$
$$
\left| (V_{\pm}^r)_x \right|     + (U_{\pm}^r)_x    + \left| (\Theta_{\pm}^r)_x \right|     + \left| V_{\pm}^r - v_{\pm}^m \right|     + \left| \Theta_{\pm}^r - \theta_{\pm}^m \right|
\leq C\delta \mathrm{e}^{-c_0(|x| + t)}
$$
and in $\Omega_{\mp}$
$$
\left\{
\begin{array}{ll}
\left| V^{c}_x \right|    + \left| U^{c}_x \right|   + \left| \Theta^{c}_x \right|      + \left| V^{c} - v_{\mp}^m \right|     + \left| \Theta^{c} - \theta_{\mp}^m \right|    \leq C\delta \mathrm{e}^{-c_0(|x|+t)},   \\[2mm]
\left| (V_{\pm}^r)_x \right|      + (U_{\pm}^r)_x     + \left| (\Theta_{\pm}^r)_x \right|      + \left| V_{\pm}^r - v_{\pm}^m \right|      + \left| \Theta_{\pm}^r - \theta_{\pm}^m \right|
\leq C\delta \mathrm{e}^{-c_0(|x|+t)}.
\end{array}
\right.
$$

\item[$\mathrm{(iv)}$] For the rarefaction waves $(v_{\pm}^r, u_{\pm}^r, \theta_{\pm}^r)(x/t)$ determined by \eqref{nstuidao2} and \eqref{rareriemanchuzhi}, it holds
$$
\lim_{t\rightarrow +\infty}\sup_{x\in\mathbb{R}} \left| (V_{\pm}^r, U_{\pm}^r, \Theta_{\pm}^r)(x,t) - (v_{\pm}^r, u_{\pm}^r, \theta_{\pm}^r)(x/t) \right|  = 0.
$$
\end{itemize}
\end{lemma}

%%%%%%%%%%%%%%%%%%%%%%%%%%%%%%%%%%%%%%%%%%%%%%%%%%%%%%%%%%%%%%%%%%%%%%%%%%%%%%%%%%%%%%

\section{Proof of Theorem \ref{thm1} (Viscous contact wave)}\label{section3}
To prove Theorem \ref{thm1} for a single viscous contact wave, we use elementary energy method.
It is easy to derive that $(\phi, \psi,\zeta, E, b)$ satisfies
\begin{equation}\label{raodong}
\left\{
\begin{aligned}
&\phi_t-\psi_x=0, \\[2mm]
&\psi_t + \left( \frac{R\zeta- {p}_+ \phi}{v} \right)_x     = \mu \left( \frac{\psi_x}{v} \right)_x     - \mu \left( \frac{\phi \bar{u}_x}{v \bar v} \right)_x   -\left[ \bar u_t - \mu \left( \frac{\bar u_x}{\bar v} \right)_x  \right]  \\[2mm]
& \qquad\qquad\qquad\qquad\quad\;\;\,    -v(E+\psi b+ \bar {u} b)b, \\[2mm]
&\frac{R}{\gamma-1} \zeta_{t} + p \psi_{x}      = \kappa\left( \frac{\bar v \zeta_x-\phi \bar \theta_x}{v \bar v} \right)_x     - \frac{R \zeta - p_+ \phi}{v} \bar u_x      + \mu\frac{(\psi_x +  \bar {u}_x)^2}{v} \\[2mm]
& \qquad\qquad\qquad\;\;\;\;\, + v(E+\psi b+ \bar {u} b)^2, \\[2mm]
&\varepsilon \left(E_{t}-\frac{u}{v}E_x\right)-\frac{1}{v}b_{x}+ E +\psi b+ \bar {u} b=0,   \\[2mm]
&b_{t}-\frac{u}{v}b_x-\frac{1}{v}E_{x}=0,
\end{aligned}
\right.
\end{equation}
with initial data for $x \in \mathbb{R}$,
\begin{equation}\label{raodongchuzhi}
\begin{aligned}[b]
&\;(\phi_0, \psi_0, \zeta_0, E_0, b_0)(x) \\[2mm]
:=&\;(v_0(x)- \bar v (x,0),u_0(x)-  \bar u (x,0),\theta_0(x)-  \bar \theta (x,0),E_0(x),b_0(x)).
\end{aligned}
\end{equation}

The local existence of solutions to the reformulated Cauchy problem $\eqref{raodong}$, $\eqref{raodongchuzhi}$ can be obtained by the
standard iteration argument.
To prove Theorem \ref{thm1} for brevity, we only devote ourselves to establishing the global-in-time {\it a priori} estimates in the following.

%´í¾ä: \begin{proposition}[{\it A priori} estimates]\label{prop1}
\begin{proposition}\label{prop1} $($\!{\it A priori} estimates$)$ Suppose all the conditions in Theorem \ref{thm1} hold. Let $(\phi,\psi,\zeta,E,b) \in X(0,T) $ be a smooth solution to the problem $\eqref{raodong}$, $\eqref{raodongchuzhi}$ on $0\leq t\leq T$ for $T>0$. There exist some suitably small positive constants $\delta_0$ and $\varepsilon_0$ such that if $\delta < \min \{\delta_0, \bar{\delta}\}$ and
\begin{equation}\label{xianyanjiashe}
 \sup_{0\leq t\leq T} \left(\| (\phi, \psi, \zeta)(t) \|_{H^2} + \|(E,b)(t)\|_{H^1} \right)  \leq \varepsilon_0,
\end{equation}
then $ (\phi, \psi, \zeta, E, b)(x,t) $ satisfies
%\begin{equation}\label{xianyanjieguo1'}
\begin{align} %[b]
&\quad \sup_{0\leq t\leq
T} \|(\phi,\psi,\zeta,\sqrt{\varepsilon}E,b)(t)\|^2_{H^1}     +   \int_0^T(\|(\phi_x,E_x,b_x)\|^2 + \|(\psi_x,\zeta_x)\|_{H^1}^2 )\,{\rm{d}}t    \nonumber \\[2mm]
& \quad    + \int_0^T \|E+\psi b+  \bar{u} b\|^2 \,{\rm{d}}t  \leq   C \delta^\frac{2}{3}  +  C\|(\phi_0,\psi_0,\zeta_0,E_0,b_0)\|^2_{H^1}     \label{xianyanjieguo1'}
\end{align}
%\end{equation}
and
%\begin{equation}\label{xianyanjieguo2}
  \begin{align} %[b]
   & \quad \sup_{0\leq t\leq
T} \| \left( \phi_{xx}, \psi_{xx}, \zeta_{xx} \right)  \|^2       +  \int_0^T \| \left( \phi_{xx}, \psi_{xxx}, \zeta_{xxx} \right)  \|^2  \,\mathrm{d}\tau       \nonumber \\[2mm]
   & \leq    C \delta^{\frac{2}{3}}    +  C \| ( \phi_0, \psi_0, \zeta_0 ) \|^2_{H^2}   +  C \| ( E_0, b_0 ) \|^2_{H^1} .      \label{xianyanjieguo2}
  \end{align}
%\end{equation}
\end{proposition}

Once Proposition \ref{prop1} is proved, we can close the {\it a priori} assumption \eqref{xianyanjiashe}.
The global existence of the solution to the Cauchy problem \eqref{raodong}, \eqref{raodongchuzhi} then follows from the
standard continuation argument based on the local existence and the {\it a priori} estimates \eqref{xianyanjieguo1'} and  \eqref{xianyanjieguo2}.
For $0<\varepsilon<\bar{C}$, the estimate \eqref{xianyanjieguo1'} and the equations \eqref{raodong} (respectively \eqref{raodong2}) imply that
\begin{equation*}
\int_0^\infty\left(\|(\phi_x,\psi_x,\zeta_x,E_x,b_x)(t)\|^2 + \left| \frac{\rm d }{{\rm d} t} \|(\phi_x,\psi_x,\zeta_x,E_x,b_x)(t)\|^2 \right| \right){\rm{d}}t <\infty,
\end{equation*}
which easily leads to
\begin{equation}\label{twuqiongfanshuling}
\lim_{t\to{+\infty}}\|(\phi_x,\psi_x,\zeta_x,E_x,b_x)(t)\| =0.
\end{equation}
Then using the Sobolev inequality \eqref{lwuqiong}, together with \eqref{xianyanjieguo1'} and \eqref{twuqiongfanshuling}, directly implies the large time behavior of the solution, that is, \eqref{dashijianxingwei1} (respectively \eqref{dashijianxingwei2}).

Proposition \ref{prop1} will be finished by the following several lemmas.
According to Remark \ref{remark1}, we only consider the case of $u_-\neq 0 $ and the case of $u_- = 0 $ is analogous.
We first give the zero-order energy estimate.

\begin{lemma}\label{dijieguji}
Suppose that all the conditions in Theorem \ref{thm1} hold. For all $0 < t < T $, there exists a constant $\bar{C}$ depending only on $v_\pm$ and $|u_-|$ such that if $ 0< \varepsilon <\bar{C} $, then
%\begin{equation}\label{jibennengliang}
  \begin{align} %[b]
 &\qquad  \| (\phi, \psi, \zeta, \sqrt{\varepsilon}E, b) \|^2         + \int_0^t   \| ( \psi_x, \zeta_x, E+\psi b+\bar{u}b ) \|^2  \,\mathrm{d}\tau     \nonumber \\[2mm]
 & \leq C\| (\phi_0, \psi_0, \zeta_0, E_0, b_0) \|^2 + C\delta    + C \delta^{\frac{2}{3}} \int_0^t\int_{\mathbb{R}} \left(\phi^{2} + \zeta^{2} + b^2 \right) \omega^2  \,\mathrm{d}x \mathrm{d}\tau   \nonumber \\[2mm]
 & \quad   +  C \delta  \int_0^t \| \phi_x  \|^2 \,\mathrm{d}\tau.       \label{jibennengliang}
  \end{align}
 %\end{equation}
Here $\omega$ is the function defined in \eqref{omegarehe} with $ \alpha = {\hat{c}}/{4} $.

\end{lemma}

\begin{proof}[Proof]
Motivated by the works of \cite{Kawashima1986} and \cite{huangfm2010}, multiplying $ \eqref{raodong}_1 $ by $ -R \bar{\theta} \left( \frac{1}{v} - \frac{1}{\bar{v}} \right)  $, $ \eqref{raodong}_2 $ by $ \psi $ and $ \eqref{raodong}_3 $ by $ \zeta \theta^{-1} $ respectively, and adding up the resulting equations, after some detailed calculation, we obtain
%\begin{equation}\label{shangliudui1}
\begin{align} %[b]
&\quad\left(\frac{1}{2} \psi^{2} +  R \bar{\theta} \Phi\left(\frac{v}{\bar v}\right) + \frac{R}{\gamma-1} \bar{\theta} \Phi\left(\frac{\theta}{\bar{\theta}}\right)\right)_{t} +\mu \frac{\psi_{x}^{2}}{v} +\kappa \frac{\zeta_{x}^{2}}{v \theta}  + H_{x}   \nonumber \\[2mm]
&=  F \psi    + Q     + v(E+\psi b+\bar{u}b)^2\frac{\zeta}{\theta}    -v(E+\psi b+\bar{u}b)\psi b,     \label{shangliudui1}
\end{align}
%\end{equation}
where
\begin{equation*}
 \Phi(s) = s - 1 -\ln s, \quad s>0,
\end{equation*}

\begin{equation*}
	H= (R\zeta-{p}_+ \phi)\frac{\psi}{v}-\mu\frac{\psi\psi_x}{v}-\kappa\frac{\zeta}{\theta v \bar{v}}(\bar{v}\zeta_x-\phi\bar{\theta}_x),
\end{equation*}

\begin{equation*}
   F = - \mu \left( \frac{\phi \bar{u}_x}{v \bar{v}} \right)_x      - \left[ \bar{u}_t - \mu \left( \frac{\bar{u}_x}{\bar{v}} \right)_x  \right]
\end{equation*}
and
   \begin{align*}
   Q &=  - \frac{R \bar{\theta} \phi^2 \bar{u}_x }{v \bar{v}^2 }     + R \bar{\theta}_t \Phi \left( \frac{v}{\bar{v}} \right)      - \frac{R}{\gamma-1} \frac{\zeta^2 \bar{\theta}_t}{\theta \bar{\theta} }         +  \frac{R}{\gamma-1} \bar{\theta}_t \Phi\left(\frac{\theta}{\bar\theta}\right)     + \kappa \frac{\phi \bar{\theta}_x \zeta_x}{\bar{v}v \theta}  \\[2mm]
     &\quad\,    + \kappa \frac{(\bar{v} \zeta_x - \phi \bar{\theta}_x)\cdot(\zeta \zeta_x + \zeta \bar{\theta}_x)}{v \bar{v} \theta^2}     - \frac{(R \zeta - p_+ \phi)\zeta}{v \theta} \bar{u}_x     + \mu \frac{(\psi_x + \bar{u}_x)^2}{v \theta}\zeta.
   \end{align*}
By applying $\eqref{raodongwuqiongguji}$ and the properties of viscous contact wave in Lemma \ref{jianduanboshuaijian}, we can derive
%\begin{equation}\label{FL1}
\begin{align*} %[b]
   \| F \|_{L^1}
   & \leq C \int_{\mathbb{R}} \left( \left| \phi_x \bar{u}_x  \right|   + \left| \phi \bar{u}_{xx} \right|  + \left| \phi \bar{u}_x \phi_x \right|   + \left| \phi \bar{u}_x \bar{v}_x  \right|     \right)  \,\mathrm{d}x     \\[2mm]
   & \quad   + \int_{\mathbb{R}} \left| \bar{u}_t - \mu \left( \frac{\bar{u}_x}{\bar{v}} \right)_x \right| \,\mathrm{d}x      \\[2mm]
   & \leq \| \bar{u}_t \|_{L^1}   + C \| \bar{u}_{xx} \|_{L^1}   + C \| \bar{u}_x \| \left( \| \phi_x \|  +  \| \bar{v}_x \| \right)      \\[2mm]
   & \leq  C \delta \left[ \left( 1+t \right)^{-1}  +  (1+t)^{-\frac{3}{4}} \| \phi_x \| \right] .      %\label{FL1}
\end{align*}
%\end{equation}
Hence, using the Sobolev inequality \eqref{lwuqiong} and the Young inequality gives
%\begin{equation}\label{jiF}
\begin{align} %[b]
  \int_{\mathbb{R}} F \psi \,\mathrm{d}x
  & \leq  \| \psi \|_{L^{\infty}} \| F \|_{L^1}
    \leq  C \| \psi \|^{\frac{1}{2}} \| \psi_x \|^{\frac{1}{2}} \cdot \delta \left[ \left( 1+t \right)^{-1}  +  (1+t)^{-\frac{3}{4}} \| \phi_x \| \right]    \nonumber \\[2mm]
  & \leq  C \delta \| \psi_x \|^2  +  C \delta (1+t)^{-\frac{4}{3}}  +  C \delta (1+t)^{-1} \| \phi_x \|^{\frac{4}{3}}      \nonumber \\[3mm]
  &  \leq C \delta \| \left( \psi_x, \phi_x \right)  \|^2   +  C \delta (1+t)^{-\frac{4}{3}}.        \label{jiF}
\end{align}
%\end{equation}
Similar to the treatment of term $ Q $ for the Navier-Stokes equations (cf. \cite{huangfm2010}), we can get
\begin{equation}\label{jiQ}
	\left| \int_{\mathbb{R}} Q \,\mathrm{d}x  \right|  \leq    C \delta (1+t)^{-\frac{5}{4}}    + C (\eta + \varepsilon_0) \| (\psi_x,\zeta_x) \|^2    + C \delta \int_{\mathbb{R}} (\phi^2 + \zeta^2) \omega^2  \,\mathrm{d}x.
\end{equation}
Integrating $\eqref{shangliudui1}$ with respect to $x$, applying $\eqref{jiF}$ and $ \eqref{jiQ} $, then choosing $\eta$, $\varepsilon_0$ and $ \delta $ suitably small, we obtain
%\begin{equation}\label{jiben1}
\begin{align}   %[b]
&\quad \frac{\mathrm{d}}{\mathrm{d}t} \int_{\mathbb{R}} \left( \frac{1}{2} \psi^{2} +  R \bar{\theta} \Phi\left(\frac{v}{\bar v}\right) + \frac{R}{\gamma-1} \bar{\theta} \Phi\left(\frac{\theta}{\bar{\theta}}\right) \right) \mathrm{d}x +C\left\|\left(\psi_{x}, \zeta_{x}\right)\right\|^{2}       \nonumber \\[2mm]
&\leq   C \delta (1+t)^{-\frac{5}{4}}      + C \delta  \int_{\mathbb{R}}  \left(\phi^{2} + \zeta^{2}\right) \omega^2   \mathrm{d} x     +  C \delta \| \phi_x  \|^2        \nonumber \\[2mm]
& \quad + \int_{\mathbb{R}}  v(E+\psi b+ \bar{u} b)^2\frac{\zeta}{\theta}\,{\rm{d}}x      -\int_{\mathbb{R}} v(E+\psi b+ \bar{u} b)\psi b\,{\rm{d}}x  .    \label{jiben1}
\end{align}
%\end{equation}

Next we try to use the structure of Maxwell equations to produce the composite good term: $\int_{\mathbb{R}} v(E+\psi b+ \bar{u} b)^2\,{\rm{d}}x$.

Multiplying $\eqref{raodong}_{4}$ by $vE$ and
$\eqref{raodong}_{5}$ by $vb$ respectively, summing them up and integrating the resulting equation with respect to $x$ over $\mathbb{R}$ gives
\begin{equation}\label{tuidao1}
\frac{\rm d}{{\rm d}t} \int_{\mathbb{R}} \frac{1}{2}(\varepsilon vE^2+vb^2) \,{\rm{d}}x     + \int_{\mathbb{R}}v(E+\psi b+\bar ub)E \,{\rm{d}}x  = 0.
\end{equation}
Moreover, multiplying $\eqref{raodong}_{4}$ by $v\bar{u} b$, integrating it with respect to $ x $ and applying $ \eqref{raodong}_{5} $, after some elementary computations, we have
\begin{align}\label{tuidao2}
& \quad \frac{\rm d}{{\rm d}t}\int_{\mathbb{R}}\varepsilon v \bar{u}Eb \,{\rm d}x    + \int_{\mathbb{R}}v(E+\psi b+\bar{u} b)\bar{u} b \,{\rm{d}}x        \nonumber \\[2mm]
& =     -\frac{1}{2}\int_{\mathbb{R}} \bar u_x(\varepsilon E^2+b^2)\,{\rm{d}}x      + \int_\mathbb{R} \varepsilon (v\bar{u}_t-u\bar{u}_x) E b \,{\rm d}x.
\end{align}
By utilizing the decay rate of viscous contact wave as in $\eqref{shuaijianlv1}$ and the Cauchy inequality, we can easily deduce that
%\begin{equation}\label{remaining1}
\begin{align} %[b]
 &\quad -\frac{1}{2}\int_{\mathbb{R}} \bar u_x(\varepsilon E^2+b^2)\,{\rm{d}}x      + \int_\mathbb{R} \varepsilon (v\bar{u}_t-u\bar{u}_x) E b \,{\rm d}x    \nonumber \\[2mm]
 &\leq C\delta \int_{\mathbb{R}} \left( \varepsilon E^2 + b^2 \right)\omega^2 \,\mathrm{d}x   + C \delta \varepsilon \int_{\mathbb{R}}  (E^2 + b^2)\omega^2  \,\mathrm{d}x       \nonumber \\[2mm]
 &\leq C \delta^{\frac{5}{6}}  \int_{\mathbb{R}}  (E^2 + b^2)\omega^2  \,\mathrm{d}x,      \label{remaining1}
\end{align}
%\end{equation}
where in the last inequality we have chosen $ \varepsilon \delta^{\frac{1}{6}} \leq 1$.
%Hence $\eqref{Exbx4}$ gives that
And the following estimate holds by using $|u|\leq \frac{3}{2}|u_-|$ in \eqref{contactushangjie}:
%\begin{equation}\label{daiquanE}
\begin{align} %[b]
  \int_{\mathbb{R}} E^2 \omega^2 \,\mathrm{d}x
  & = \int_{\mathbb{R}} (E + ub - ub)^2 \omega^2 \,\mathrm{d}x
    \leq 2 \int_{\mathbb{R}} \left[ \left( E + ub \right)^2   + \left( ub \right)^2  \right] \omega^2  \,\mathrm{d}x     \nonumber \\[2mm]
  & \leq 2 \int_{\mathbb{R}} \left( E + \psi b + \bar{u} b  \right)^2  \,\mathrm{d}x    + 5 u_-^2 \int_{\mathbb{R}} b^2 \omega^2 \,\mathrm{d}x.       \label{daiquanE}
\end{align}
%\end{equation}
Putting \eqref{daiquanE} into \eqref{remaining1} and taking $ C \delta^{\frac{1}{6} } u_-^2 \leq 1 $ can deduce that
%\begin{equation}\label{remaining3}
\begin{align} %[b]
 & \quad -\frac{1}{2}\int_{\mathbb{R}} \bar u_x(\varepsilon E^2+b^2)\,{\rm{d}}x      + \int_\mathbb{R} \varepsilon (v\bar{u}_t-u\bar{u}_x) E b \,{\rm d}x      \nonumber \\[2mm]
 & \leq C \delta^{\frac{5}{6}} \int_{\mathbb{R}} \left( E + \psi b + \bar{u} b  \right)^2  \,\mathrm{d}x     +  C \delta^{\frac{2}{3}}\int_{\mathbb{R}} b^2 \omega^2 \,\mathrm{d}x.      \label{remaining3}
\end{align}
%\end{equation}
Combining $ \eqref{jiben1} $--$ \eqref{tuidao2} $ and \eqref{remaining3}, then taking $ \delta $ suitably small, together with \eqref{raodongwuqiongguji}, we obtain that
%\begin{equation}\label{jibenneng}
\begin{align} %[b]
& \frac{\mathrm{d}}{\mathrm{d}t} \int_{\mathbb{R}}  \left(  \frac{1}{2} \psi^{2} +  R \bar{\theta} \Phi\left(\frac{v}{\bar v}\right) + \frac{R}{\gamma-1} \bar{\theta} \Phi\left(\frac{\theta}{\bar{\theta}}  \right)         + \frac{1}{2}\varepsilon vE^2       + \frac{1}{2} vb^2      + \varepsilon v \bar{u}Eb  \right) \mathrm{d}x        \nonumber \\[2mm]
& + C\left\|\left(\psi_{x}, \zeta_{x}, E+\psi b+ \bar{u} b \right)\right\|^{2}    \leq   C \delta (1+t)^{-\frac{5}{4}}        +  C \delta \| \phi_x  \|^2       \nonumber \\[2mm]
&\qquad\qquad\qquad\qquad\qquad\qquad\qquad   + C \delta^{\frac{2}{3}} \int_{\mathbb{R}}  \left(\phi^{2} + \zeta^{2} + b^2 \right) \omega^2   \mathrm{d} x.     \label{jibenneng}
\end{align}
%\end{equation}
We observe from $ \Phi(1)=\Phi'(1)=0 $ and $ \Phi''(s) = {1}/{s^2} $ that
%\begin{equation}\label{etajieguo}
\begin{align}   %[b]
  \frac{1}{2} \psi^{2} +  R \bar{\theta} \Phi\left(\frac{v}{\bar v}\right) + \frac{R}{\gamma-1} \bar{\theta} \Phi\left(\frac{\theta}{\bar{\theta}}\right)
  & =  \frac{1}{2} \psi^{2} +  R \bar{\theta} \frac{\phi^2}{2 \bar{v}^2 \xi^2_1 }        + \frac{R}{\gamma-1} \cdot \frac{\zeta^2}{2 \bar{\theta} \xi^2_2}         \nonumber \\[2mm]
  & \geq  C \left( \phi^2 + \psi^2 + \zeta^2 \right),     \label{etajieguo}
\end{align}
%\end{equation}
where $ \xi_1 $ takes value between $1$ and $ {v}/{\bar{v}} $; $ \xi_2 $ takes value between $1$ and $ {\theta}/{\bar{\theta}} $.
Recall $\left| \bar{u} \right| \leq  \frac{5}{4} \left| u_- \right|  $ in \eqref{contactushangjie}. Thus we can claim the following estimate by using the Cauchy inequality:
\begin{equation}\label{danjifen1}
  \left| \varepsilon v \bar{u} E b  \right|    \leq  v \left| \frac{5}{4} \varepsilon u_- E \cdot b  \right|       \leq  v \left(  \frac{25}{16} \varepsilon ^2 u_-^2 E^2     +  \frac{1}{4}  b^2  \right)       \leq   \frac{25}{64}  \varepsilon v E^2     +    \frac{1}{4} v b^2 ,
\end{equation}
where in the last inequality we have taken $ \varepsilon u_-^2 \leq \frac{1}{4} $.
Hence, by integrating \eqref{jibenneng} with respect to $t$, together with \eqref{etajieguo} and \eqref{danjifen1}, we reach \eqref{jibennengliang}. The proof of Lemma \ref{dijieguji} is completed.
\end{proof}

The following lemma is concerning the time-space integrability property of $\phi^{2}$, $\zeta^{2}$ and $b^2$ with an attached $ \omega^2 $ weight.
\begin{lemma}\label{quanguji}
Suppose that all the conditions in Theorem \ref{thm1} hold. For all $0 < t < T $, there exists a constant $\bar{C}$ depending only on $v_\pm$ and $|u_-|$ such that if $ 0< \varepsilon <\bar{C} $, then
\begin{equation}\label{daiquanguji}
\int_0^t\int_{\mathbb{R}}  \left(\phi^{2} + \psi^2 + \zeta^{2} + b^2 \right)  \omega^2  \,\mathrm{d}x \mathrm{d}\tau
   \leq C + C \int_0^t \| \left( \phi_x, \psi_x, \zeta_x, b_x, E + \psi b + \bar{u} b  \right)  \|^2  \,\mathrm{d}\tau .
\end{equation}
Here $\omega$ is the function defined in \eqref{omegarehe} with $ \alpha = {\hat{c}}/{4} $.

\end{lemma}

\begin{proof}[Proof]
The proof of estimate \eqref{daiquanguji} consists of the following three parts:
\begin{equation}\label{daiquanguji1}
\int_0^t\int_{\mathbb{R}}  b^2 \omega^2  \,\mathrm{d}x \mathrm{d}\tau    \leq C + C \int_0^t \| \left( \phi_x, \psi_x, b_x, E + \psi b + \bar{u} b  \right)  \|^2  \,\mathrm{d}\tau ,
\end{equation}
%\begin{equation}\label{daiquanguji2}
\begin{align} %[b]
  \int_0^t\int_{\mathbb{R}} \left[ (R \zeta - p_+ \phi)^2  + \psi^2 \right] \omega^2 \,\mathrm{d}x \mathrm{d}\tau
& \leq   C     + C \delta \int_0^t\int_{\mathbb{R}} (\phi^2 + \zeta^2) \omega^2  \,\mathrm{d}x \mathrm{d}\tau     \nonumber \\[2mm]
&   + C \int_0^t  \| \left( \phi_x, \psi_x, \zeta_x, E + \psi b + \bar{u} b  \right)  \|^2  \,\mathrm{d}\tau   \label{daiquanguji2}
\end{align}
%\end{equation}
and
%\begin{equation}\label{daiquanguji3}
\begin{align} %[b]
  \int_0^t\int_{\mathbb{R}}   \left[ R \zeta + (\gamma-1)p_+ \phi \right]^2 \omega^2  \,\mathrm{d}x \mathrm{d}\tau
& \leq     C    +  C (\delta + \eta) \int_0^t\int_{\mathbb{R}} (\phi^2 + \zeta^2) \omega^2  \,\mathrm{d}x \mathrm{d}\tau     \nonumber \\[2mm]
&        + C \int_0^t  \| \left( \phi_x, \psi_x, \zeta_x, E + \psi b + \bar{u} b  \right)  \|^2  \,\mathrm{d}\tau .      \label{daiquanguji3}
\end{align}
%\end{equation}
In fact, adding \eqref{daiquanguji2} to \eqref{daiquanguji3} and taking first $\eta$ then $\delta$ suitably small, and then combining the resulting inequality with \eqref{daiquanguji1} thus implies \eqref{daiquanguji} easily.

%we transform $\eqref{raodong}_5$ as
We first prove \eqref{daiquanguji1}. Due to $ v_t = u_x $, we can rewrite $\eqref{raodong}_5$ as
\begin{equation}\label{faxian}
  (vb)_t  -  (E + ub)_x  =0.
\end{equation}
Taking $ h = b $ in Lemma \ref{heatkernel} and using \eqref{faxian} gives
%\begin{equation}\label{neiji}
  \begin{align} %[b]
  	\left< b_t, b g^2 \right>
    & = \int_{\mathbb{R}} v b_t  \frac{b g^2}{v} \,\mathrm{d}x      = \int_{\mathbb{R}} (vb)_t  \frac{b g^2}{v} \,\mathrm{d}x    - \int_{\mathbb{R}} v_t \frac{b^2 g^2}{v} \,\mathrm{d}x     \nonumber \\[2mm]
    & = \int_{\mathbb{R}}  (E + ub)_x   \frac{b g^2}{v} \,\mathrm{d}x     - \int_{\mathbb{R}} \phi_t \frac{b^2 g^2}{v} \,\mathrm{d}x      - \int_{\mathbb{R}} \bar{v}_t  \frac{b^2 g^2}{v} \,\mathrm{d}x      \nonumber \\[2mm]
    & =    - \frac{\mathrm{d}}{\mathrm{d}t} \int_{\mathbb{R}} \phi \frac{b^2 g^2}{v} \,\mathrm{d}x      + \int_{\mathbb{R}} \phi \left( \frac{b^2}{v} \right)_t g^2  \,\mathrm{d}x     + \int_{\mathbb{R}} 2 \phi \frac{b^2}{v} g g_t \,\mathrm{d}x       \nonumber \\[2mm]
    &  \quad\, + \int_{\mathbb{R}}  (E + ub)_x   \frac{bg^2}{v} \,\mathrm{d}x    - \int_{\mathbb{R}} \bar{u}_x  \frac{b^2 g^2}{v} \,\mathrm{d}x       \nonumber \\[2mm]
    & =: - \frac{\mathrm{d}}{\mathrm{d}t} \int_{\mathbb{R}} \phi \frac{b^2 g^2}{v} \,\mathrm{d}x       + \sum_{i=1}^4 N_i,      \label{neiji}
  \end{align}
%\end{equation}
where $N_i (1 \leq i \leq 4)$ denote the corresponding terms of the last equation of \eqref{neiji}. We use \eqref{faxian} again to get
%\begin{equation}\label{n0}
\begin{align} %[b]
  N_1   & = \int_{\mathbb{R}} \phi \left( \frac{b^2}{v} \right)_t g^2  \,\mathrm{d}x     = \int_{\mathbb{R}} \phi\frac{2bb_t }{v} g^2 \,\mathrm{d}x   - \int_{\mathbb{R}} \phi b^2 \frac{v_t}{v^2} g^2  \,\mathrm{d}x     \nonumber \\[2mm]
 & =  2 \int_{\mathbb{R}} \frac{\phi b (v b)_t }{v^2} g^2 \,\mathrm{d}x    - 3 \int_{\mathbb{R}} \phi b^2 \frac{v_t}{v^2} g^2 \,\mathrm{d}x     \nonumber \\[2mm]
 & =  2 \int_{\mathbb{R}}  (E + ub)_x  \frac{\phi b g^2}{v^2}  \,\mathrm{d}x      - 3 \int_{\mathbb{R}} \phi b^2 \frac{u_x}{v^2} g^2 \,\mathrm{d}x       = : N_{1,1}   +  N_{1,2}.    \label{n0}
\end{align}
%\end{equation}
Recall $\left| \bar{v}_x \right|  \leq C \delta \omega $ due to  $ \alpha = {\hat{c}}/{4} $. By applying integration by parts, the Sobolev inequality \eqref{lwuqiong}, the Cauchy inequality and \eqref{raodongwuqiongguji}, one has
%\begin{equation}\label{n11}
\begin{align} %[b]
  N_{1,1}
  & =  - 2 \int_{\mathbb{R}}  (E + ub) \cdot \left( \frac{\phi b g^2}{v^2} \right)_x  \,\mathrm{d}x      \nonumber \\[2mm]
  & \leq   C \int_{\mathbb{R}} \left| E + ub \right| \cdot  \left( \left| \phi_x b \right|  +  \left| \phi b_x \right|  + \left| \phi b \omega  \right|   +  \left|  \phi b \phi_x \right|   +  \left| \phi b \bar{v}_x \right|  \right)  \,\mathrm{d}x      \nonumber \\[2mm]
  & \leq C \int_{\mathbb{R}} \left[  \left( E + \psi b + \bar{u} b  \right)^2    +  \phi_x^2  + b_x^2  + \varepsilon_0^2 b^2 \omega^2    \right]  \,\mathrm{d}x     \label{n11}
\end{align}
%\end{equation}
and
\begin{align}
  N_{1,2} & = - 3 \int_{\mathbb{R}} \phi b^2 \frac{u_x}{v^2} g^2 \,\mathrm{d}x       \leq  C \int_{\mathbb{R}} \left( \left| \phi b^2 \psi_x \right|        + \left| \phi b^2 \bar{u}_x \right|   \right)   \,\mathrm{d}x      \nonumber \\[2mm]
  & \leq  C \int_{\mathbb{R}}  \left( \psi_x^2     +  \bar{u}_x^2     + \phi^2 b^4   \right)     \,\mathrm{d}x      \leq  C  \| \psi_x \|^2      + C \delta (1+t)^{-\frac{3}{2}}     + C \| b \|_{L^\infty}^4 \| \phi \|^2     \nonumber \\[2mm]
  & \leq  C  \| (\psi_x,  b_x) \|^2     + C (1+t)^{-\frac{3}{2}} .     \label{n12}
\end{align}
Putting \eqref{n11} and \eqref{n12} into \eqref{n0} gives
\begin{equation}\label{n1}
  N_1  \leq  C (1+t)^{-\frac{3}{2}}   +  C \|\left( \phi_x, \psi_x, b_x, E + \psi b + \bar{u} b \right)\|^2    + C \varepsilon_0^2 \int_{\mathbb{R}} b^2 \omega^2 \,\mathrm{d}x .
\end{equation}
Similar to the calculation of \eqref{n11} and noting that $\left| \bar{u}_x \right| \leq C \delta \omega^2 $ since $ \alpha = {\hat{c}}/{4} $, we obtain
%\begin{equation}\label{n3n4}
\begin{align} %[b]
  N_3  +  N_4
  & = - \int_{\mathbb{R}}  (E + ub) \left( \frac{bg^2}{v} \right)_x  \,\mathrm{d}x       - \int_{\mathbb{R}} \bar{u}_x  \frac{b^2 g^2}{v} \,\mathrm{d}x     \nonumber \\[2mm]
  & \leq C \int_{\mathbb{R}} \left| E + ub \right| \cdot \left( \left| b_x \right|   + \left| b \omega \right|   +  \left| b \phi_x \right|       + \left| b \bar{v}_x \right|  \right)    \,\mathrm{d}x          + C \delta \int_{\mathbb{R}} b^2 \omega^2 \,\mathrm{d}x     \nonumber \\[2mm]
  & \leq  C \| \left(  \phi_x, b_x, E + \psi b + \bar{u} b  \right) \|^2      + C (\eta + \delta) \int_{\mathbb{R}} b^2 \omega^2 \,\mathrm{d}x.      \label{n3n4}
\end{align}
%\end{equation}
By using $ 4 \alpha g_t = \omega_x $ in $\eqref{gtguanxi}$, the Sobolev inequality \eqref{lwuqiong} and \eqref{raodongwuqiongguji}, we can get
\begin{align}
  N_2
  &  =      \int_{\mathbb{R}} 2 \phi \frac{b^2}{v} g g_t \,\mathrm{d}x
     = \int_{\mathbb{R}}  \frac{2}{\hat{c}}  \frac{\phi b^2}{v} g \omega _x \,\mathrm{d}x
     \leq C \| \phi \|_{L^{\infty}} \| b \|_{L^{\infty}}^2 \int_{\mathbb{R}} \left| \omega_x \right|  \,\mathrm{d}x      \nonumber \\[2mm]
  &  \leq C  (1+t)^{-\frac{1}{2}}  \| \phi \|^{\frac{1}{2}}   \| \phi_x \|^{\frac{1}{2}}   \| b \| \cdot \| b_x \|   \leq C \varepsilon_0^\frac{3}{2}  \left(   (1+t)^{-\frac{1}{2}}     \| \phi_x \|^{\frac{1}{2}}     \| b_x \|   \right)        \nonumber \\[2mm]
  &  \leq C \left(   (1+t)^{-2}     + \| \phi_x \|^{2}       +  \| b_x \|^2   \right),      \label{n2}
\end{align}
where we have used the simple equality: $ \int_{\mathbb{R}}  \left| \omega_x \right| \,\mathrm{d}x = 2(1+t)^{-\frac{1}{2}}  $.

Plugging $\eqref{n1}$--$\eqref{n2}$ into \eqref{neiji} gives
%\begin{equation}\label{neijijieguo}
  \begin{align} %[b]
  	\left< b_t, b g^2 \right>
  	& \leq - \frac{\mathrm{d}}{\mathrm{d}t} \int_{\mathbb{R}} \phi \frac{b^2 g^2}{v} \,\mathrm{d}x      + C \| \left( \phi_x, \psi_x, b_x,  E + \psi b + \bar{u} b  \right) \|^2      \nonumber \\[2mm]
  	& \quad\,       + C (1+t)^{-\frac{3}{2}}     + C (\eta + \varepsilon_0^2  + \delta) \int_{\mathbb{R}} b^2 \omega^2 \,\mathrm{d}x.     \label{neijijieguo}
  \end{align}
%\end{equation}
By employing \eqref{reheguji} of Lemma \ref{heatkernel} and using \eqref{neijijieguo}, together with choosing first $ \eta $ then $ \varepsilon_0 $, $ \delta $ suitably small, we arrive at \eqref{daiquanguji1}.

Next we prove \eqref{daiquanguji2} by employing the similar arguments as in \cite{huangfm2010}. Define
$$ f(x,t) = \int_{-\infty}^x \omega^2(y,t)  \,\mathrm{d}y.$$
We have
\begin{equation*}
\| f(\cdot,t) \|_{L^{\infty}} \leq {2 \alpha}^{-\frac{1}{2}} (1+t)^{-\frac{1}{2}} , \qquad \| f_t(\cdot,t) \|_{L^{\infty}} \leq {4 \alpha^{-\frac{1}{2}}} (1+t)^{-\frac{3}{2}}.
\end{equation*}
Multiplying $\eqref{raodong}_2$ by $ (R \zeta - p_+ \phi)vf $ and integrating the resulting equation with respect to $ x $ over $ \mathbb{R} $ gives that
\begin{small}
\begin{align}
  &\qquad \frac{1}{2} \int_{\mathbb{R}} (R \zeta - p_+ \phi)^2 \omega^2 \,\mathrm{d}x        \nonumber \\[2mm]
  & = \int_{\mathbb{R}} \psi_t (R \zeta - p_+ \phi) v f \,\mathrm{d}x        - \int_{\mathbb{R}} \frac{(R \zeta - p_+ \phi)^2}{v} \phi_x f \,\mathrm{d}x    - \int_{\mathbb{R}} \frac{(R \zeta - p_+ \phi)^2}{v} \bar{v}_x f \,\mathrm{d}x      \nonumber \\[2mm]
  & \quad  - \mu \int_{\mathbb{R}} \left( \frac{\psi_x}{v} \right)_x (R \zeta - p_+ \phi) v f  \,\mathrm{d}x    + \mu \int_{\mathbb{R}} \left( \frac{\phi \bar{u}_x}{v \bar{v}} \right)_x (R \zeta - p_+ \phi) v f  \,\mathrm{d}x       \nonumber \\[2mm]
  & \quad  + \int_{\mathbb{R}} \left[ \bar{u}_t - \mu \left( \frac{\bar{u}_x}{\bar{v}} \right)_x  \right] (R \zeta - p_+ \phi) v f \,\mathrm{d}x       +  \int_{\mathbb{R}} v (E + \psi b + \bar{u} b )b \cdot(R \zeta - p_+ \phi) v f \,\mathrm{d}x          \nonumber \\[2mm]
  & = \frac{\mathrm{d}}{\mathrm{d}t} \int_{\mathbb{R}} \psi (R \zeta - p_+ \phi)vf \,\mathrm{d}x        -\int_{\mathbb{R}} \psi (R \zeta - p_+ \phi)_t v f \,\mathrm{d}x         - \int_{\mathbb{R}} \psi v_t (R \zeta - p_+ \phi) f \,\mathrm{d}x         \nonumber \\[2mm]
  & \quad - \int_{\mathbb{R}} \psi(R \zeta - p_+ \phi) v f_t \,\mathrm{d}x
  - \int_{\mathbb{R}} \frac{(R \zeta - p_+ \phi)^2}{v} \phi_x f \,\mathrm{d}x       - \int_{\mathbb{R}} \frac{(R \zeta - p_+ \phi)^2}{v} \bar{v}_x f \,\mathrm{d}x       \nonumber \\[2mm]
  & \quad  - \mu \int_{\mathbb{R}} \left( \frac{\psi_x}{v} \right)_x (R \zeta - p_+ \phi) v f  \,\mathrm{d}x    + \mu \int_{\mathbb{R}} \left( \frac{\phi \bar{u}_x}{v \bar{v}} \right)_x (R \zeta - p_+ \phi) v f  \,\mathrm{d}x        \nonumber \\[2mm]
  & \quad  + \int_{\mathbb{R}} \left[ \bar{u}_t - \mu \left( \frac{\bar{u}_x}{\bar{v}} \right)_x  \right] (R \zeta - p_+ \phi) v f \,\mathrm{d}x        +  \int_{\mathbb{R}}  (E + \psi b + \bar{u} b )b \cdot(R \zeta - p_+ \phi) v^2 f \,\mathrm{d}x           \nonumber \\[2mm]
  & = : \frac{\mathrm{d}}{\mathrm{d}t} \int_{\mathbb{R}} \psi (R \zeta - p_+ \phi)vf \,\mathrm{d}x           + \sum_{i=1}^9 I_i.      \label{shizi1}
\end{align}
\end{small}
Equations $\eqref{raodong}_1$ and $\eqref{raodong}_3$ imply that
%\begin{equation}\label{bianxin}
\begin{align} %[b]
  \left( \frac{R}{\gamma-1}\zeta + p_+ \phi \right) _t
  &=  - \frac{R \zeta - p_+ \phi}{v} \left( \psi_x + \bar{u}_x  \right)         + \kappa \left( \frac{\bar{v} \zeta_x - \phi \bar{\theta}_x}{v \bar{v}} \right)_x          \nonumber \\[2mm]
  & \quad + \mu \frac{(\psi_x + \bar{u}_x)^2}{v}   + v(E + \psi b + \bar{u} b )^2.      \label{bianxin}
\end{align}
%\end{equation}
Then applying equation \eqref{bianxin} yields that
\begin{small}
\begin{align}
  I_1
  & =  -\int_{\mathbb{R}} \psi (R \zeta - p_+ \phi)_t v f \,\mathrm{d}x        \nonumber \\[2mm]
  & =  \int_{\mathbb{R}} \psi  \gamma p_+ \phi_t v f \,\mathrm{d}x          -\int_{\mathbb{R}} \psi \left[ R \zeta   + (\gamma-1)p_+ \phi  \right] _t v f \,\mathrm{d}x          \nonumber \\[2mm]
  & =  \int_{\mathbb{R}} \psi  \gamma p_+ \psi_x  v f \,\mathrm{d}x                -(\gamma-1) \int_{\mathbb{R}}  \psi vf \left( \frac{R}{\gamma-1}\zeta + p_+ \phi \right) _t    \,\mathrm{d}x         \nonumber \\[2mm]
  & =  -\frac{1}{2}  \int_{\mathbb{R}} \gamma p_+ v \psi^2 \omega^2 \,\mathrm{d}x          - \frac{1}{2} \int_{\mathbb{R}} \gamma p_+ v_x f \psi^2  \,\mathrm{d}x       + (\gamma-1) \int_{\mathbb{R}} \psi f (R \zeta - p_+ \phi) \psi_x  \,\mathrm{d}x      \nonumber \\[2mm]
  &  \quad     + (\gamma-1) \int_{\mathbb{R}} \psi f (R \zeta - p_+ \phi) \bar{u}_x  \,\mathrm{d}x        -(\gamma-1) \int_{\mathbb{R}} \psi vf \kappa \left( \frac{\bar{v} \zeta_x - \phi \bar{\theta}_x}{v \bar{v}} \right)_x   \,\mathrm{d}x            \nonumber \\[2mm]
  & \quad -(\gamma-1) \int_{\mathbb{R}} \psi f \mu \left( \psi_x + \bar{u}_x \right)^2  \,\mathrm{d}x           - (\gamma-1) \int_{\mathbb{R}} \psi f v^2 (E + \psi b + \bar{u} b )^2 \,\mathrm{d}x        \nonumber \\[2mm]
  & = : -\frac{1}{2}  \int_{\mathbb{R}} \gamma p_+ v \psi^2 \omega^2 \,\mathrm{d}x      +   \sum_{i=1}^{6} I_{1,i} .      \label{i1}
\end{align}
\end{small}
Comparing with the Navier-Stokes equations, we need to treat the additional terms as follows:
\begin{small}
\begin{align}
  &\qquad  I_9 + I_{1,6}   \nonumber \\[2mm]
  & =     \int_{\mathbb{R}}  (E + \psi b + \bar{u} b )b \cdot(R \zeta - p_+ \phi) v^2 f \,\mathrm{d}x     - (\gamma-1) \int_{\mathbb{R}} \psi f v^2 (E + \psi b + \bar{u} b )^2 \,\mathrm{d}x       \nonumber \\[2mm]
  & \leq  C \int_{\mathbb{R}} \left| E + \psi b + \bar{u} b  \right|\cdot \left| b \right| \cdot \left( \left| \zeta \right|  + \left| \phi \right|  \right) (1+t)^{-\frac{1}{2}}  \,\mathrm{d}x    +  C \int_{\mathbb{R}} (E + \psi b + \bar{u} b )^2 \,\mathrm{d}x      \nonumber \\[2mm]
  & \leq C \| E + \psi b + \bar{u} b  \|^2     + C \int_{\mathbb{R}} b^2 \left( \zeta^2   +   \phi^2 \right) (1+t)^{-1}  \,\mathrm{d}x      \nonumber \\[3mm]
  & \leq C \| E + \psi b + \bar{u} b  \|^2     + C (1+t)^{-1}  \| b \|^2 \left( \| \zeta \|^2_{L^{\infty}}   + \| \phi \|^2_{L^{\infty}} \right)       \nonumber \\[4mm]
  & \leq C \| E + \psi b + \bar{u} b  \|^2     + C (1+t)^{-1}  \| b \|^2 \left( \| \zeta \| \| \zeta_x \| + \| \phi \| \| \phi_x \|   \right)          \nonumber \\[4mm]
  & \leq C \| E + \psi b + \bar{u} b  \|^2     + C \left( (1+t)^{-2}      + \| \zeta_x \|^2     +\| \phi_x \|^2   \right).       \label{duochu1}
\end{align}
\end{small}
Performing the similar computations as in \cite{huangfm2010} to the remaining terms of \eqref{shizi1} and \eqref{i1},
then integrating the resulting inequalities and \eqref{duochu1} with respect to $ t $ leads to \eqref{daiquanguji2}.

At last, we prove \eqref{daiquanguji3}. Taking $ h = R \zeta + (\gamma -1 )p_+ \phi $ in Lemma \ref{heatkernel} and using \eqref{bianxin} deduces that
%\begin{equation}\label{neiji2}
\begin{align} %[b]
   \left< h_t, h g^2 \right>
   & = \int_{\mathbb{R}} (\gamma -1) \left( \frac{R}{\gamma - 1}\zeta + p_+ \phi \right)_t \cdot h g^2  \,\mathrm{d}x    \nonumber \\[2mm]
   & = -\int_{\mathbb{R}} (\gamma -1 )\frac{R \zeta - p_+ \phi}{v} \psi_x h g^2 \,\mathrm{d}x       -\int_{\mathbb{R}} (\gamma -1 )\frac{R \zeta - p_+ \phi}{v} \bar{u}_x h g^2 \,\mathrm{d}x     \nonumber \\[2mm]
   & \quad     + \kappa \int_{\mathbb{R}} (\gamma -1)\left( \frac{\bar{v} \zeta_x - \phi \bar{\theta}_x}{v \bar{v}} \right)_x h g^2  \,\mathrm{d}x        + \int_{\mathbb{R}} (\gamma-1) \mu \frac{\left( \psi_x + \bar{u}_x \right)^2 }{v} hg^2 \,\mathrm{d}x        \nonumber \\[2mm]
   & \quad     + \int_{\mathbb{R}} (\gamma-1)v \left( E + \psi b + \bar{u} b  \right)^2 hg^2  \,\mathrm{d}x    =: \sum_{i=1}^{5} J_i.     \label{neiji2}
\end{align}
%\end{equation}
Noticing that $R \zeta - p_+ \phi = h - \gamma p_+ \phi $ and using $ \psi_x = \phi_t$, we can compute $ J_1 $ as follows:
%\begin{equation}\label{j1}
\begin{align}    %[b]
  -\frac{1}{\gamma-1} J_1
  & =  \int_{\mathbb{R}} \frac{R \zeta - p_+ \phi}{v} \phi_t h g^2  \,\mathrm{d}x     = \int_{\mathbb{R}} \frac{ (h - \gamma p_+ \phi) h \phi_t g^2   }{v} \,\mathrm{d}x    \nonumber \\[2mm]
  & =  \int_{\mathbb{R}} \frac{h^2 \phi_t g^2}{v} \,\mathrm{d}x   -  \int_{\mathbb{R}} \frac{\gamma p_+ \frac{1}{2} \left( \phi^2 \right)_t  h g^2  }{v}  \,\mathrm{d}x      \nonumber \\[2mm]
  & =   \frac{\mathrm{d}}{\mathrm{d}t} \int_{\mathbb{R}} \frac{h \phi g^2 \left( h - \frac{1}{2} \gamma p_+ \phi \right) }{v}   \,\mathrm{d}x        -\int_{\mathbb{R}} \frac{h_t \phi g^2  \left( 2h - \frac{1}{2} \gamma p_+ \phi \right)  }{v} \,\mathrm{d}x        \nonumber \\[2mm]
  & \quad    -\frac{1}{2\alpha} \int_{\mathbb{R}} \frac{h \phi g \omega_x \left( h - \frac{1}{2} \gamma p_+ \phi \right) }{v} \,\mathrm{d}x        + \int_{\mathbb{R}} \frac{h \phi g^2 u_x \left( h - \frac{1}{2} \gamma p_+ \phi \right) }{v^2} \,\mathrm{d}x       \nonumber \\[2mm]
  & =:  \frac{\mathrm{d}}{\mathrm{d}t} \int_{\mathbb{R}} \frac{h \phi g^2 \left( h - \frac{1}{2} \gamma p_+ \phi \right) }{v}   \,\mathrm{d}x       + \sum_{i=1}^{3} J_{1,i}.    \label{j1}
\end{align}
%\end{equation}
Using \eqref{bianxin} again gives
%\begin{equation}\label{j11}
\begin{align}   %[b]
  J_{1,1}
   & =    -\int_{\mathbb{R}} (\gamma -1) \left( \frac{R}{\gamma - 1}\zeta + p_+ \phi \right)_t  \frac{ \phi g^2  \left( 2h - \frac{1}{2} \gamma p_+ \phi \right)  }{v} \,\mathrm{d}x        \nonumber \\[2mm]
   & =     \int_{\mathbb{R}} (\gamma -1)  (R \zeta - p_+ \phi) \left( \psi_x + \bar{u}_x  \right)           \frac{ \phi g^2  \left( 2h - \frac{1}{2} \gamma p_+ \phi \right)  }{v^2} \,\mathrm{d}x        \nonumber \\[2mm]
   &\quad   -  \int_{\mathbb{R}} (\gamma -1)   \kappa \left( \frac{\bar{v} \zeta_x - \phi \bar{\theta}_x}{v \bar{v}} \right)_x      \frac{ \phi g^2  \left( 2h - \frac{1}{2} \gamma p_+ \phi \right)  }{v} \,\mathrm{d}x       \nonumber \\[2mm]
   &\quad   -  \int_{\mathbb{R}} (\gamma -1)  \mu  (\psi_x + \bar{u}_x)^2  \frac{ \phi g^2  \left( 2h - \frac{1}{2} \gamma p_+ \phi \right)  }{v^2} \,\mathrm{d}x       \nonumber \\[2mm]
   &\quad   -  \int_{\mathbb{R}} (\gamma -1)  (E + \psi b + \bar{u} b )^2    \phi g^2  \left( 2h - \frac{1}{2} \gamma p_+ \phi \right)  \,\mathrm{d}x       =: \sum_{i=1}^{4} J_{1,1,i}.     \label{j11}
\end{align}
%\end{equation}
Compared with the Navier-Stokes equations, the additional terms $ J_5 $ and $J_{1,1,4}$ can be processed as follows:
%\begin{equation}\label{duochu2}
\begin{align} %[b]
  J_5 + J_{1,1,4}
  & = \int_{\mathbb{R}} (\gamma-1)v \left( E + \psi b + \bar{u} b  \right)^2 hg^2  \,\mathrm{d}x       \nonumber \\[2mm]
  & \quad -  \int_{\mathbb{R}} (\gamma -1)  (E + \psi b + \bar{u} b )^2    \phi g^2  \left( 2h - \frac{1}{2} \gamma p_+ \phi \right)  \,\mathrm{d}x     \nonumber \\[2mm]
  & \leq C \| E + \psi b + \bar{u} b  \|^2.   \label{duochu2}
\end{align}
%\end{equation}
By using the similar way as in \cite{huangfm2010} to estimate the remaining terms of $\eqref{neiji2}$--$\eqref{j11}$,  then integrating the resulting inequalities and \eqref{duochu2} with respect to $ t $, together with Lemma \ref{heatkernel}, we reach \eqref{daiquanguji3}. Therefore, the proof of Lemma \ref{quanguji} is completed.
\end{proof}

%XXXXXXXXXXXXXXXXXXXXXXXXXXXXXXXXXXXXXXXXXXXXXXX
%XXXXXXXXXXXXXXXXXXXXXXXXXXXXXXXXXXXXXXXXXXXXXXX

Now we combine \eqref{jibennengliang} and \eqref{daiquanguji} then choose $ \delta $ suitably small to obtain
%\begin{equation}\label{jibennenglianggai}
  \begin{align} %[b]
 &\qquad  \| (\phi, \psi, \zeta, \sqrt{\varepsilon}E, b) \|^2         + \int_0^t   \| ( \psi_x, \zeta_x, E+\psi b+\bar{u}b ) \|^2  \,\mathrm{d}\tau    \nonumber \\[2mm]
 & \leq C\| (\phi_0, \psi_0, \zeta_0, E_0, b_0) \|^2 + C\delta^{\frac{2}{3}}    + C\delta^{\frac{2}{3}} \int_0^t \| \left( \phi_x, b_x \right)  \|^2  \,\mathrm{d}\tau     \label{jibennenglianggai}
  \end{align}
%\end{equation}
and
\begin{equation}\label{daiquangujigai}
\int_0^t\int_{\mathbb{R}}  \left(\phi^{2}  + \psi^2  + \zeta^{2} + b^2 \right)  \omega^2  \,\mathrm{d}x \mathrm{d}\tau
\leq C + C \int_0^t \| \left( \phi_x, b_x \right)  \|^2  \,\mathrm{d}\tau .
\end{equation}
In the following lemma, we control the energy $\|(\sqrt{\varepsilon}E_x,b_x)\|^2$.
\begin{lemma} \label{exbxyijiedao}
Suppose that all the conditions in Theorem \ref{thm1} hold. For all $0 < t < T $, there exists a constant $\bar{C}$ depending only on $v_\pm$ and $|u_-|$ such that if $ 0< \varepsilon <\bar{C} $, then
%\begin{equation}\label{exbx}
\begin{align} %[b]
&\qquad \|(\sqrt{\varepsilon}E_x,b_x)\|^2+\int_0^t\|(E_x,b_x)\|^2{\rm{d}}\tau     \nonumber \\[2mm]
& \leq       C \|(\phi_0, \psi_0, \zeta_0)\|^2       + C \|( E_0, b_0)\|_{H^1}^2         +  C\delta^\frac{2}{3}       + C \left( \delta^{\frac{2}{3}}  +  \varepsilon _0 \right) \int_0^t \| \phi_x \|^2  \,\mathrm{d}\tau .    \label{exbx}
\end{align}
%\end{equation}
\end{lemma}
\begin{proof}[Proof]
Firstly, taking the derivative of $ \eqref{raodong}_4 $ with respect to $ x $ and multiplying it by $ vE_x $, then integrating the resulting equality with respect to $ x $, we obtain
%\begin{equation}\label{Exbx1}
\begin{align} %[b]
&\qquad \frac{\mathrm{d}}{\mathrm{d}t} \int_{\mathbb{R}} \frac{1}{2}\varepsilon vE_x^2 \,\mathrm{d}x + \int_{\mathbb{R}} vE_x^2 \,\mathrm{d}x     \nonumber \\[2mm]
&= \int_{\mathbb{R}} b_{xx}E_x \,\mathrm{d}x  + \int_{\mathbb{R}} \varepsilon u_x E_x^2 \,\mathrm{d}x  -\int_{\mathbb{R}} \varepsilon\frac{u}{v}v_x E_x^2 \,\mathrm{d}x      \nonumber \\[2mm]
&\quad - \int_{\mathbb{R}} \frac{v_x}{v}b_xE_x \,\mathrm{d}x - \int_{\mathbb{R}} (\psi b)_x vE_x\,\mathrm{d}x - \int_{\mathbb{R}} (\bar{u}b)_x vE_x\,\mathrm{d}x.     \label{Exbx1}
\end{align}
%\end{equation}
Secondly, taking the derivative of $ \eqref{raodong}_5 $ with respect to $ x $ and multiplying $ vb_x $, then integrating the resulting equality with respect to $ x $, we get
\begin{equation}\label{Exbx2}
\frac{\mathrm{d}}{\mathrm{d}t} \int_{\mathbb{R}} \frac{1}{2}vb_x^2 \,\mathrm{d}x = -\int_{\mathbb{R}} b_{xx}E_x \,\mathrm{d}x + \int_{\mathbb{R}} u_x b_x^2 \,\mathrm{d}x -\int_{\mathbb{R}} \frac{u}{v}v_x b_x^2 \,\mathrm{d}x - \int_{\mathbb{R}} \frac{v_x}{v}b_xE_x \,\mathrm{d}x.
\end{equation}
Combining $ \eqref{Exbx1} $ and $ \eqref{Exbx2} $, we obtain
%\begin{equation}\label{Exbx3}
\begin{align} %[b]
&\qquad \frac{\mathrm{d}}{\mathrm{d}t} \int_{\mathbb{R}} \left( \frac{1}{2}\varepsilon vE_x^2 + \frac{1}{2}v b_x^2 \right)  \,\mathrm{d}x + \int_{\mathbb{R}} vE_x^2 \,\mathrm{d}x     \nonumber \\[2mm]
&= \int_{\mathbb{R}} \varepsilon u_x E_x^2 \,\mathrm{d}x  -\int_{\mathbb{R}} \varepsilon\frac{u}{v}v_x E_x^2 \,\mathrm{d}x + \int_{\mathbb{R}} u_x b_x^2 \,\mathrm{d}x - \int_{\mathbb{R}} \frac{u}{v} v_x b_x^2 \,\mathrm{d}x        \nonumber \\[2mm]
&\quad -2 \int_{\mathbb{R}} \frac{v_x}{v}b_xE_x \,\mathrm{d}x - \int_{\mathbb{R}} (\psi b)_x vE_x\,\mathrm{d}x - \int_{\mathbb{R}} (\bar{u}b)_x vE_x\,\mathrm{d}x.      \label{Exbx3}
\end{align}
%\end{equation}
Recall that $\left| \bar{u} \right| \leq  \frac{5}{4} \left| u_- \right|  $.
By using the decay rate of viscous contact wave as in $\eqref{shuaijianlv1}$, together with $\eqref{raodongwuqiongguji}$, the Sobolev inequality \eqref{lwuqiong} and the Cauchy inequality, we can derive from \eqref{Exbx3} that
%\begin{equation}\label{Exbx4}
\begin{align} %[b]
&\qquad \frac{\mathrm{d}}{\mathrm{d}t} \int_{\mathbb{R}} \left( \frac{1}{2}\varepsilon vE_x^2 + \frac{1}{2}v b_x^2 \right)  \,\mathrm{d}x   + \int_{\mathbb{R}} vE_x^2 \,\mathrm{d}x     \nonumber \\[2mm]
& \leq C \int_{\mathbb{R}} (\left| \phi_x \right| + \left| \psi_x \right|)(E_x^2 + b_x^2) \,\mathrm{d}x         + C \int_{\mathbb{R}} ( \left| \bar{v}_x \right|  +  \left| \bar{u}_x \right|  )( E_x^2 + b_x^2 ) \,\mathrm{d}x       \nonumber \\[2mm]
& \quad    +   C \int_{\mathbb{R}} (\left| \psi_x b \right|   + \left| \psi b_x \right|   + \left| \bar{u}_x b \right| ) \left| E_x \right|  \,\mathrm{d}x             +  \int_{\mathbb{R}}   \left| v \bar{u} b_x E_x  \right|  \,\mathrm{d}x        \nonumber \\[2mm]
& \leq     C( \| \phi_x \|_{L^\infty}    + \| \psi_x \|_{L^\infty} ) ( \| \sqrt{v} E_x \|^2   +  \| b_x \|^2 )      + C (\delta  + \varepsilon_0 ) \| \left( \sqrt{v}  E_x, \psi_x, b_x \right)  \|^2          \nonumber \\[2mm]
& \quad        + C \delta (1+t)^{-\frac{3}{2}}      + \frac{1}{2} \int_{\mathbb{R}} v E_x^2 \,\mathrm{d}x         + \frac{1}{2} \int_{\mathbb{R}} v \bar{u}^2 b_x^2 \,\mathrm{d}x       \nonumber \\[2mm]
& \leq     C (\delta  + \varepsilon_0 ) \| \left( \sqrt{v}  E_x, \psi_x, b_x \right)  \|^2         + C \delta (1+t)^{-\frac{3}{2}}       + \frac{1}{2} \int_{\mathbb{R}} v E_x^2 \,\mathrm{d}x      \nonumber \\[2mm]
& \quad      +   \frac{25}{32} u_-^2 \max \{ v_\pm\} \int_{\mathbb{R}} b_x^2 \,\mathrm{d}x  .    \label{Exbx4}
\end{align}
%\end{equation}
Then by choosing $ \varepsilon_0 $ and $ \delta $ suitably small, we can reach from \eqref{Exbx4}
%\begin{equation}\label{lemshi1}
\begin{align} %[b]
 &\qquad \frac{\mathrm{d}}{\mathrm{d}t} \int_{\mathbb{R}} \left( \frac{1}{2}\varepsilon vE_x^2 + \frac{1}{2}v b_x^2 \right)  \,\mathrm{d}x       + \frac{1}{4}\int_{\mathbb{R}} vE_x^2 \,\mathrm{d}x       \nonumber \\[2mm]
 & \leq    C \delta (1+t)^{-\frac{3}{2}}        + C (\delta  + \varepsilon_0 )\| \left(  \psi_x, b_x \right)  \|^2        + u_-^2 \max \{ v_\pm\} \int_{\mathbb{R}} b_x^2 \,\mathrm{d}x .      \label{lemshi1}
\end{align}
%\end{equation}

Now we only need to estimate the last term on the right-hand side of $\eqref{lemshi1}$. Multiplying $\eqref{raodong}_{4}$ by $-vb_x$, integrating the resulting equation with respect to $x$ over $\mathbb{R}$ and applying $\eqref{raodong}_{5}$, we can easily deduce
\begin{small}
\begin{align}
& \qquad -\frac{\mathrm{d}}{\mathrm{d}t}\int_{\mathbb{R}} \varepsilon v Eb_x \,\mathrm{d}x + \int_{\mathbb{R}} b_x^2 \,\mathrm{d}x      \nonumber \\[2mm]
& = -\int_{\mathbb{R}} \varepsilon u_x E b_x \,\mathrm{d}x    - \int_{\mathbb{R}} \varepsilon v E b_{tx} \,\mathrm{d}x -\int_{\mathbb{R}} \varepsilon u E_x b_x \,\mathrm{d}x       + \int_{\mathbb{R}} (E + \psi b + \bar{u} b )v b_x \,\mathrm{d}x       \nonumber \\[2mm]
& = -\int_{\mathbb{R}} \varepsilon \psi_x E b_x \,\mathrm{d}x     - \int_{\mathbb{R}} \varepsilon \bar{u}_x E b_x \,\mathrm{d}x     +  \left(\int_{\mathbb{R}} \varepsilon \frac{u}{v} E \phi_x b_x \,\mathrm{d}x             + \int_{\mathbb{R}} \varepsilon \frac{u}{v} E \bar{v}_x b_x \,\mathrm{d}x  \right.        \nonumber \\[2mm]
& \quad    \left.   + \int_{\mathbb{R}} \varepsilon \frac{1}{v} \phi_x E E_x  \,\mathrm{d}x       +  \int_{\mathbb{R}}  \varepsilon \frac{1}{v} \bar{v}_x E E_x \,\mathrm{d}x     + \int_{\mathbb{R}} \varepsilon E_x^2 \,\mathrm{d}x      \right)      + \int_{\mathbb{R}} (E + \psi b + \bar{u} b )v b_x \,\mathrm{d}x     \nonumber \\[2mm]
& \leq C \varepsilon \varepsilon_0 \int_{\mathbb{R}} \left( \left| \psi_x b_x  \right|   + \left| \phi_x b_x \right|   + \left| \phi_x E_x \right|  \right)  \,\mathrm{d}x            + C \varepsilon \delta \int_{\mathbb{R}} \left| E \omega  \right| \left( \left| b_x \right|  + \left| E_x \right|  \right)   \,\mathrm{d}x            \nonumber \\[2mm]
& \quad      + \varepsilon \int_{\mathbb{R}} E_x^2  \,\mathrm{d}x            + C \int_{\mathbb{R}} \left| (E + \psi b + \bar{u} b )b_x \right|  \,\mathrm{d}x     \nonumber \\[2mm]
& \leq C \varepsilon (\delta  +  \varepsilon_0) \| (\phi_x, \psi_x, b_x, E_x) \|^2     + C \varepsilon \delta \| E \omega \|^2         + \varepsilon \| E_x \|^2       + \frac{1}{2} \| b_x \|^2  \nonumber \\[2mm]
& \quad      + C \| E + \psi b + \bar{u} b  \|^2 ,    \label{gujiExbx6}
\end{align}
\end{small}
i.e.,
%\begin{equation}\label{lemshi2}
\begin{align} %[b]
 -\frac{\mathrm{d}}{\mathrm{d}t}\int_{\mathbb{R}} \varepsilon v Eb_x \,\mathrm{d}x     + \frac{1}{2} \| b_x \|^2
  & \leq C \varepsilon (\delta  +  \varepsilon_0) \| (\phi_x, \psi_x, b_x, E_x) \|^2     + C \varepsilon \delta \| E \omega \|^2       \nonumber \\[2mm]
  & \quad       + C \| E + \psi b + \bar{u} b  \|^2      + \varepsilon \| E_x \|^2 .    \label{lemshi2}
\end{align}
%\end{equation}
Multiplying \eqref{lemshi2} by $ 4u_-^2 \max\{v_\pm\} $ and summing it to \eqref{lemshi1}, together with $ v > \frac{1}{2} \min \left\{ v_\pm \right\} $ gives
%\begin{equation*}
\begin{align} %[b]
 &\quad \frac{\mathrm{d}}{\mathrm{d}t} \int_{\mathbb{R}} \left( \frac{1}{2}\varepsilon vE_x^2  + \frac{1}{2}v b_x^2   -4\varepsilon u_-^2 \max\{v_\pm\}   v Eb_x   \right)  \,\mathrm{d}x        \nonumber \\[2mm]
 &\quad        + \left( \frac{1}{8} \min\{v_\pm\}  -4 \varepsilon  u_-^2 \max\{v_\pm\} \right)  \| E_x \|^2         +  \left( 2u_-^2 - u_-^2 \right)  \max\{v_\pm\}  \| b_x \|^2        \nonumber \\[2mm]
 & \leq      C \delta (1+t)^{-\frac{3}{2}}        + C (\delta  + \varepsilon_0 )\| \left( \psi_x, b_x \right)  \|^2          + C \varepsilon u_-^2 (\delta + \varepsilon_0 )  \| (\phi_x, \psi_x, b_x, E_x) \|^2      \nonumber \\[2mm]
 &\quad      + C \varepsilon  u_-^2  \delta  \| E \omega \|^2     + C  \| E + \psi b + \bar{u} b  \|^2 .     \label{lemshizhongjian}
\end{align}
%\end{equation*}
By taking $ \frac{1}{8} \min\{v_\pm\}  - 4 \varepsilon  u_-^2 \max\{v_\pm\} >  \frac{1}{16} \min\{v_\pm\} $ in \eqref{lemshizhongjian},
i.e., $ \varepsilon u_-^2 < \frac{\min\{v_\pm\}}{64 \max\{v_\pm\}} $, we can obtain
%\begin{equation}\label{lemshi3}
\begin{align}  %[b]
 & \frac{\mathrm{d}}{\mathrm{d}t} \int_{\mathbb{R}} \left( \frac{1}{2}\varepsilon vE_x^2  + \frac{1}{2}v b_x^2   -4\varepsilon u_-^2 \max\{v_\pm\}   v Eb_x   \right)  \,\mathrm{d}x            +  \frac{1}{16} \min\{v_\pm\}  \| E_x \|^2               \nonumber \\[2mm]
 &  +  u_-^2 \max\{v_\pm\}  \| b_x \|^2    \leq     C \delta (1+t)^{-\frac{3}{2}}        + C (\delta + \varepsilon_0 ) \| (\phi_x, \psi_x, b_x, E_x) \|^2          \nonumber \\[2mm]
 &  \qquad\qquad\qquad\qquad\qquad       + C  \delta  \| E \omega \|^2           + C   \| E + \psi b + \bar{u} b  \|^2  .      \label{lemshi3}
\end{align}
%\end{equation}
Integrating \eqref{lemshi3} with respect to $t$, then choosing $ \varepsilon_0 $ and $ \delta $ suitably small leads to
\begin{small}
\begin{align}\label{lemshi4}
 & \quad  \int_{\mathbb{R}} \left( \frac{1}{2}\varepsilon vE_x^2  + \frac{1}{2}v b_x^2   \right)  \,\mathrm{d}x                      +  \frac{1}{32} \min\{v_\pm\}  \int_0^t \| E_x \|^2  \,\mathrm{d}\tau           +  \frac{1}{2} u_-^2 \max\{v_\pm\}  \int_0^t \| b_x \|^2 \,\mathrm{d}\tau        \nonumber  \\[2mm]
 & \leq    C \| (E_0, E_{0x}, b_{0x}) \|^2       + C \delta      + C (\delta  + \varepsilon_0 )  \int_0^t \| (\phi_x, \psi_x) \|^2 \,\mathrm{d}\tau        + C  \delta \int_0^t  \| E \omega \|^2 \,\mathrm{d}\tau    \nonumber \\[2mm]
 &  \quad         + C   \int_0^t \| E + \psi b + \bar{u} b  \|^2   \,\mathrm{d}\tau      + 4\varepsilon u_-^2 \max\{v_\pm\}  \int_{\mathbb{R}}  \left| v Eb_x \right| \,\mathrm{d}x.
\end{align}
\end{small}
Then by using the Cauchy inequality to the last term of \eqref{lemshi4}, we can get
\begin{small}
\begin{align}
 & 4\varepsilon u_-^2 \max\{v_\pm\}  \int_{\mathbb{R}}  \left| v Eb_x \right|   \,\mathrm{d}x
  \leq  C \int_{\mathbb{R}} v \left| \varepsilon u_-^2 E \right| \left| b_x \right|  \,\mathrm{d}x
   \leq  \int_{\mathbb{R}} v \left( \frac{1}{4} b_x^2  + C \varepsilon ^2 u_-^4 E^2 \right)  \,\mathrm{d}x     \nonumber \\[2mm]
 & \qquad\qquad\qquad\quad \;\;   \leq  \frac{1}{4} \int_{\mathbb{R}} vb_x^2  \,\mathrm{d}x     + C \varepsilon u_-^4 \int_{\mathbb{R}} \varepsilon E^2 \,\mathrm{d}x
   \leq  \frac{1}{4} \int_{\mathbb{R}} vb_x^2  \,\mathrm{d}x    + C u_-^2 \int_{\mathbb{R}} \varepsilon E^2 \,\mathrm{d}x,      \label{danjifen2}
\end{align}
\end{small}
where in the last inequality we have recalled that $ \varepsilon u_-^2 < \frac{\min\{v_\pm\}}{64 \max\{v_\pm\}} $.

Plugging \eqref{danjifen2} into \eqref{lemshi4}, then employing \eqref{daiquanE}, \eqref{jibennenglianggai} and \eqref{daiquangujigai}, together with taking $ \varepsilon_0 $ and $ \delta $ small enough, we can conclude \eqref{exbx}.

From the above discussion about $ \varepsilon u_-^2 $,
we can determine that
\begin{equation}\label{vepu-2}
  \varepsilon u_-^2   < \min \left\{ \frac{1}{4}, \frac{\min\{v_\pm\}}{64 \max\{v_\pm\}}  \right\}    =   \frac{\min\{v_\pm\}}{64 \max\{v_\pm\}}  .
\end{equation}
This completes the proof of Lemma \ref{exbxyijiedao}.
\end{proof}

In the following lemma, we control the energy  $\|\phi_x\|^2$.

\begin{lemma} \label{phiyijiedao}
Suppose that all the conditions in Theorem \ref{thm1} hold. For all $0 < t < T $, there exists a constant $\bar{C}$ depending only on $v_\pm$ and $|u_-|$ such that if $ 0< \varepsilon <\bar{C} $, then
\begin{equation}\label{phix}
 \|\phi_x\|^2 +\int_0^t \|\phi_{x}\|^2 {\rm{d}}\tau
  \leq    C \|(\psi_0, \zeta_0)\|^2      + C \|(\phi_0, E_0, b_0)\|_{H^1}^2     +  C\delta^\frac{2}{3} .
\end{equation}

\end{lemma}

\begin{proof}[Proof]
Due to
\begin{equation*}
 \left( \frac{u_x}{v} \right)_x =  \left( \frac{v_t}{v} \right)_x =  \left( \ln v \right)_{tx} =  \left( \frac{v_x}{v} \right)_t =  \left( \frac{\phi_x}{v} \right)_t +  \left( \frac{\bar{v}_x}{v} \right)_t
\end{equation*}
and
\begin{equation*}
\left( p - p_+ \right)_x   =  \left( \frac{R\zeta - {p_+}\phi}{v} \right)_x = -\frac{{p_+}\phi_x}{v} + \frac{R\zeta_x}{v} - (R\zeta - {p_+}\phi) \frac{\phi_x + \bar{v}_x}{v^2} ,
\end{equation*}
we can rewrite $\eqref{lagrange}_2$ as
%\begin{equation}\label{gujiphix0}
\begin{align}  %[b]
\mu \left( \frac{\phi_x}{v} \right)_t        +  \left( p_+ + \frac{R \zeta - p_+ \phi}{v} \right) \frac{\phi_x}{v}
& =  \psi_t     + \frac{R\zeta_x}{v}    -\frac{R\zeta - p_+ \phi}{v^2}\bar{v}_x            \nonumber \\[2mm]
& \quad   +  \left[ \bar{u}_t  -\mu \left( \frac{\bar{v}_x}{v} \right)_t \right]    + v(E+\psi b+\bar{u}b)b.      \label{gujiphix0}
\end{align}
%\end{equation}
Multiplying \eqref{gujiphix0} by $ \frac{\phi_x}{v} $, then integrating the resulting equality with respect to $x$ over $\mathbb{R}$, noticing that
\begin{equation*}
  \left( \frac{\phi_x}{v} \right)_t  = \frac{\psi_{xx}}{v}  - \frac{\phi_x \psi_x + \phi_x \bar{u}_x}{v^2}
\end{equation*}
due to $ \phi_t = \psi_x $, we have
%\begin{equation}\label{gujiphix1}
\begin{align} %[b]
   & \qquad \frac{\mathrm{d}}{\mathrm{d}t} \int_{\mathbb{R}} \left( \frac{1}{2} \mu \frac{\phi_x^2}{v^2}  - \frac{\psi \phi_x}{v} \right)  \,\mathrm{d}x      + \int_{\mathbb{R}} \left( p_+  + \frac{R \zeta - p_+ \phi}{v} \right) \frac{\phi_x^2}{v^2}  \,\mathrm{d}x      \nonumber \\[2mm]
   & = \int_{\mathbb{R}} \frac{\psi_x^2}{v} \,\mathrm{d}x     + \int_{\mathbb{R}} \frac{\bar{u}_x \psi \phi_x }{v^2} \,\mathrm{d}x       - \int_{\mathbb{R}} \frac{\bar{v}_x \psi \psi_x}{v^2} \,\mathrm{d}x       - \int_{\mathbb{R}} \frac{R \zeta - p_+ \phi}{v^3} \bar{v}_x \phi_x \,\mathrm{d}x     \nonumber \\[2mm]
   & \quad    + \int_{\mathbb{R}} \frac{R \zeta_x \phi_x}{v^2} \,\mathrm{d}x       + \int_{\mathbb{R}} \left[ \bar{u}_t - \mu \left( \frac{\bar{v}_x}{v} \right)_t  \right] \frac{\phi_x}{v} \,\mathrm{d}x        + \int_{\mathbb{R}} (E + \psi b + \bar{u} b )b \phi_x \,\mathrm{d}x.      \label{gujiphix1}
\end{align}
%\end{equation}
Using the Cauchy inequality and the estimates for viscous contact wave in Lemma \ref{jianduanbol2mo} to \eqref{gujiphix1} yields that
%\begin{equation}\label{gujiphix2}
\begin{align} %[b]
   & \qquad \frac{\mathrm{d}}{\mathrm{d}t} \int_{\mathbb{R}} \left( \frac{1}{2} \mu \frac{\phi_x^2}{v^2}  - \frac{\psi \phi_x}{v} \right)  \,\mathrm{d}x         + \int_{\mathbb{R}} p_+ \frac{\phi_x^2}{v^2}  \,\mathrm{d}x     \nonumber \\[2mm]
   & \leq C \delta (1+t)^{-\frac{3}{2}}      +  C(\eta + \varepsilon _0) \int_{\mathbb{R}} \frac{\phi_x^2}{v^2}  \,\mathrm{d}x          + C \delta \| (\phi, \psi, \zeta) \omega \|^2      \nonumber \\[2mm]
   & \quad      + C \| (\psi_x, \zeta_x, E + \psi b + \bar{u} b ) \|^2 .     \label{gujiphix2}
\end{align}
%\end{equation}
Lemma \ref{phiyijiedao} thus follows easily from \eqref{gujiphix2}, \eqref{jibennenglianggai}--\eqref{exbx}
by first choosing $ \eta $ suitably small then $ \varepsilon_0 $ small enough.
\end{proof}

%%%%%%%%%%%%%%%%%%%%%%%%%%%%%%%%%%%%%%%%%%%

\begin{lemma} \label{psixdanjifen}
Suppose that all the conditions in Theorem \ref{thm1} hold. For all $0 < t < T $, there exists a constant $\bar{C}$ depending only on $v_\pm$ and $|u_-|$ such that if $ 0< \varepsilon <\bar{C} $, then
\begin{equation}\label{psixzetax}
\| (\psi_x, \zeta_x) \|^2 + \int_0^t \| (\psi_{xx}, \zeta_{xx} )\|^2  \,\mathrm{d}\tau      \leq  C \|(\phi_0, \psi_0, \zeta_0, E_0, b_0)\|_{H^1}^2          + C \delta^\frac{2}{3} .
\end{equation}

\end{lemma}
\begin{proof}[Proof]
Multiplying $ \eqref{raodong}_2 $ by $-\psi_{xx}$ and integrating the resulting equality with respect to $x$, we have
%\begin{equation}\label{gujipsix1}
\begin{align}  %[b]
   & \qquad  \frac{\mathrm{d}}{\mathrm{d}t}  \int_{\mathbb{R}} \frac{1}{2} \psi_x^2 \,\mathrm{d}x     + \mu \int_{\mathbb{R}} \frac{\psi_{xx}^2}{v} \,\mathrm{d}x         \nonumber \\[2mm]
   &  =  \mu \int_{\mathbb{R}} \frac{\phi_x \psi_x \psi_{xx}}{v^2} \,\mathrm{d}x     + \mu \int_{\mathbb{R}} \frac{\bar{v}_x \psi_x \psi_{xx}}{v^2} \,\mathrm{d}x           + \int_{\mathbb{R}} \left( \frac{R \zeta - p_+ \phi}{v} \right)_x \psi_{xx}  \,\mathrm{d}x         \nonumber \\[2mm]
   &  \quad      + \mu  \int_{\mathbb{R}} \left( \frac{\phi \bar{u}_x}{v \bar{v}} \right)_x \psi_{xx}  \,\mathrm{d}x        + \int_{\mathbb{R}} \left[ \bar{u}_t  - \mu \left( \frac{\bar{u}_x}{\bar{v}} \right)_x  \right] \psi_{xx}  \,\mathrm{d}x        \nonumber \\[2mm]
   & \quad     + \int_{\mathbb{R}} v (E + \psi b + \bar{u} b ) b \psi_{xx} \,\mathrm{d}x   =: \sum_{i=1}^6 L_i.       \label{gujipsix1}
\end{align}
%\end{equation}
Using $ \| \phi_x \|_{L^{\infty}}\leq \sqrt{2}\varepsilon_0 $ in \eqref{raodongwuqiongguji} and H$\ddot{\mathrm{o}}$lder inequality yields that
\begin{equation}\label{l1}
  \left| L_1 \right|  \leq  C \| \phi_x \|_{L^{\infty}} \| \psi_x \| \| \psi_{xx} \| \leq C \varepsilon _0 \left( \| \psi_x \|^2  +  \| \psi_{xx} \|^2  \right).
\end{equation}
Moreover, utilizing the Cauchy inequality, the estimate $\eqref{shuaijianlv1}$ and Lemma \ref{jianduanbol2mo} leads to
%\begin{equation}\label{liqiyu}
\begin{align} %[b]
  \sum_{i=2}^6 \left| L_i \right|
  & \leq C (\eta + \delta) \| \psi_{xx} \|^2     + C \delta (1+t)^{-\frac{5}{2}}     + C \delta \| (\phi, \zeta)\omega \|^2          \nonumber \\[2mm]
  &\quad  + C \| (\phi_x, \psi_x, \zeta_x, E + \psi b + \bar{u} b ) \|^2.      \label{liqiyu}
\end{align}
%\end{equation}
By taking $ \eqref{l1}, \eqref{liqiyu}$ into $\eqref{gujipsix1}$,
and integrating the resulting equality over $ (0,t) $,
then using \eqref{jibennenglianggai}--\eqref{exbx} and \eqref{phix}, as well as choosing first $ \eta $ suitably small next $ \varepsilon_0  $, $ \delta $ small enough, we obtain
\begin{equation}\label{gujipsix}
\| \psi_x \|^2 + \int_0^t \| \psi_{xx} \|^2  \,\mathrm{d}\tau \leq   C \|(\psi_0, \phi_0, E_0, b_0)\|_{H^1}^2      + C \| \zeta_0 \|^2         +  C\delta^\frac{2}{3}   .
\end{equation}

On the other hand, multiplying the $\eqref{raodong}_3$ by $-\zeta_{xx}$ and integrating the resulting equality over $ \mathbb{R} $, we have
\begin{small}
\begin{align}
& \qquad \frac{\mathrm{d}}{\mathrm{d}t}\int_{\mathbb{R}} \frac{R}{2(\gamma-1)}\zeta^2_x \,\mathrm{d}x +  \kappa \int_{\mathbb{R}}  \frac{\zeta_{xx}^2}{v} \,\mathrm{d}x     \nonumber \\[2mm]
& =  \kappa \int_{\mathbb{R}} \frac{\zeta_x \zeta_{xx} (\phi_x + \bar{v}_x )}{v^2} \,\mathrm{d}x       + \int_{\mathbb{R}} p \psi_x \zeta_{xx} \,\mathrm{d}x      + \kappa \int_{\mathbb{R}} \left( \frac{\phi \bar{\theta}_x }{v \bar{v}} \right)_x \zeta_{xx}  \,\mathrm{d}x        \nonumber \\[2mm]
& \quad     + \int_{\mathbb{R}} \frac{R \zeta - p_+ \phi}{v} \bar{u}_x \zeta_{xx} \,\mathrm{d}x       - \mu \int_{\mathbb{R}} \frac{\left( \psi_x + \bar{u}_x \right)^2  }{v} \zeta_{xx} \,\mathrm{d}x       - \int_{\mathbb{R}} v(E + \psi b + \bar{u} b )^2 \zeta_{xx}  \,\mathrm{d}x.      \label{gujizetax0}
\end{align}
\end{small}
Similarly, we can get from \eqref{gujizetax0}
\begin{equation}\label{gujizetax}
\| \zeta_x \|^2 + \int_0^t \| \zeta_{xx} \|^2  \,\mathrm{d}\tau  \leq    C \|(\phi_0, \psi_0, \zeta_0, E_0, b_0)\|_{H^1}^2          + C \delta^\frac{2}{3} .
\end{equation}
Summing \eqref{gujizetax} to \eqref{gujipsix} yields \eqref{psixzetax}. The proof of Lemma \ref{psixdanjifen} is completed.
\end{proof}

%%%%%%%%%%%%%%%%%%%%%%%%%%%%%%%%%%%%%%%%%%%%%%%%%%%%%
%%%%%%%%%%%%%%%%%%%%%%%%%%%%%%%%%%%%%%%%%%%%%%%%%%%%%

In order to close the {\it a priori} assumption $ \sup_{0\leq t\leq T} \| \phi_{xx}(t) \| \leq \varepsilon _0  $, we still need to control the high-order energy $\|\phi_{xx}\|^2$.
\begin{lemma}\label{phixxguji}
Suppose that all the conditions in Theorem \ref{thm1} hold. For all $0 < t < T $, there exists a constant $\bar{C}$ depending only on $v_\pm$ and $|u_-|$ such that if $ 0< \varepsilon <\bar{C} $, then
\begin{equation}\label{phixx}
\| \phi_{xx} \|^2 +  \int_0^t \| \phi_{xx} \|^2  \,\mathrm{d}\tau    \leq   C \delta^\frac{2}{3}     + C\|\phi_{0}\|_{H^2}^2        + C \| (\psi_0, \zeta_0, E_0, b_0) \|^2_{H^1}   .
\end{equation}

\end{lemma}

\begin{proof}[Proof]

Differentiating $\eqref{gujiphix0}$ with respect to $x$ and multiplying it by $ \phi_{xx} $, then integrating the resulting equality over $\mathbb{R}$, together with $\phi_t = \psi_x $, we can get
%\begin{equation}\label{gujiphixx0}
\begin{align}  %[b]
& \quad    \frac{\mathrm{d}}{\mathrm{d}t} \int_{\mathbb{R}} \left(  \frac{1}{2}\mu \frac{\phi_{xx}^2 }{v}  - \mu \frac{(\phi_x +\bar{v}_x) \phi_x \phi_{xx}  }{v^2 }    - \psi_x \phi_{xx}   \right)  \,\mathrm{d}x       + \int_{\mathbb{R}} p_+ \frac{\phi_{xx}^2}{v} \,\mathrm{d}x       \nonumber  \\[2mm]
& =        \frac{1}{2} \mu \int_{\mathbb{R}} \frac{\psi_x + \bar{u}_x}{v^2} \phi_{xx}^2 \,\mathrm{d}x         + 2 \mu \int_{\mathbb{R}} \frac{\phi_x \phi_{xx}  \psi_{xx}}{v^2}  \,\mathrm{d}x          + \mu \int_{\mathbb{R}} \frac{\bar{v}_x   \phi_{xx} \psi_{xx} }{v^2} \,\mathrm{d}x       \nonumber  \\[2mm]
& \quad     + \mu \int_{\mathbb{R}} \frac{\bar{v}_{xx} \phi_x \psi_{xx}  }{v^2} \,\mathrm{d}x       - 2 \mu \int_{\mathbb{R}} \left( \phi_x^3  +  2 \bar{v}_x \phi_x^2  +  \bar{v}_x^2 \phi_x \right) \frac{\psi_{xx}}{v^3}  \,\mathrm{d}x          \nonumber  \\[2mm]
& \quad     + \int_{\mathbb{R}} p_+ \frac{\phi_x^2 \phi_{xx}}{v^2} \,\mathrm{d}x          + \int_{\mathbb{R}} p_+ \frac{\bar{v}_x \phi_x \phi_{xx} }{v^2} \,\mathrm{d}x        - \int_{\mathbb{R}} \left( \frac{R \zeta - p_+ \phi}{v^2} \phi_x \right)_x  \phi_{xx}  \,\mathrm{d}x    \nonumber  \\[2mm]
& \quad     +\int_{\mathbb{R}} \psi_{xx}^2  \,\mathrm{d}x          + R \int_{\mathbb{R}}  \left( \frac{\zeta_x}{v} \right)_x \phi_{xx}  \,\mathrm{d}x             - \int_{\mathbb{R}} \left( \frac{R \zeta - p_+ \phi}{v^2} \bar{v}_x \right)_x  \phi_{xx}  \,\mathrm{d}x         \nonumber  \\[2mm]
& \quad     + \int_{\mathbb{R}} \left[ \bar{u}_t - \mu \left( \frac{\bar{v}_x}{v} \right)_t  \right]_x  \phi_{xx} \,\mathrm{d}x           + \int_{\mathbb{R}} \left[ v(E + \psi b + \bar{u} b )b  \right]_x  \phi_{xx} \,\mathrm{d}x       =:   \sum_{i=1}^{13}  M_i.    \label{gujiphixx0}
\end{align}
%\end{equation}
Employing $ \| \left( \phi_x, \psi_x \right) \|_{L^{\infty}} \leq \sqrt{2} \varepsilon_0 $ in \eqref{raodongwuqiongguji} and the Cauchy inequality leads to
\begin{equation}\label{m1-12}
    \sum_{i=1}^{12}  \left| M_i \right|  \leq  C (\eta + \varepsilon _0  +  \delta) \| \phi_{xx} \|^2     + C \delta (1+t)^{-\frac{3}{2}}       + C \| \left( \phi_x, \psi_x, \zeta_x, \psi_{xx}, \zeta_{xx} \right)  \|^2
\end{equation}
and
%\begin{equation}\label{m13}
\begin{align} %[b]
  \left| M_{13}  \right|
  & \leq  \int_{\mathbb{R}} \left| \left[ v(E + \psi b + \bar{u} b )b  \right]_x  \phi_{xx} \right|  \,\mathrm{d}x        \nonumber \\[2mm]
  & \leq  C \int_{\mathbb{R}} \left| (\phi_x + \bar{v}_x) b  \right|  \cdot  \left| \phi_{xx} \right|   \,\mathrm{d}x        + C \int_{\mathbb{R}}   \left| b_x \right| \cdot \left| \phi_{xx} \right|  \,\mathrm{d}x      \nonumber \\[2mm]
  & \quad    +  C \int_{\mathbb{R}} \left|  E_x + \psi_x b + \psi b_x + \bar{u}_x b + \bar{u} b_x   \right| \cdot  \left| \phi_{xx} \right|  \,\mathrm{d}x     \nonumber \\[2mm]
  & \leq  C \eta \| \phi_{xx} \|^2        +  C \| (\phi_x, \psi_x, b_x, E_x) \|^2        + C \delta \|  b \omega \|^2.     \label{m13}
\end{align}
%\end{equation}
By taking $ \eqref{m1-12}, \eqref{m13}$ into $\eqref{gujiphixx0}$, and choosing first $ \eta $ suitably small next $ \varepsilon_0  $, $ \delta $ small enough, integrating the resulting equality over $ (0,t) $, then using \eqref{jibennenglianggai}--\eqref{exbx}, \eqref{phix} and \eqref{psixzetax},  we obtain \eqref{phixx}.

The proof of Lemma \ref{phixxguji} is completed.
\end{proof}

%%%%%%%%%%%%%%%%%%%%%%%%%%%%%%%%%%%%%%%%%
%%%%%%%%%%%%%%%%%%%%%%%%%%%%%%%%%%%%%%%%%

To get higher order regularity of the velocity and the absolute temperature, we will estimate the energy $ \| \psi_{xx} \|^2  $ and $ \| \zeta_{xx} \|^2  $ respectively.

\begin{lemma}\label{psixxguji}
Suppose that all the conditions in Theorem \ref{thm1} hold. For all $0 < t < T $, there exists a constant $\bar{C}$ depending only on $v_\pm$ and $|u_-|$ such that if $ 0< \varepsilon <\bar{C} $, then
\begin{equation}\label{psixx}
  \| \psi_{xx} \|^2  +  \int_0^t \| \psi_{xxx} \|^2  \,\mathrm{d}\tau   \leq  C \delta^{\frac{2}{3}}   +  C \| (\phi_0, \psi_0) \|^2_{H^2}   +  C \| (\zeta_0, E_0, b_0) \|^2_{H^1} .
\end{equation}

\end{lemma}

\begin{proof}[Proof]
Differentiating $\eqref{raodong}_{2}$ with respect to $x$ and multiplying by $ -\psi_{xxx} $, then integrating the resulting equality with respect to $x$ over $\mathbb{R}$, we can get
\begin{small}
\begin{align}
&\qquad \frac{\mathrm{d}}{\mathrm{d}t} \int_{\mathbb{R}} \frac{1}{2} \psi_{xx}^2 \,\mathrm{d}x       + \mu \int_{\mathbb{R}} \frac{\psi_{xxx}^2}{v} \,\mathrm{d}x     \nonumber \\[2mm]
&= \mu \int_{\mathbb{R}}  \frac{\psi_{xx} v_x }{v^2} \psi_{xxx}  \,\mathrm{d}x         +  \mu \int_{\mathbb{R}} \left( \frac{\psi_x v_x}{v^2} \right)_x \psi_{xxx}  \,\mathrm{d}x           +  \int_{\mathbb{R}} \left( \frac{R \zeta - p_+ \phi}{v} \right)_{xx} \psi_{xxx}  \,\mathrm{d}x       \nonumber \\[2mm]
&\quad     +    \mu \int_{\mathbb{R}}  \left( \frac{\phi \bar{u}_x }{v \bar{v}} \right)_{xx} \psi_{xxx}  \,\mathrm{d}x       + \int_{\mathbb{R}} \left[ \bar{u}_t  -  \mu \left( \frac{\bar{u}_x}{\bar{v}} \right)_{x}  \right]_x \psi_{xxx}  \,\mathrm{d}x        \nonumber \\[2mm]
&\quad     +  \int_{\mathbb{R}} \left[ v(E+\psi b + \bar{u}b)b \right]_x \psi_{xxx}  \,\mathrm{d}x  =: \sum_{i=1}^{6} Y_i.       \label{psixx0}
\end{align}
\end{small}
By using $ \| \left( \phi_x, \psi_x \right) \|_{L^{\infty}} \leq \sqrt{2} \varepsilon_0 $ in \eqref{raodongwuqiongguji}, the Cauchy inequality and Lemma \ref{jianduanbol2mo}, we can get
%\begin{equation}\label{y1y2}
\begin{align} %[b]
   Y_1 + Y_2
   & \leq C \int_{\mathbb{R}}   \left| \psi_{xxx} \right| \cdot \left[ \left| \psi_{xx} \right| \left( \left| \phi_x \right|  +  \left| \bar{\theta}_x \right| \right)  +  \left| \psi_x \right| \left( \left| \phi_{xx} \right|  + \phi_x^2   +  \left| \bar{\theta}_{xx}  \right|     + \bar{\theta}_x^2     \right)  \right]  \,\mathrm{d}x     \nonumber \\[2mm]
   & \leq C \eta \| \psi_{xxx} \|^2    +   C \| (\psi_x, \phi_{xx}, \psi_{xx} \|^2,      \label{y1y2}
\end{align}
%\end{equation}
%%%%%%%%%%%%%%%%%%%%%%%%%%%%%%%
%\begin{equation}\label{y3}
\begin{align} %[b]
   Y_3
   & \leq C \int_{\mathbb{R}}  \left| \psi_{xxx} \right| \cdot \left[ \left| \phi_{xx} \right|  +  \left| \zeta_{xx} \right|   +  \left( \left| \phi_x \right| + \left| \zeta_x \right| \right)\left( \left| \phi_x \right|  +  \left| \bar{\theta}_x \right|  \right)   \right]    \,\mathrm{d}x      \nonumber \\[2mm]
   & \quad  + C \int_{\mathbb{R}}  \left| \psi_{xxx} \right| \cdot \left[ \left( \left| \phi \right|  +  \left| \zeta \right| \right) \left( \left| \phi_{xx} \right|  +  \phi_x^2  +  \bar{\theta}_x^2   +  \left| \bar{\theta}_{xx} \right| \right)   \right]   \,\mathrm{d}x      \nonumber \\[2mm]
   & \leq C \eta \| \psi_{xxx} \|^2     +  C \| \left( \phi _{x}, \zeta _x, \phi _{xx}, \zeta _{xx} \right) \|^2   + C \delta (1+t)^{-\frac{3}{2}},     \label{y3}
 \end{align}
%\end{equation}
%%%%%%%%%%%%%%%%%%%%%%%%%%%%%%%%%%%
%\begin{equation}\label{y4}
\begin{align} %[b]
  Y_4
  & \leq  C \int_{\mathbb{R}} \left| \psi _{xxx} \right| \cdot \left( \left| \phi \partial_x^4 \bar{\theta} \right|  +  \left| \phi_x \bar{\theta}_{xxx}  \right|  + \left| \phi _{xx} \bar{\theta} _{xx} \right|  +  \left| \phi_x^2 \bar{\theta}_{xx} \right| \right)  \,\mathrm{d}x      \nonumber \\[2mm]
  & \leq  C \eta \| \psi _{xxx}  \|^2   +  C \delta(1+t)^{-\frac{3}{2}}   +  C \| (\phi_x, \phi _{xx} ) \|^2 ,     \label{y4}
\end{align}
%\end{equation}
%%%%%%%%%%%%%%%%%%%%%%%%%%%%%%%%%%%%%%%
\begin{equation}\label{y5}
   \left| Y_5 \right|  \leq C \int_{\mathbb{R}} \left| \partial_x^4 \bar{\theta} \right| \left| \psi_{xxx} \right| \,\mathrm{d}x  \leq C \eta \| \psi _{xxx}  \|^2  +  C \delta (1+t)^{-\frac{3}{2}}
\end{equation}
%%%%%%%%%%%%%%%%%%%%%%%%%%%%%%%%%%%%%%%%%
and
%\begin{equation}\label{y6}
\begin{align} %[b]
  Y_6
  & \leq C \int_{\mathbb{R}} \left| \psi _{xxx} \right| \cdot \left| E + \psi b + \bar{u} b  \right|  \,\mathrm{d}x       + C \int_{\mathbb{R}} \left| \psi_{xxx} \right| \cdot \left| b_x \right|  \,\mathrm{d}x       \nonumber \\[2mm]
  &  \quad   +  C \int_{\mathbb{R}} \left| \psi_{xxx} \right| \left( \left| E_x \right| + \left| \psi_x \right| + \left| b_x \right| + \left| \bar{\theta}_{xx} \right|     \right)  \,\mathrm{d}x        \nonumber \\[2mm]
  &  \leq  C \eta \| \psi_{xxx} \|^2   +  C \| \left( \psi_x, E_x, b_x, E + \psi b + \bar{u} b  \right) \|^2  +  C \delta (1+t)^{-\frac{3}{2}}.      \label{y6}
\end{align}
%\end{equation}
Putting $ \eqref{y1y2}$--$\eqref{y6}$ into $\eqref{psixx0}$ and choosing $\eta$ suitably small, we can reach
%\begin{equation}\label{psixx1}
\begin{align} %[b]
   &\qquad \frac{\mathrm{d}}{\mathrm{d}t} \int_{\mathbb{R}} \frac{1}{2} \psi_{xx}^2 \,\mathrm{d}x       + C \int_{\mathbb{R}} \psi_{xxx}^2 \,\mathrm{d}x     \nonumber\\[2mm]
   & \leq     C \| \left( \phi_x, \psi_x, \zeta_x, E_x, b_x, \phi_{xx}, \psi_{xx}, \zeta_{xx}, E + \psi b + \bar{u} b  \right)  \|^2     + C \delta (1+t)^{-\frac{3}{2}} .     \label{psixx1}
\end{align}
%\end{equation}
Estimate \eqref{psixx} thus follows from \eqref{psixx1}, \eqref{jibennenglianggai}, \eqref{exbx}, \eqref{phix}, \eqref{psixzetax} and \eqref{phixx}.

Lemma \ref{psixxguji} is proved.
\end{proof}

\begin{lemma}\label{zetaxxguji}
Suppose that all the conditions in Theorem \ref{thm1} hold. For all $0 < t < T $, there exists a constant $\bar{C}$ depending only on $v_\pm$ and $|u_-|$ such that if $ 0< \varepsilon <\bar{C} $, then
\begin{equation}\label{zetaxx}
  \| \zeta_{xx} \|^2    +   \int_0^t \| \zeta_{xxx} \|^2  \,\mathrm{d}\tau    \leq    C \delta^{\frac{2}{3}}    +  C \| ( \phi_0, \psi_0, \zeta_0 ) \|^2_{H^2}   +  C \| ( E_0, b_0 ) \|^2_{H^1} .
\end{equation}

\end{lemma}

\begin{proof}[Proof]
Differentiating $\eqref{raodong}_3$ with respect to $x$ and multiplying by $ -\zeta_{xxx} $, then integrating the resulting equality with respect to $x$ over $\mathbb{R}$, we can get
%\begin{equation}\label{zetaxx0}
\begin{align} %[b]
&\qquad \frac{\mathrm{d}}{\mathrm{d}t}\int_{\mathbb{R}} \frac{R}{2(\gamma-1)}\zeta^2_{xx} \,\mathrm{d}x    +   \kappa\int_{\mathbb{R}} \frac{{\zeta_{xxx}^2}}{v} \,\mathrm{d}x     \nonumber \\[2mm]
& =  \kappa \int_{\mathbb{R}} \left[ \frac{\zeta_{xx} v_x}{v^2}  +  \left( \frac{\zeta_x v_x}{v^2} \right)_x  + \left( \frac{\phi \bar{\theta}_x }{v \bar{v}} \right)_{xx} \right] \zeta_{xxx}  \,\mathrm{d}x        \nonumber \\[2mm]
& \quad  + \int_{\mathbb{R}} \left( p \psi_x \right)_x \zeta_{xxx}  \,\mathrm{d}x   +  \int_{\mathbb{R}} \left( \frac{R \zeta - p_+ \phi}{v} \bar{u}_x \right)_x \zeta_{xxx}  \,\mathrm{d}x        \nonumber \\[2mm]
& \quad  -\mu \int_{\mathbb{R}} \left[ \frac{\left( \psi_x + \bar{u}_x \right)^2  }{v} \right]_x \zeta_{xxx}  \,\mathrm{d}x      - \int_{\mathbb{R}} \left[ v \left( E + \psi b + \bar{u} b  \right)^2   \right]_x \zeta _{xxx}   \,\mathrm{d}x        = : \sum_{i=1}^5 Z_i.    \label{zetaxx0}
\end{align}
%\end{equation}
Similar to the treatment of \eqref{psixx0},
due to $\|(\phi_x, \psi_x, \zeta_x)\|_{L^{\infty}}\leq \sqrt{2}\varepsilon_0 $ in \eqref{raodongwuqiongguji},
we can estimate the \eqref{zetaxx0} as follows:
%\begin{equation}\label{z1}
\begin{align} %[b]
  Z_1
  & \leq C \int_{\mathbb{R}} \left( \left| \phi_{xx} \bar{\theta}_x  \right|    + \left| \phi_x^2 \bar{\theta}_x \right|     +  \left| \phi_x \bar{\theta}_{xx} \right|      +  \left| \phi \bar{\theta}_{xxx} \right| \right) \left| \zeta_{xxx} \right| \,\mathrm{d}x       \nonumber \\[2mm]
  & \quad   + C \int_{\mathbb{R}} \left(  \left| \zeta_{xx} \phi_x \right|  +  \left| \zeta_{xx} \bar{\theta}_x  \right|  + \left| \zeta_x \phi_{xx} \right|   + \left| \zeta_x \phi_x^2 \right|    + \left| \zeta_x \bar{\theta}_{xx}  \right|  \right) \left| \zeta_{xxx} \right|   \,\mathrm{d}x      \nonumber \\[2mm]
  & \leq C \eta \| \zeta_{xxx} \|^2       +  C \| \left( \phi_{x}, \phi_{xx}, \zeta_x, \zeta_{xx} \right) \|^2     +  C \delta (1+t)^{- \frac{3}{2}},    \label{z1}
\end{align}
%\end{equation}
%%%%%%%%%%%%%%%%%%%%%%%%%%%%%%%%%%%%%%%%%%%%
%\begin{equation}\label{z2}
\begin{align} %[b]
  Z_2
  & \leq C \int_{\mathbb{R}} \left[ \left( \left| \phi_x \right|  +  \left| \zeta_x \right|  +  \left| \bar{\theta}_x \right| \right) \left| \psi_x \right| +  \left| \psi_{xx} \right| \right]   \left| \zeta_{xxx} \right| \,\mathrm{d}x
   \nonumber \\[2mm]
  & \leq C \eta \| \zeta _{xxx}  \|^2   +  C \| (\psi_x, \psi_{xx}) \|^2,     \label{z2}
\end{align}
\begin{align} %[b]
   Z_3
   & \leq  C \int_{\mathbb{R}}  \left[ \left(  \left| \phi_x \right|  +  \left| \zeta_x \right|  \right) \left| \bar{\theta}_{xx} \right|   +  \left(  \left| \phi \right|  +  \left| \zeta \right| \right)  \left| \bar{\theta}_{xxx} \right| \right]   \left| \zeta_{xxx} \right| \,\mathrm{d}x     \nonumber \\[2mm]
   & \leq C \eta \| \zeta_{xxx} \|^2   +  C \| (\phi_x, \zeta_x) \|^2   +  C \delta (1+t)^{-\frac{3}{2}} ,     \label{z3}
\end{align}
%%%%%%%%%%%%%%%%%%%%%%%%%%%
\begin{align}
   Z_4
   & = - \mu \int_{\mathbb{R}} \left[ \frac{2 \left( \psi_x + \bar{u}_x \right) \left( \psi_{xx} + \bar{u}_{xx} \right)  }{v}   - \frac{\left( \psi_x + \bar{u}_x \right)^2 v_x }{v^2} \right] \zeta_{xxx} \,\mathrm{d}x    \nonumber \\[2mm]
   & \leq C \int_{\mathbb{R}} \left[ \left( \left| \psi_x  \right|  + \left| \bar{\theta}_{xx} \right| \right) \left( \left| \psi_{xx} \right| + \left| \bar{\theta}_{xxx} \right| \right)  +  \left( \psi_x^2 + \bar{\theta}_{xx}^2  \right) \left( \left| \phi_x \right| + \left| \bar{\theta}_x \right| \right) \right] \left| \zeta_{xxx} \right| \,\mathrm{d}x      \nonumber \\[2mm]
   & \leq  C \eta \| \zeta_{xxx} \|^2          + C \| \left( \psi_x, \psi_{xx} \right)  \|^2        + C \delta (1+t)^{-\frac{3}{2}}    \label{z4}
\end{align}
%%%%%%%%%%%%%%%%%%%%%%%%%%%%%%%%%%
and
%%%%%%%%%%%%%%%%%%%%%%%%%%%%%%%%%%
\begin{align}
   Z_5
  & \leq   C\int_{\mathbb{R}} \left( \left| \phi_x \right|  +  \left| \bar{\theta}_x \right| \right) \left( E + \psi b + \bar{u} b  \right)^2  \left| \zeta_{xxx} \right|  \,\mathrm{d}x       \nonumber \\[2mm]
  & \quad  +  C\int_{\mathbb{R}} \left| E + \psi b + \bar{u} b  \right| \left( \left| E_x \right| + \left| \psi_x b \right|  +  \left| \psi b_x \right|   +  \left| \bar{u}_x b \right|  +  \left| \bar{u} b_x \right| \right) \left| \zeta_{xxx} \right| \,\mathrm{d}x      \nonumber \\[2mm]
  & \leq   C \eta \| \zeta_{xxx} \|^2   +  C \| \left( \psi_x, E_x, b_x, E + \psi b + \bar{u} b  \right) \|^2   + C \delta (1+t)^{-\frac{3}{2}}.        \label{z5}
\end{align}
%%%%%%%%%%%%%%%%%%%%%%%%%%%%%%%%%%
Putting $ \eqref{z1} $--$\eqref{z5}$ into $\eqref{zetaxx0}$ and choosing $\eta$ suitably small, we can obtain
%%%%%%%%%%%%%%%%%%%%%%%%%%%%%%%%%%
\begin{align}
   &\qquad \frac{\mathrm{d}}{\mathrm{d}t}\int_{\mathbb{R}} \frac{R}{2(\gamma-1)}\zeta^2_{xx} \,\mathrm{d}x       +  C \int_{\mathbb{R}} \zeta_{xxx}^2 \,\mathrm{d}x     \nonumber \\[2mm]
   & \leq  C \| \left( \phi_x, \psi_x, \zeta_x, E_x, b_x, \phi_{xx}, \psi_{xx}, \zeta_{xx}, E + \psi b + \bar{u} b  \right)\|^2  +  C \delta (1+t)^{-\frac{3}{2}} .     \label{zetaxx1}
\end{align}
%%%%%%%%%%%%%%%%%%%%%%%%%%%%%%%%%%
Then integrating $ \eqref{zetaxx1} $ with respect to $t$, together with \eqref{jibennenglianggai}, \eqref{exbx}, \eqref{phix}, \eqref{psixzetax}, \eqref{phixx} and \eqref{psixx}, we conclude $ \eqref{zetaxx} $.

The proof of Lemma \ref{zetaxxguji} is completed.
\end{proof}

\begin{proof}[Proof of Proposition \ref{prop1}:]
We combine Lemma \ref{dijieguji}--Lemma \ref{zetaxxguji}, then choose $\varepsilon_0$ and $\delta$ small enough to finish the proof of Proposition \ref{prop1}.
\end{proof}

%%%%%%%%%%%%%%%%%%%%%%%%%%%%%%%%%%%%%%%%%%%%%%%%%%%%%%%%%%%%%%%%%%%%%%%%%%%%%%%%%%%%%%%%%

\section{Proof of Theorem \ref{thm2} (Composite wave)}\label{section4}

To prove Proposition \ref{prop1} for composite wave of the pattern of $ R_1CR_3 $,
noticing that $ \left( V^r_\pm, U^r_\pm, \Theta^r_\pm \right)  $ satisfies Euler system \eqref{nstuidao2} and $ \left( V^c, U^c, \Theta^c \right) $ satisfies $\eqref{lagrange}_1$, \eqref{gouzao1} and \eqref{smoothjianduanbo};
we transform the Cauchy problem $\eqref{lagrange}$, $\eqref{chuzhi'}$, $\eqref{wuqiong1}$ as
\begin{equation}\label{raodong2}
\left\{
\begin{aligned}
&\phi_t-\psi_x=0, \\[2mm]
&\psi_t + \left( \frac{R\zeta - P \phi}{v} \right)_x     = \mu \left( \frac{\psi_x}{v} \right)_x    + F    -v(E+\psi b+ U b)b,     \\[2mm]
&\frac{R}{\gamma - 1} \zeta_{t}     + \left( pu_x  - PU_x \right)     = \kappa\left( \frac{V \zeta_x-\phi \Theta_x }{v V} \right)_x     + G  + v \left( E + \psi b + U b  \right)^2,      \\[2mm]
&\varepsilon \left(E_{t} - \frac{u}{v}E_x\right) -\frac{1}{v}b_{x} + E +\psi b + U b = 0,     \\[2mm]
&b_{t}-\frac{u}{v}b_x-\frac{1}{v}E_{x}=0,   \\[2mm]
&(\phi, \psi, \zeta, E, b)(x,0) = (\phi_0, \psi_0, \zeta_0, E_0, b_0)(x),\quad x \in \mathbb{R},    \\[2mm]
&\lim_{x \to \pm \infty}(\phi_0, \psi_0, \zeta_0, E_0, b_0)(x) = 0,
\end{aligned}
\right.
\end{equation}
where
\begin{equation*}
 P  = \frac{R\Theta}{V},     \qquad      P^r_\pm = \frac{R\Theta^r_\pm}{V^r_\pm}, \qquad    p^{m} = P^c = \frac{R \Theta^{c}}{V^{c}},
\end{equation*}
\begin{equation*}
 F = \left( P^r_- + P^r_+ - P \right)_x   +  \left[ \mu \left( \frac{U_x}{v} \right)_x  - U^c_t  \right]
\end{equation*}
and
%\begin{equation*}
\begin{align} %[b]
  & G = \left[ \left( p^m - P \right) U^c_x  +  \left( P^r_- - P \right) \left( U^r_- \right)_x    + \left( P^r_+ - P \right) \left( U^r_+ \right)_x    \right]     \nonumber \\[2mm]
  & \qquad + \mu \frac{\left( \psi_x + U_x \right)^2 }{v}   +  \kappa \left[ \left( \frac{\Theta_x}{V} \right)_x  -  \left(\frac{\Theta^c_x}{V^c} \right)_x  \right] .  \nonumber
\end{align}
%\end{equation*}

Like Lemma \ref{dijieguji} and Lemma \ref{quanguji}, the following key estimate holds.

\begin{lemma}\label{lemma4.1}
Suppose that all the conditions in Theorem \ref{thm2} hold. For all $0 < t < T $, there exists a constant $\bar{C}$ depending only on $ v_\pm $, $ u_\pm $ and $ \theta_\pm $ such that if $ 0< \varepsilon <\bar{C} $, then
\begin{align} 
 & \qquad   \| \left( \phi, \psi, \zeta, \sqrt{\varepsilon}E, b \right)  \|^2         +  \int_0^t   \| ( \psi_x, \zeta_x, E+\psi b+ U b ) \|^2  \,\mathrm{d}\tau     \nonumber \\[2mm]
 & \quad       +  \int_0^t\int_{\mathbb{R}} \left( \phi^2 + \zeta^2 + \varepsilon E^2 + b^2  \right)   \left( \left( U^r_- \right)_x  +  \left( U^r_+ \right)_x  \right)    \,\mathrm{d}x \mathrm{d}\tau    \nonumber \\[2mm]
 & \leq   C\| (\phi_0, \psi_0, \zeta_0, E_0, b_0) \|^2      + C\delta^{\frac{1}{6}}      + C \delta^{\frac{1}{2}} \int_0^t \| \phi_x \|^2  \,\mathrm{d}\tau      + C \delta^{\frac{1}{2}} \int_0^t \| b_x \|^2  \,\mathrm{d}\tau .     \label{jibennengliang'}
\end{align}
Here $\omega$ is the function defined in \eqref{omegarehe} with $ \alpha = {\hat{c}}/{4} $.

\end{lemma}

\begin{proof}[Proof]
Motivated by Lemma $8$ of \cite{huangfm2010}, for $ \Phi(s) = s -1 - \ln s $, $ s>0 $, noticing that
\begin{equation*}
  \left[ R \Theta \Phi\left( \frac{v}{V} \right) \right]_t
   = -R \Theta \left( \frac{1}{v} - \frac{1}{V} \right) \phi_t  - \frac{P \phi^2}{v V}V_t + R \Theta_t \Phi \left( \frac{v}{V} \right)
\end{equation*}
and
\begin{equation*}
  \left[ \Theta \Phi \left( \frac{\theta}{\Theta} \right) \right]_t
   = \frac{\zeta}{\theta}\zeta_t - \Phi \left( \frac{\Theta}{\theta} \right) \Theta_t,
\end{equation*}
we directly calculate
%\begin{equation}\label{shangliudui2-1}
\begin{align} %[b]
  & \left( \frac{1}{2} \psi^2 + R \Theta \Phi \left( \frac{v}{V} \right) + \frac{R}{\gamma-1} \Theta \Phi \left( \frac{\theta}{\Theta} \right)  \right)_t    \nonumber \\[3mm]
  =  & \left[  \psi \psi_t    -R \Theta \left( \frac{1}{v} - \frac{1}{V} \right) \phi_t    + \frac{R}{\gamma-1} \zeta_t \frac{\zeta}{\theta}       + \frac{\zeta}{\theta} \left( p - P \right) U_x  \right]          \nonumber \\[3mm]
  &  - \left[  - R \Theta_t \Phi \left( \frac{v}{V} \right)        + \frac{P \phi^2}{v V}V_t      + \frac{R}{\gamma-1} \Phi \left( \frac{\Theta}{\theta} \right) \Theta_t    + \frac{\zeta}{\theta} \left( p - P \right) U_x        \right] .     \label{shangliudui2-1}
\end{align}
%\end{equation}
Further calculation shows that
%\begin{equation}\label{shangliudui2-2}
\begin{align} %[b]
  & \psi \psi_t    -R \Theta \left( \frac{1}{v} - \frac{1}{V} \right) \phi_t    + \frac{R}{\gamma-1} \zeta_t \frac{\zeta}{\theta}       + \frac{\zeta}{\theta} \left( p - P \right) U_x      \nonumber \\[3mm]
  = & - \mu \frac{\psi_x^2}{v}   - \kappa \frac{\zeta_x^2}{v \theta}  +  \left[ \mu \frac{\psi \psi_x }{v}  - \left( R \zeta - P \phi \right) \frac{\psi}{v}  \right]_x       +  \kappa \left[  \frac{\zeta}{\theta v V} \left( V \zeta_x - \phi \Theta_x \right) \right] _x         \nonumber \\[3mm]
  &   + \kappa \frac{\phi \zeta_x \Theta_x}{v \theta V }    + \kappa \frac{\zeta \zeta_x \theta_x}{v \theta^2}      - \kappa \frac{\phi \zeta \Theta_x \theta_x}{v \theta^2 V}        + G \frac{\zeta}{\theta}   + F \psi    +  v  \left( E + \psi b + U b \right)^2 \frac{\zeta}{\theta}      \nonumber \\[3mm]
  &   -v \left( E + \psi b + U b  \right) \psi b.      \label{shangliudui2-2}
\end{align}
%\end{equation}
Due to
%\begin{equation*}
\begin{align} %[b]
  -R \Theta_t
  & = \left( \gamma-1 \right) P^r_- \left( U^r_- \right)_x       +  \left( \gamma-1 \right) P^r_+ \left( U^r_+ \right)_x       - p^m U^c_x    \nonumber \\[2mm]
  & = \left( \gamma-1 \right) P \left( U^r_- \right)_x       +  \left( \gamma-1 \right) P \left( U^r_+ \right)_x       + \left( \gamma-1 \right) \left(  P^r_-  - P \right)  \left( U^r_- \right)_x     \nonumber \\[2mm]
  &\quad   +  \left( \gamma-1 \right) \left( P^r_+ - P \right)  \left( U^r_+ \right)_x      - p^m U^c_x  ,     \nonumber
\end{align}
%\end{equation*}
we can compute that
\begin{small}
\begin{equation}\label{shangliudui2-3}
  - R \Theta_t \Phi \left( \frac{v}{V} \right)        + \frac{P \phi^2}{v V}V_t      + \frac{R}{\gamma-1} \Phi \left( \frac{\Theta}{\theta} \right) \Theta_t    + \frac{\zeta}{\theta} \left( p - P \right) U_x     = Q_1 \left[ \left( U^r_- \right)_x   +   \left( U^r_+ \right)_x   \right]     +  Q_2 ,
\end{equation}
\end{small}
where
%\begin{equation}\label{Q1}
\begin{align} %[b]
  Q_1
  &  =  \frac{P \phi^2}{v V}  +  \left( \gamma - 1 \right) P \Phi \left( \frac{v}{V} \right)   - P \Phi \left( \frac{\Theta}{\theta} \right)    +  \frac{\zeta}{\theta} \left( p - P \right)      \nonumber \\[2mm]
  &  =  \gamma P \Phi \left( \frac{v}{V} \right)        + \left[ \frac{P \phi^2}{v V}   - P \Phi \left( \frac{v}{V} \right)   - P \Phi \left( \frac{\Theta}{\theta} \right)    +  \frac{\zeta}{\theta} \left( p - P \right)   \right]      \nonumber \\[2mm]
  &  =  \gamma P \Phi \left( \frac{v}{V} \right)       + \left[ \frac{P \phi^2}{v V}    - \frac{R \phi \zeta}{v V}      +  P \left(  \Phi \left( \frac{\theta}{\Theta} \right)  - \Phi \left( \frac{v}{V} \right)   \right)    \right]      \nonumber \\[2mm]
  &  =  \gamma P \Phi \left( \frac{v}{V} \right)       + P \Phi \left( \frac{\theta V}{\Theta v} \right)     \geq C \left( \phi^2 + \zeta^2 \right)    \label{Q1}
\end{align}
%\end{equation}
and
%\begin{equation*}
\begin{align}
Q_{2}
& =  U_{x}^{c}\left(\frac{P \phi^{2}}{v V}-p^{m} \Phi\left(\frac{v}{V}\right)+\frac{p^{m}}{\gamma-1} \Phi\left(\frac{\Theta}{\theta}\right)+\frac{\zeta}{\theta}(p-P)\right)    \nonumber \\[2mm]
&\quad  +(\gamma-1)\left(P^r_{-}-P\right)\left(U_{-}^{r}\right)_{x}\left(\Phi\left(\frac{v}{V}\right)-\frac{1}{\gamma-1} \Phi\left(\frac{\Theta}{\theta}\right)\right)   \nonumber \\[2mm]
&\quad  +(\gamma-1)\left(P^r_{+}-P\right)\left(U_{+}^{r}\right)_{x}\left(\Phi\left(\frac{v}{V}\right)-\frac{1}{\gamma-1} \Phi\left(\frac{\Theta}{\theta}\right)\right) .      \nonumber
\end{align}
%\end{equation*}
Putting \eqref{shangliudui2-2} and \eqref{shangliudui2-3} into \eqref{shangliudui2-1}, we obtain that
%\begin{equation}\label{shangliudui3}
\begin{align} %[b]
 & \quad   \left( \frac{1}{2} \psi^2 + R \Theta \Phi \left( \frac{v}{V} \right) + \frac{R}{\gamma-1} \Theta \Phi \left( \frac{\theta}{\Theta} \right)  \right)_t        \nonumber \\[2mm]
 & \quad   + \mu \frac{\psi_x^2}{v}       + \kappa \frac{\zeta_x^2}{v \theta}       + Q_1 \left[ \left( U^r_- \right)_x   +   \left( U^r_+ \right)_x   \right]             + H_x         + Q      \nonumber \\[2mm]
 &   =  G \frac{\zeta}{\theta}    + F \psi    + v \left( E + \psi b + U b  \right)^2 \frac{\zeta}{\theta}   - v \left( E + \psi b + U b  \right) \psi b,     \label{shangliudui3}
\end{align}
%\end{equation}
where
\begin{equation*}
   H =   -\mu \frac{\psi \psi_x }{v}      + \left( R \zeta - P \phi \right) \frac{\psi}{v}        -  \kappa \frac{\zeta}{\theta v V} \left( V \zeta_x - \phi \Theta_x \right)
\end{equation*}
and
\begin{equation*}
  Q = Q_2     - \kappa \frac{\phi \zeta_x \Theta_x}{v \theta V }       - \kappa \frac{\zeta \zeta_x \theta_x}{v \theta^2}        + \kappa \frac{\phi \zeta \Theta_x \theta_x}{v \theta^2 V}.
\end{equation*}
Following the same computations as in Lemma $8$ of \cite{huangfm2010}, it holds that
%\begin{equation}\label{Qjue}
\begin{align}  %[b]
\left| Q \right|
& \leq  C_{\eta} \left(\phi^{2} + \zeta^{2}\right) \left(\Theta_{x}^{2} + \left|\Theta^{c}_{x x}\right|\right)    + \eta \zeta_{x}^{2}     + C \delta \mathrm{e}^{-c_{0}(|x|+t)}      \nonumber \\[3mm]
& \leq  C_{\eta} \left(\phi^{2} + \zeta^{2}\right) \left(\delta^{\frac{1}{4}}(1+t)^{-\frac{7}{4}}    + \delta (1+t)^{-1} \mathrm{e}^{ -\frac{\hat{c} x^{2}}{1+t} } \right)        + \eta \zeta_{x}^{2}    + C \delta \mathrm{e}^{-c_{0}(|x|+t)},     \label{Qjue}
\end{align}
%\end{equation}
\begin{equation}\label{fl12}
\|F\|_{L^{1}} \leq C \delta^{\frac{1}{8}}(1+t)^{-\frac{7}{8}}  + C \delta^{\frac{1}{2}}\left\|\phi_{x}\right\|^{2}
\end{equation}
and
\begin{equation}\label{gl1}
\|G\|_{L^{1}} \leq C \delta^{\frac{1}{8}}(1+t)^{-\frac{7}{8}}    + C \left\|\psi_{x}\right\|^{2}.
\end{equation}
From \eqref{fl12}, \eqref{gl1} and the Sobolev inequality \eqref{lwuqiong}, we can deduce that
\begin{align}\label{jisuan-liufenzhiyi}
  &\qquad \int_{\mathbb{R}} \left( G \frac{\zeta}{\theta}    + F \psi \right)  \,\mathrm{d}x       \nonumber \\[2mm]
  & \leq  C \left( \| \zeta \|_{{L^{\infty}}}  +  \| \psi \|_{{L^{\infty}}} \right)  \left( \| G \|_{L^1}  +  \| F \|_{L^1}  \right)          \nonumber \\[2mm]
  & \leq  C \left( \| \zeta \|_{{L^{\infty}}}  +  \| \psi \|_{{L^{\infty}}} \right)
  \left( \delta^{\frac{1}{8}}(1+t)^{-\frac{7}{8}}     + \left\|\psi_{x}\right\|^{2}     + \delta^{\frac{1}{2}}\left\|\phi_{x}\right\|^{2}  \right)             \nonumber \\[2mm]
  & \leq C \left( \| \zeta \|^{\frac{1}{2} }  \| \zeta_x \|^{\frac{1}{2} }   + \| \psi \|^{\frac{1}{2} } \| \psi_x \|^{\frac{1}{2} }   \right) \cdot \delta^{\frac{1}{8}} (1+t)^{-\frac{7}{8}}     +  C \varepsilon_0 \left( \| \psi_x \|^2    + \delta^{\frac{1}{2} } \| \phi_x \|^2   \right)             \nonumber \\[2mm]
  & \leq   C \delta^{\frac{1}{6}}(1+t)^{-\frac{7}{6}}        + C \varepsilon_0^2 \left( \| \zeta_x \|^2   +  \| \psi_x \|^2   \right)         +  C \varepsilon_0 \left( \| \psi_x \|^2    + \delta^{\frac{1}{2} } \| \phi_x \|^2   \right) .
\end{align}
Integrating $\eqref{shangliudui3}$ with respect to $x$, and combining \eqref{Q1}, \eqref{Qjue} and \eqref{jisuan-liufenzhiyi}, then choosing $\eta$, $\varepsilon_0$ and $ \delta $ suitably small, we can obtain
\begin{align} %[b]
&\quad \frac{\mathrm{d}}{\mathrm{d}t} \int_{\mathbb{R}} \left( \frac{1}{2} \psi^{2}       +  R \Theta \Phi\left(\frac{v}{V} \right)       + \frac{R}{\gamma-1} {\Theta} \Phi\left(\frac{\theta}{{\Theta}}\right) \right) \mathrm{d}x        \nonumber \\[2mm]
&\quad  + C\left\|\left(\psi_{x}, \zeta_{x}\right)\right\|^{2}       + C \int_{\mathbb{R}} \left( \phi^2 + \zeta^2 \right) \left[ \left( U^r_- \right)_x   +   \left( U^r_+ \right)_x  \right]   \,\mathrm{d}x      \nonumber \\[2mm]
&\leq   C \delta^{\frac{1}{6}} (1+t)^{-\frac{7}{6}}        + C \delta  \int_{\mathbb{R}}  \left(\phi^{2} + \zeta^{2}\right) \omega^2   \mathrm{d} x        + C \delta^{\frac{1}{2}} \| \phi_x  \|^2      \nonumber \\[2mm]
& \quad  + \int_{\mathbb{R}}  v(E+\psi b+ U b)^2\frac{\zeta}{\theta}\,{\rm{d}}x      -\int_{\mathbb{R}} v(E+\psi b+ U b)\psi b\,{\rm{d}}x  .      \label{jiben2}
\end{align}
%\end{equation}
Next we will go on to produce the composite good term: $ \int_0^t\int_{\mathbb{R}}  v(E+\psi b+ U b)^2 \,{\rm{d}}x\mathrm{d}\tau $.

Similar to the derivation of \eqref{tuidao1}, \eqref{tuidao2} in Lemma \ref{dijieguji}, by using the equation $ (vb)_t  -  (E + ub)_x  = 0  $, we have
\begin{equation}\label{tuidao1'}
\frac{\rm d}{{\rm d}t} \int_{\mathbb{R}} \frac{1}{2}(\varepsilon vE^2+vb^2) \,{\rm{d}}x       + \int_{\mathbb{R}}v(E + \psi b + U b)E \,{\rm{d}}x  = 0
\end{equation}
and
\begin{small}
\begin{align}
&\quad \frac{\rm d}{{\rm d}t}\int_{\mathbb{R}} \varepsilon  v U b E \,{\rm{d}}x    +\frac{1}{2}\int_{\mathbb{R}} [(U^r_-)_x+(U^r_+)_x] (\varepsilon E^2 + b^2)\,{\rm{d}}x      + \int_{\mathbb{R}} v (E+\psi b+ Ub) U b\,{\rm{d}}x     \nonumber \\[2mm]
& =  - \frac{1}{2} \int_{\mathbb{R}} U^c_x \left( \varepsilon E^2 + b^2 \right) \,\mathrm{d}x        + \int_{\mathbb{R}} \varepsilon \left( v  U^{c}_t  -  u U^{c}_x  \right)  E b \,{\rm{d}}x     \nonumber \\[2mm]
& \quad + \int_{\mathbb{R}} \varepsilon vb \left[  (U_{-}^{r})_t  +  (U_{+}^{r})_t \right]  E\,{\rm{d}}x       - \int_{\mathbb{R}} \varepsilon ub \left[ (U_{-}^{r})_x +  (U_{+}^{r})_x    \right]  E\,{\rm{d}}x,     \label{guji8}
\end{align}
\end{small}
respectively.
Set $ \beta :=  \max \left\{ \left| u_\pm \right|  \right\} > 0 $ for the sake of simplicity.
By employing similar procedures from \eqref{remaining1} to \eqref{remaining3}, we obtain
%\begin{equation}\label{jianduanboleisi}
\begin{align} %[b]
  &\quad - \frac{1}{2} \int_{\mathbb{R}} U^c_x \left( \varepsilon E^2 + b^2 \right) \,\mathrm{d}x        + \int_{\mathbb{R}} \varepsilon \left( v  U^{c}_t  -  u U^{c}_x  \right)  E b \,{\rm{d}}x      \nonumber \\[2mm]
  & \leq C \delta^{\frac{5}{6}} \int_{\mathbb{R}} \left( E + \psi b + U b  \right)^2  \,\mathrm{d}x     +  C \delta^{\frac{2}{3}}\int_{\mathbb{R}} b^2 \omega^2 \,\mathrm{d}x.      \label{jianduanboleisi}
\end{align}
%\end{equation}
Moreover, $\eqref{gouzaoxishubo}_2$ gives that
%\begin{equation}\label{j1j2}
\begin{align} %[b]
& \qquad \int_{\mathbb{R}} \varepsilon vb \left[  (U_{-}^{r})_t  +  (U_{+}^{r})_t \right]  E\,{\rm{d}}x       - \int_{\mathbb{R}} \varepsilon ub \left[ (U_{-}^{r})_x +  (U_{+}^{r})_x    \right]  E\,{\rm{d}}x      \nonumber \\[2mm]
& =: \int_{\mathbb{R}} \varepsilon vb (U_{\pm}^{r})_t  E\,{\rm{d}}x      - \int_{\mathbb{R}} \varepsilon ub (U_{\pm}^{r})_x   E\,{\rm{d}}x     \nonumber \\[2mm]
& =  - \int_{\mathbb{R}} \varepsilon v \lambda_\pm \left( V^r_\pm, s_\pm \right) \left( U^r_\pm \right)_x bE   \,\mathrm{d}x        - \int_{\mathbb{R}} \varepsilon ub (U_{\pm}^{r})_x   E\,{\rm{d}}x .      \label{j1j2}
\end{align}
%\end{equation}
Due to \eqref{fuhevshangjie}, \eqref{Uushangxiajie} and $ \max \left\{ \left| \lambda _\pm \left( V^r_\pm, s_\pm \right) \right| \right\} < \max \left\{ \sqrt{\gamma R \theta_\pm v_\pm ^{-2}} \right\} $,
we have
%\begin{equation}
\begin{align} %[b]
 &\quad - \int_{\mathbb{R}} \varepsilon v \lambda_\pm \left( V^r_\pm, s_\pm \right) \left( U^r_\pm \right)_x bE   \,\mathrm{d}x       \nonumber \\[2mm]
 & \leq   2 \max \left\{ v_\pm \right\} \cdot \max \left\{ \sqrt{\gamma R \theta_\pm v_\pm ^{-2}}  \right\} \int_{\mathbb{R}} \varepsilon  \left( U^r_\pm \right)_x \left| bE \right|  \,\mathrm{d}x      \nonumber \\[2mm]
 & =: C_3 \int_{\mathbb{R}} \varepsilon \left( U^r_\pm \right)_x \left| bE \right| \,\mathrm{d}x      \nonumber \\[2mm]
 & \leq C_3 \int_{\mathbb{R}} \varepsilon \left( U^r_\pm \right)_x \left| b \left( E + ub \right)  \right| \,\mathrm{d}x    +  C_3 \int_{\mathbb{R}} \varepsilon \left( U^r_\pm \right)_x \left| u b^2  \right| \,\mathrm{d}x      \nonumber \\[2mm]
 & \leq C_3 \varepsilon \| b \|_{L^{\infty}} \| \left( U^r_\pm \right)_x  \|  \| E + \psi b + U b  \|   + 2C_3 \varepsilon \beta \int_{\mathbb{R}}  \left( U^r_\pm \right)_x b^2  \,\mathrm{d}x      \nonumber \\[2mm]
 & \leq C \varepsilon \| b \|^{\frac{1}{2}} \| b_x \|^{\frac{1}{2}} \delta^{\frac{1}{2}} (1+t)^{-\frac{1}{2}} \| E + \psi b + U b  \|      + 2C_3 \varepsilon \beta \int_{\mathbb{R}}  \left( U^r_\pm \right)_x b^2  \,\mathrm{d}x       \nonumber \\[2mm]
 & \leq  C \varepsilon \varepsilon_0^{\frac{1}{2}} \delta^{\frac{1}{2}} \left(   (1+t)^{-2}      + \| b_x \|^2     + \| E + \psi b + U b  \|^2    \right)     + 2C_3 \varepsilon \beta \int_{\mathbb{R}}  \left( U^r_\pm \right)_x b^2  \,\mathrm{d}x .    \label{diyige0}
\end{align}
%\end{equation}
First choosing $ \varepsilon _0^{\frac{1}{2}} < 2 \beta $ and then taking $ 2C_3 \varepsilon \beta < \frac{1}{8} $ in \eqref{diyige0}, i.e.,
\begin{equation}\label{vepbeta1}
  \varepsilon  \beta < \frac{1}{32 \max \left\{ v_\pm \right\} \cdot \max \left\{ \sqrt{\gamma R \theta_\pm v_\pm ^{-2}}  \right\} }
\end{equation}
leads to
%\begin{equation}\label{diyige}
\begin{align} %[b]
   &\quad - \int_{\mathbb{R}} \varepsilon v \lambda_\pm \left( V^r_\pm, s_\pm \right) \left( U^r_\pm \right)_x bE   \,\mathrm{d}x      \nonumber \\[2mm]
   & \leq  C \delta^{\frac{1}{2}} \left( (1+t)^{-2}      + \| b_x \|^2     + \| E + \psi b + U b  \|^2   \right)     + \frac{1}{8} \int_{\mathbb{R}}  \left( U^r_\pm \right)_x b^2  \,\mathrm{d}x .         \label{diyige}
\end{align}
%\end{equation}
By using the Cauchy inequality and $ \left| u \right| < 2\beta $, we obtain
%\begin{equation*}
\begin{align} %[b]
  - \int_{\mathbb{R}} \varepsilon ub (U_{\pm}^{r})_x   E\,{\rm{d}}x
  & \leq \int_{\mathbb{R}}  \left( U^r_\pm \right)_x \left| b \right| \cdot 2\varepsilon \beta \left| E \right|  \,\mathrm{d}x
   \leq \int_{\mathbb{R}}  \left( U^r_\pm \right)_x \left( \frac{1}{4} b^2  +  4 \varepsilon ^2 \beta^2 E^2   \right)   \,\mathrm{d}x     \nonumber \\[2mm]
  & \leq \frac{1}{4} \int_{\mathbb{R}} \left( U^r_\pm \right)_x b^2 \,\mathrm{d}x     +  4\varepsilon \beta^2 \int_{\mathbb{R}} \left( U^r_\pm \right)_x \varepsilon E^2 \,\mathrm{d}x .     \label{dierge0}
\end{align}
%\end{equation*}
Further taking $  \varepsilon \beta^2 < \frac{1}{16} $ in \eqref{dierge0} implies
\begin{equation}\label{dierge}
    - \int_{\mathbb{R}} \varepsilon ub (U_{\pm}^{r})_x   E\,{\rm{d}}x  \leq \frac{1}{4} \int_{\mathbb{R}} \left( U^r_\pm \right)_x b^2 \,\mathrm{d}x      +  \frac{1}{4} \int_{\mathbb{R}} \left( U^r_\pm \right)_x \varepsilon E^2 \,\mathrm{d}x .
\end{equation}
Thus combining \eqref{guji8}, \eqref{jianduanboleisi}, \eqref{diyige} and \eqref{dierge} leads to
%\begin{equation}\label{tuidao2'}
\begin{align} %[b]
   & \quad \frac{\rm d}{{\rm d}t}\int_{\mathbb{R}} \varepsilon  v U E b \,{\rm{d}}x     + \frac{1}{8} \int_{\mathbb{R}} (U^r_\pm)_x (\varepsilon E^2 + b^2)\,{\rm{d}}x      + \int_{\mathbb{R}} v (E+\psi b+ Ub) U b\,{\rm{d}}x      \nonumber \\[2mm]
   & \leq  C \delta^{\frac{1}{2}} \left(   (1+t)^{-2}   + \| b_x \|^2    + \| E + \psi b + U b  \|^2    \right)     + C \delta^{\frac{2}{3}} \int_{\mathbb{R}} b^2 \omega^2  \,\mathrm{d}x  .     \label{tuidao2'}
\end{align}
%\end{equation}
Summing \eqref{jiben2}, \eqref{tuidao1'} and \eqref{tuidao2'} up and following the similar calculation way as \eqref{etajieguo} and \eqref{danjifen1}, then integrating the resulting inequality with respect to $t$ yields that
%%%%%%%%%%%%%%%%%%%%%%%%%%%%%%%%%%%%%%%%%%%%
%\begin{equation}\label{jibennengliang''}
 \begin{align} %[b]
 & \qquad   \| \left( \phi, \psi, \zeta, \sqrt{\varepsilon}E, b \right)  \|^2         +  \int_0^t   \| ( \psi_x, \zeta_x, E+\psi b+ U b ) \|^2  \,\mathrm{d}\tau    \nonumber \\[2mm]
 & \quad       +  \int_0^t\int_{\mathbb{R}} \left( \phi^2 + \zeta^2 + \varepsilon E^2 + b^2  \right)   \left( \left( U^r_- \right)_x  +  \left( U^r_+ \right)_x  \right)    \,\mathrm{d}x \mathrm{d}\tau    \nonumber \\[2mm]
 & \leq   C\| (\phi_0, \psi_0, \zeta_0, E_0, b_0) \|^2      + C\delta^{\frac{1}{6}}      + C \delta^{\frac{2}{3}} \int_0^t\int_{\mathbb{R}} \left(\phi^{2} + \zeta^{2} + b^2 \right) \omega^2  \,\mathrm{d}x \mathrm{d}\tau     \nonumber \\[2mm]
 & \quad   + C \delta^{\frac{1}{2}} \int_0^t \| \phi_x \|^2  \,\mathrm{d}\tau      + C \delta^{\frac{1}{2}} \int_0^t \| b_x \|^2  \,\mathrm{d}\tau ,    \label{jibennengliang''}
 \end{align}
 %\end{equation}
where we have taken $ \varepsilon \beta^2 \leq \frac{1}{16} $.
Finally, like \eqref{daiquanguji}, it is not hard to verify that for $ \alpha = {\hat{c}}/{4} $,
\begin{small}
\begin{align}
& \int_0^t\int_{\mathbb{R}}  \left(\phi^{2} + \psi^2 + \zeta^{2} + b^2 \right)  \omega^2  \,\mathrm{d}x \mathrm{d}\tau       \leq   C + C \int_0^t \| \left( \phi_x, \psi_x, \zeta_x, b_x, E + \psi b + \bar{u} b  \right)  \|^2  \,\mathrm{d}\tau       \nonumber \\[2mm]
& \qquad\qquad\qquad\qquad\qquad\quad \;\;    +  C \int_0^t\int_{\mathbb{R}} \left( \phi^2 + \zeta^2 + \varepsilon E^2 + b^2  \right) \left( \left( U^r_- \right)_x  +  \left( U^r_+ \right)_x  \right)    \,\mathrm{d}x \mathrm{d}\tau       \label{daiquanguji'}
\end{align}
\end{small}
holds. Then substituting \eqref{daiquanguji'} into \eqref{jibennengliang''} and choosing $ \delta $ small enough leads to \eqref{jibennengliang'}.

This completes the proof of Lemma \ref{lemma4.1}.
\end{proof}

%%%%%%%%%%%%%%%%%%%%%%%%%%%%%%%%%%%%%%%%%%%%

It remains to prove the following lemma about the energy $\|(\sqrt{\varepsilon}E_x,b_x)\|^2$.

\begin{lemma} \label{exbxyijiedao'}
Suppose that all the conditions in Theorem \ref{thm2} hold. For all $0 < t < T $, there exists a constant $\bar{C}$ depending only on $ v_\pm $, $ u_\pm $ and $ \theta_\pm $ such that if $ 0< \varepsilon <\bar{C} $, then
%\begin{equation}\label{exbx'}
\begin{align} %[b]
& \qquad \|(\sqrt{\varepsilon}E_x,b_x)\|^2+\int_0^t\|(E_x,b_x)\|^2{\rm{d}}\tau      \nonumber \\[2mm]
&  \leq  C \|(\phi_0, \psi_0, \zeta_0)\|^2       + C \|( E_0, b_0)\|_{H^1}^2         +  C\delta^\frac{1}{6}       +  C ( \delta^{\frac{1}{2}}   +  \varepsilon _0 ) \int_0^t \| \phi_x \|^2  \,\mathrm{d}\tau.      \label{exbx'}
\end{align}

%\end{equation}
\end{lemma}
\begin{proof}[Proof]
Firstly, by employing similar computations as from \eqref{Exbx1} to \eqref{Exbx3}, we have
%\begin{equation}\label{Exbx3'}
\begin{align} %[b]
&\qquad \frac{\mathrm{d}}{\mathrm{d}t} \int_{\mathbb{R}} \left( \frac{1}{2}\varepsilon vE_x^2 + \frac{1}{2}v b_x^2 \right)  \,\mathrm{d}x + \int_{\mathbb{R}} vE_x^2 \,\mathrm{d}x       \nonumber \\[2mm]
&= \int_{\mathbb{R}} \varepsilon u_x E_x^2 \,\mathrm{d}x  -\int_{\mathbb{R}} \varepsilon\frac{u}{v}v_x E_x^2 \,\mathrm{d}x + \int_{\mathbb{R}} u_x b_x^2 \,\mathrm{d}x - \int_{\mathbb{R}} \frac{u}{v} v_x b_x^2 \,\mathrm{d}x      \nonumber \\[2mm]
&\quad -2 \int_{\mathbb{R}} \frac{v_x}{v}b_xE_x \,\mathrm{d}x - \int_{\mathbb{R}} (\psi b)_x vE_x\,\mathrm{d}x - \int_{\mathbb{R}} (U b)_x vE_x\,\mathrm{d}x.    \label{Exbx3'}
\end{align}
%\end{equation}
%%%%%%%%%%%%%%%%%%%%%%%%%%%%%%%%%
Recall that $\left| U \right| < \frac{3}{2} \beta  $ and $ \| \left( \phi, \psi \right)  \|_{H^2} \leq \varepsilon_0  $. By using the decay rate of viscous contact wave, the Sobolev inequality \eqref{lwuqiong} and the Cauchy inequality, we can derive from \eqref{Exbx3'} that
%\begin{equation}\label{Exbx4'}
\begin{align} %[b]
& \quad  \frac{\mathrm{d}}{\mathrm{d}t} \int_{\mathbb{R}} \left( \frac{1}{2}\varepsilon vE_x^2 + \frac{1}{2}v b_x^2 \right)  \,\mathrm{d}x   + \int_{\mathbb{R}} vE_x^2 \,\mathrm{d}x     \nonumber \\[2mm]
& \leq  C \int_{\mathbb{R}} (\left| \phi_x \right| + \left| \psi_x \right|)(v E_x^2 + b_x^2) \,\mathrm{d}x         + C \int_{\mathbb{R}} ( \left| V_x \right|  +  \left| U_x \right|  )(v E_x^2 + b_x^2 ) \,\mathrm{d}x       \nonumber \\[2mm]
&  \quad      +   C \int_{\mathbb{R}} (\left| \psi_x b \right|   + \left| \psi b_x \right|   + \left| U_x b \right| ) \left| E_x \right|  \,\mathrm{d}x             +  \int_{\mathbb{R}}   \left| v U b_x E_x  \right|  \,\mathrm{d}x       \nonumber \\[2mm]
& \leq  C\| \left( \phi_x, \psi_x \right)  \|_{L^\infty} \| \left( \sqrt{v} E_x,  b_x  \right) \|^2       + C ( \delta^{\frac{1}{2}}  + \varepsilon_0 ) \| \left( \sqrt{v}  E_x, \psi_x, b_x \right)  \|^2         \nonumber \\[2mm]
&   \quad     + C \delta^{-\frac{1}{2}} \int_{\mathbb{R}} U^2_x b^2 \,\mathrm{d}x          + \frac{1}{2} \int_{\mathbb{R}} v E_x^2 \,\mathrm{d}x         + \frac{1}{2} \int_{\mathbb{R}} v U^2 b_x^2 \,\mathrm{d}x     \nonumber \\[2mm]
& \leq   C ( \delta^{\frac{1}{2}}  + \varepsilon_0 ) \| \left( \sqrt{v}  E_x, \psi_x, b_x \right)  \|^2     +  C \delta^{\frac{1}{2}}  (1+t)^{-2}         + \frac{1}{2} \| \sqrt{v} E_x \|^2          \nonumber \\[2mm]
& \quad     +  \frac{9}{8} \beta^2 \max \{ v_\pm\} \| b_x \|^2 ,     \label{Exbx4'}
\end{align}
%\end{equation}
where in the last inequality we have used
\begin{align*}
\delta^{-\frac{1}{2}} \int_{\mathbb{R}} U^2_x b^2 \,\mathrm{d}x
& \leq    C \delta^{-\frac{1}{2}} \| b \|^2_{L^{\infty}} \| U_x \|^2        \leq    C \delta^{-\frac{1}{2}} \| b \|^2_{L^{\infty}} \left( \| U^c_x \|^2     +   \| \left( U^r_\pm \right)_x  \|^2  \right)         \\[2mm]
& \leq    C \delta^{-\frac{1}{2}} \| b \|\cdot \| b_x \|    \left[  \delta^2 (1+t)^{-\frac{3}{2}}   +   \delta (1+t)^{-1}   \right]         \leq  C \delta^{\frac{1}{2}} \| b_x \|  (1+t)^{-1}      \\[3mm]
& \leq     C \delta^{\frac{1}{2}} \left[ \| b_x \|^2  +  (1+t)^{-2}   \right]  .
\end{align*}
Then by choosing  $ \varepsilon_0 $ and $ \delta $ suitably small, we can reach from \eqref{Exbx4'}
%\begin{equation}\label{lemshi1'}
\begin{align} %[b]
 &\quad \frac{\mathrm{d}}{\mathrm{d}t} \int_{\mathbb{R}} \left( \frac{1}{2}\varepsilon vE_x^2 + \frac{1}{2}v b_x^2 \right)  \,\mathrm{d}x       + \frac{1}{4}\int_{\mathbb{R}} vE_x^2 \,\mathrm{d}x    \nonumber \\[2mm]
 &  \leq  C ( \delta^{\frac{1}{2}}  + \varepsilon_0 )  \| \left( \psi_x, b_x \right)  \|^2        + C \delta^{\frac{1}{2}} (1+t)^{-2}      + \frac{9}{8} \beta^2 \max \left\{ v_\pm \right\} \| b_x \|^2 .   \label{lemshi1'}
\end{align}
%\end{equation}
Secondly, similar to the derivation of \eqref{gujiExbx6}, we have
\begin{small}
\begin{align}
& \quad -\frac{\mathrm{d}}{\mathrm{d}t}\int_{\mathbb{R}} \varepsilon v Eb_x \,\mathrm{d}x + \int_{\mathbb{R}} b_x^2 \,\mathrm{d}x     \nonumber \\[2mm]
& =  -\int_{\mathbb{R}} \varepsilon \psi_x E b_x \,\mathrm{d}x       - \int_{\mathbb{R}} \varepsilon U_x E b_x \,\mathrm{d}x     +  \int_{\mathbb{R}} \varepsilon \frac{u}{v} E \phi_x b_x \,\mathrm{d}x             + \int_{\mathbb{R}} \varepsilon \frac{u}{v} E V_x b_x \,\mathrm{d}x        \nonumber \\[2mm]
& \quad   + \int_{\mathbb{R}} \varepsilon \frac{1}{v} \phi_x E E_x  \,\mathrm{d}x       +  \int_{\mathbb{R}}  \varepsilon \frac{1}{v} V_x E E_x \,\mathrm{d}x     + \int_{\mathbb{R}} \varepsilon E_x^2 \,\mathrm{d}x        + \int_{\mathbb{R}} (E + \psi b + U b )v b_x \,\mathrm{d}x     \nonumber \\[2mm]
& \leq   C \varepsilon \varepsilon_0 \int_{\mathbb{R}} \left( \left| \psi_x b_x  \right|   + \left| \phi_x b_x \right|   + \left| \phi_x E_x \right|  \right)  \,\mathrm{d}x         + C \varepsilon  \int_{\mathbb{R}} \left( \left| U_x E \right| + \left| V_x E \right| \right) \left( \left| b_x \right|  + \left| E_x \right|  \right)   \,\mathrm{d}x    \nonumber \\[2mm]
& \quad   + \varepsilon \int_{\mathbb{R}} E_x^2  \,\mathrm{d}x            + C \int_{\mathbb{R}} \left| (E + \psi b + U b )b_x \right|  \,\mathrm{d}x       \nonumber \\[2mm]
& \leq   C \varepsilon \varepsilon_0  \| (\phi_x, \psi_x, b_x, E_x) \|^2        + C \varepsilon \delta^{-\frac{1}{2}} \int_{\mathbb{R}} \left( \left| U_x E \right|^2  +  \left| V_x E \right|^2   \right)  \,\mathrm{d}x         \nonumber \\[2mm]
& \quad   + C \varepsilon \delta^{\frac{1}{2}}  \int_{\mathbb{R}}  \left( b_x^2  +  E_x^2  \right)  \,\mathrm{d}x          + \varepsilon \| E_x \|^2       +C \| E + \psi b + U b  \|^2        +\frac{1}{2} \| b_x \|^2,        \nonumber \\[2mm]
& \leq   C \varepsilon (\delta^{\frac{1}{2}}  +  \varepsilon_0) \| (\phi_x, \psi_x, b_x, E_x) \|^2         +C \varepsilon \delta^{\frac{3}{2}} \| E \omega \|^2        + C \varepsilon \delta^{\frac{1}{2}} (1+t)^{-2}          \nonumber \\[2mm]
& \quad  + \varepsilon \| E_x \|^2        + C \| E + \psi b + U b  \|^2        +\frac{1}{2} \| b_x \|^2,      \label{gujiExbx6'}
\end{align}
\end{small}
where in the last inequality of \eqref{gujiExbx6'} we have used
%\begin{equation}
\begin{align*} %[b]
& \quad  \varepsilon \delta^{-\frac{1}{2}} \int_{\mathbb{R}} \left( \left| U_x E \right|^2  +  \left| V_x E \right|^2   \right)  \,\mathrm{d}x       \\[2mm]
& \leq C \varepsilon \delta^{-\frac{1}{2}}  \left( \int_{\mathbb{R}}  \left( \Theta^c_x \right)^2  E^2  \,\mathrm{d}x           +  \int_{\mathbb{R}} \left( U^r_\pm  \right)_x^2 E^2  \,\mathrm{d}x  \right)        \\[2mm]
& \leq C \varepsilon \delta^{-\frac{1}{2}}  \left( \int_{\mathbb{R}} \delta^2 \omega^2 E^2  \,\mathrm{d}x          +  \| E \|^2_{L^\infty} \| \left( U^r_\pm \right)_x \|^2   \right)       \\[2mm]
& \leq   C \varepsilon \delta^{\frac{3}{2}} \| E \omega \|^2       + C \varepsilon  \delta^{-\frac{1}{2}} \| E \| \| E_x \| \cdot \delta (1+t)^{-1}       \\[3mm]
& \leq   C \varepsilon \delta^{\frac{3}{2}} \| E \omega \|^2    + C \varepsilon \delta^{\frac{1}{2}} \| E_x \|^2     + C \varepsilon \delta^{\frac{1}{2}} (1+t)^{-2} .
\end{align*}
%\end{equation}
Hence,
%\begin{equation}\label{lemshi2'}
\begin{align} %[b]
 -\frac{\mathrm{d}}{\mathrm{d}t}\int_{\mathbb{R}} \varepsilon v Eb_x \,\mathrm{d}x     + \frac{1}{2} \| b_x \|^2
  & \leq C \varepsilon (\delta^{\frac{1}{2}}  +  \varepsilon_0) \| (\phi_x, \psi_x, b_x, E_x) \|^2       + C \varepsilon \delta^{\frac{3}{2}}  \| E \omega \|^2        \nonumber \\[2mm]
  &  \quad     + C \varepsilon \delta^{\frac{1}{2}} (1+t)^{-2}        + C \| E + \psi b + U b  \|^2      + \varepsilon \| E_x \|^2 .    \label{lemshi2'}
\end{align}
%\end{equation}
Multiplying \eqref{lemshi2'} by $ \frac{5}{2} \beta^2 \max\{v_\pm\} $ and summing it to \eqref{lemshi1'}, together with $ v > \frac{1}{4} \min \left\{ v_\pm \right\} $ gives
\begin{small}
\begin{align}\label{lemshi3'}
 &\quad \frac{\mathrm{d}}{\mathrm{d}t}  \int_{\mathbb{R}}  \left( \frac{1}{2}\varepsilon vE_x^2    + \frac{1}{2}v b_x^2      - \frac{5}{2} \varepsilon \beta^2  \max\{v_\pm\} v Eb_x   \right)  \,\mathrm{d}x        \nonumber \\[2mm]
 &\quad    + \left( \frac{1}{16} \min\{v_\pm\}   -  \frac{5}{2} \varepsilon  \beta^2 \max\{v_\pm\} \right)  \| E_x \|^2       \nonumber \\[2mm]
 &\quad    +  \left( \frac{5}{4} \beta^2 \max\{v_\pm\}  - \frac{9}{8} \beta^2 \max\{v_\pm\} \right)  \| b_x \|^2       \nonumber \\[2mm]
 & \leq      C \delta^{\frac{1}{2}} (1+t)^{-2}        + C ( \delta^{\frac{1}{2}}  + \varepsilon_0 )\| \left( \psi_x, b_x \right)  \|^2        + C \varepsilon \beta^2 (\delta^{\frac{1}{2}} + \varepsilon_0 ) \| (\phi_x, \psi_x, b_x, E_x) \|^2             \nonumber \\[3mm]
 &\quad       + C \varepsilon \beta^2 \delta^{\frac{3}{2}} \| E \omega \|^2          + C \varepsilon \beta^2  \delta^{\frac{1}{2}} (1+t)^{-2}          + C \| E + \psi b + U b  \|^2 .
\end{align}
\end{small}
By taking $ \frac{1}{16} \min\{v_\pm\}   - \frac{5}{2} \varepsilon  \beta^2 \max\{v_\pm\} > \frac{1}{32} \min \left\{ v_\pm \right\} $ in \eqref{lemshi3'}, i.e. $ \varepsilon \beta^2 < \frac{\min\{v_\pm\}}{80 \max\{v_\pm\}} $, we can obtain
\begin{align}
 & \frac{\mathrm{d}}{\mathrm{d}t}  \int_{\mathbb{R}}  \left( \frac{1}{2}\varepsilon vE_x^2  + \frac{1}{2}v b_x^2      - \frac{5}{2} \varepsilon \beta^2  \max\{v_\pm\} v Eb_x   \right)  \,\mathrm{d}x              + \frac{1}{32} \min \left\{ v_\pm \right\} \| E_x \|^2                \nonumber \\[2mm]
 &   +  \frac{1}{8} \beta^2 \max\{v_\pm\} \| b_x \|^2       \leq   C \delta^{\frac{1}{2}} (1+t)^{-2}    + C ( \delta^{\frac{1}{2}} + \varepsilon_0 )  \| (\phi_x, \psi_x, b_x, E_x) \|^2          \nonumber \\[2mm]
 &\qquad\qquad\qquad\qquad\qquad\quad     + C  \delta^{\frac{3}{2}} \| E \omega \|^2      + C \| E + \psi b + U b  \|^2 .     \label{lemshi3''}
\end{align}
Integrating \eqref{lemshi3''} with respect to $t$, then choosing $ \varepsilon_0 $ and $ \delta $ suitably small leads to
\begin{align}
 & \quad \int_{\mathbb{R}} \left( \frac{1}{2}\varepsilon vE_x^2  + \frac{1}{2}v b_x^2   \right)  \,\mathrm{d}x        +  \frac{1}{64} \min\{v_\pm\}  \int_0^t \| E_x \|^2  \,\mathrm{d}\tau           \nonumber \\[2mm]
 &  +  \frac{1}{16} \beta^2 \max\{v_\pm\}  \int_0^t \| b_x \|^2 \,\mathrm{d}\tau    \leq  C \| (E_0, E_{0x}, b_{0x}) \|^2        + C \delta^{\frac{1}{2}}           \nonumber \\[2mm]
 &\qquad\qquad    + C ( \delta^{\frac{1}{2}}  + \varepsilon_0 )  \int_0^t \| (\phi_x, \psi_x) \|^2 \,\mathrm{d}\tau         + C  \delta^{\frac{2}{3}} \int_0^t  \| E \omega \|^2 \,\mathrm{d}\tau           \nonumber \\[2mm]
 &\qquad\qquad     + C \int_0^t \| E + \psi b + \bar{u} b  \|^2   \,\mathrm{d}\tau         + \frac{5}{2} \varepsilon \beta^2  \max\{v_\pm\} \int_{\mathbb{R}} \left| v Eb_x \right|   \,\mathrm{d}x .       \label{lemshi4'}
\end{align}
By using the Cauchy inequality to the last term of \eqref{lemshi4'} and recalling that $ \varepsilon \beta^2 < \frac{\min\{v_\pm\}}{80 \max\{v_\pm\}} $, we can obtain
\begin{equation}\label{danjifen2'}
\frac{5}{2} \varepsilon \beta^2  \max\{v_\pm\} \int_{\mathbb{R}} \left| v Eb_x \right|   \,\mathrm{d}x
\leq \frac{1}{4} \int_{\mathbb{R}} vb_x^2  \,\mathrm{d}x     + C \beta^2 \int_{\mathbb{R}} \varepsilon E^2 \,\mathrm{d}x.
\end{equation}
Plugging \eqref{danjifen2'} into \eqref{lemshi4'}, then employing \eqref{daiquanE}, \eqref{jibennengliang'} and \eqref{daiquanguji'},
together with taking $ \varepsilon_0 $ and $ \delta $ small enough, we can conclude \eqref{exbx'}.

Recalling $ \beta :=  \max \left\{ \left| u_\pm \right|  \right\} > 0 $, from the discussion about $ \varepsilon \beta $ in \eqref{vepbeta1} and $ \varepsilon \beta ^2 $ in this section, we can determine the constant $ \bar{C} $ that
\begin{small}
\begin{align} \label{vepu-2'}
  \varepsilon
  & < \min \left\{ \frac{1}{ 16 \beta^2 }, \;\; \frac{\min\{v_\pm\}}{80 \max\{v_\pm\} \beta^2 }, \;\; \frac{\left( \max \left\{   \sqrt{\gamma R \theta_+ v_+^{-2} },\,  \sqrt{\gamma R \theta_- v_-^{-2} } \right\} \right)^{-1} }{ 32 \max \left\{ v_\pm \right\}  \beta }  \right\}     \nonumber \\[3mm]
  & =   \min \left\{ \frac{\min\{v_\pm\}}{ 80 \max\{v_\pm\}  \cdot  \left( \max \left\{ \left| u_\pm \right|  \right\} \right)^2 } , \;\;        \frac{   \left( \sqrt{\gamma R} \max \left\{   \sqrt{\theta_+} v_+ ^{-1},\, \sqrt{\theta_-} v_- ^{-1}   \right\} \right)^{-1}  }{   32 \max \left\{ v_\pm \right\} \cdot  \max \left\{ \left| u_\pm \right| \right\} }  \right\}       \nonumber \\[3mm]
  & =: \bar{C} .
\end{align}
\end{small}

This completes the proof of Lemma \ref{exbxyijiedao'}.
\end{proof}

Due to Theorem \ref{thm1} and Remark \ref{remark2}, it is not hard to check that the other estimates for a single viscous contact wave still hold for the combination of viscous contact wave with rarefaction waves.
Hence, we finish the proof of Proposition \ref{prop1} for composite wave of the pattern of $ R_1CR_3 $, and so the proof of Theorem \ref{thm2} is completed.

%%%%%%%%%%%%%%%%%%%%%%%%%%%%%%%%%%%%%%%%%%%%%%%%%%%%%%%%%%%%%%%%%%%%%%%%%%%%%%%%%%%%%%%%%%%%%%%%%%

\section*{Acknowledgments}
The research was supported by the National Natural Science Foundation of China $\#$11771150, 11831003, 11926346 and Guangdong Basic and Applied Basic Research Foundation $\#$2020B1515310015. 
The authors are grateful to the anonymous referees for the valuable comments and the helpful suggestions on the manuscript.

\end{document}